\documentclass[11pt,reqno,oneside,a4paper]{amsart}

\usepackage{amsmath, amsthm, amsopn, amssymb}
\usepackage{enumerate}
\usepackage{enumitem}
\usepackage{stmaryrd}
\usepackage{hyperref}
\usepackage{mathtools}
\usepackage{tikz}

\usepackage{geometry}
\newtheorem{thm}{Theorem}[section]
\newtheorem{lemma}[thm]{Lemma}
\newtheorem{prop}[thm]{Proposition}
\newtheorem{conj}[thm]{Conjecture}
\newtheorem{cor}[thm]{Corollary}

\theoremstyle{definition}
\newtheorem{claim}[thm]{Claim}
\newtheorem*{dfn}{Definition}
\newtheorem*{remark}{Remark}

\counterwithout{equation}{section}

\newcommand{\floor}[1]{\left \lfloor{#1}\right \rfloor}

\newcommand{\br}[1]{\llbracket{#1}\rrbracket}

\newcommand{\eps}{\varepsilon}

\renewcommand{\Pr}{\mathbb{P}}
\newcommand{\Ex}{\mathbb{E}}
\newcommand{\dTV}{d_\textrm{TV}}
\newcommand{\Pois}{\mathrm{Pois}}

\newcommand{\cB}{\mathcal{B}}
\newcommand{\cF}{\mathcal{F}}
\newcommand{\cG}{\mathcal{G}}
\newcommand{\cH}{\mathcal{H}}
\newcommand{\cK}{\mathcal{K}}
\newcommand{\cL}{\mathcal{L}}
\newcommand{\cM}{\mathcal{M}}
\newcommand{\cP}{\mathcal{P}}
\newcommand{\cU}{\mathcal{U}}
\newcommand{\cZ}{\mathcal{Z}}

\newcommand{\Free}{\cF_{n,m}(H)}
\newcommand{\Freen}{\cF_{n}(H)}

\newcommand{\Freedg}{\cF_{n,m}(H;\delta,\gamma)}
\newcommand{\Freedp}{\cF_{n,m}(H;\delta,\Pi)}
\newcommand{\FreesdP}{\cF^*_{n,m}(H;\delta,\Pi)}
\newcommand{\Frees}{\cF^*}
\newcommand{\FreeB}{\cF^*_B}
\newcommand{\FreeT}{\cF^*(T)}
\newcommand{\FreeZ}{\cF^*(T;Z)}
\newcommand{\Freedgp}{\cF^+_{n,m}(H;\delta,\gamma)}

\newcommand{\Dh}{D_H}
\newcommand{\Ds}{D_*}

\newcommand{\NN}{\mathbb{N}}

\newcommand{\cKp}{\mathcal{K}'}

\renewcommand{\SS}{\mathcal{S}}
\newcommand{\SSp}{\mathcal{S}'}

\newcommand{\SSpJ}{\SSp(J)}
\newcommand{\PhipT}{\Phi'_T}

\newcommand{\cKIJS}{\cK(I,J,S,S')}
\newcommand{\Kphi}{K_\varphi}
\newcommand{\Kphip}{K_{\varphi'}}
\newcommand{\Sphi}{S_\varphi}
\newcommand{\Sphip}{S_{\varphi'}}
\newcommand{\Hm}{H^-}
\newcommand{\emi}{e_{H^-}}

\newcommand{\TT}{\mathcal{T}}
\newcommand{\TTlow}{\TT_{\mathrm{L}}}
\newcommand{\TThigh}{\TT_{\mathrm{H}}}
\newcommand{\TTL}{\TT_{\mathrm{L}}(\Pi)}
\newcommand{\TTH}{\TT_{\mathrm{H}}(\Pi)}

\newcommand{\TTsubB}{\TT_\Pi(B,t,\ell,h)}
\newcommand{\TTpsubB}{\TT'_\Pi(B,t',b)}

\newcommand{\ZT}{\cZ(T)}
\newcommand{\Zr}{\cZ^{R}(T)}
\newcommand{\Zfr}{\cZ^{R}_1(T)}
\newcommand{\Zsr}{\cZ^{R}_2(T)}
\newcommand{\Zsrz}{\cZ^{R}_2(T_Z)}
\newcommand{\Zi}{\cZ^{I}(T)}
\newcommand{\Zib}{\cZ^{I}_{\Pi}(b)}
\newcommand{\ZibT}{\cZ^{I}_{\Pi}(b;T')}

\newcommand{\tilc}{\tilde{c}}
\newcommand{\tilC}{\tilde{C}}

\newcommand{\crit}{\mathrm{crit}}
\newcommand{\vcrit}{v_c}
\newcommand{\Siz}{S_{i_0}}

\newcommand{\FIJ}{F_{I,J}}
\newcommand{\HIJ}{H_{I,J}}

\newcommand{\Tpd}{\vec{U'}}
\newcommand{\Duv}{\mathcal{D}_{u,v}}
\newcommand{\Suv}{\mathcal{S}_{u,v}}

\newcommand{\SigBR}{\Sigma_B^{R}}
\newcommand{\SigTR}{\Sigma_T^{R}}
\newcommand{\SigBI}{\Sigma_B^{I}}
\newcommand{\SigTpI}{\Sigma_{T'}^{I}}

\newcommand{\dd}{\vec{d}}
\newcommand{\Tz}{\vec{U}}

\newcommand{\twoDensTS}{n^{2-\frac{1}{m_2(H)}}}
\newcommand{\criticalityTS}{n^{2-\frac{1}{\eta(H)}}(\log{n})^{\frac{1}{\zeta(H) -k-1}}}

\newcommand{\Gnm}{\cG_{n,m}}
\newcommand{\Gnmp}{\cG_{n,m'}}
\newcommand{\Grk}{\cG_{n,m}(r,k)}

\newcommand{\Part}{\cP_{n,r}}
\newcommand{\Partg}{\cP_{n,r}(\gamma)}

\newcommand{\Pic}{\Pi^c}
\newcommand{\partition}{\{V_1, \dotsc, V_r\}}

\newcommand{\BPk}{\mathcal{B}(\Pi,k)}
\newcommand{\BPz}{\mathcal{B}(\Pi,0)}
\newcommand{\BPkk}{\mathcal{B}(\Pi,k;\kappa)}
\newcommand{\BPzk}{\mathcal{B}(\Pi,0;\kappa)}
\newcommand{\BPkg}{\mathcal{B}_g(\Pi,k)}
\newcommand{\BPkv}{\mathcal{B}_{v_H}(\Pi,k)}
\newcommand{\BPkpo}{\mathcal{B}(\Pi,k+1)}

\newcommand{\BPpk}{\mathcal{B}(\Pi',k)}

\newcommand{\GPB}{\cG_{m}(\Pi,B)}

\newcommand{\GPpBp}{\cG_{m}(\Pi',B')}

\newcommand{\UPB}{\cU_m(\Pi,B)}
\newcommand{\UPBe}{\cU_m(\Pi,B \cup e)}

\newcommand{\GPBe}{\cG_{m}(\Pi,B \cup e)}

\newcommand{\ex}{\mathrm{ex}}
\newcommand{\exnH}{\mathrm{ex}(n,H)}

\newcommand{\By}[2]{\overset{\mbox{\tiny{#1}}}{#2}}
\newcommand{\ByRef}[2]{   \By{\eqref{#1}}{#2} }
\newcommand{\leBy}[1]{    \By{#1}{\le} }
\newcommand{\geBy}[1]{    \By{#1}{\ge} }
\newcommand{\lByRef}[1]{  \ByRef{#1}{<} }
\newcommand{\leByRef}[1]{ \ByRef{#1}{\le} }
\newcommand{\geByRef}[1]{ \ByRef{#1}{\ge} }

\renewcommand{\le}{\leqslant}
\renewcommand{\ge}{\geqslant}

\title[Typical structure of graphs not containing a vertex-critical subgraph]{On the typical structure of graphs not containing a fixed vertex-critical subgraph}

\author{Oren Engelberg}
\address{School of Mathematical Sciences, Tel Aviv University, Tel Aviv 6997801, Israel}
\email{oengel48@gmail.com}

\author{Wojciech Samotij}
\address{School of Mathematical Sciences, Tel Aviv University, Tel Aviv 6997801, Israel}
\email{samotij@tauex.tau.ac.il}

\author{Lutz Warnke}
\address{Department of Mathematics, University of California San Diego, La Jolla CA~92093, USA}
\email{lwarnke@ucsd.edu}

\thanks{
November 11, 2021; revised February 29, 2024. This research was supported by the Israel Science Foundation grants 1147/14 and 1145/18 (OE and WS), 
and by NSF~grant DMS-1703516, NSF~CAREER grant~DMS-2225631, and a Sloan Research Fellowship (LW)}

\begin{document}

\begin{abstract}
  This work studies the typical structure of sparse $H$-free graphs, that is, graphs that do not contain a subgraph isomorphic to a given graph $H$.  Extending the seminal result of Osthus, Pr\"omel, and Taraz that addressed the case where $H$ is an odd cycle, Balogh, Morris, Samotij, and Warnke proved that, for every $r \ge 3$, the structure of a random $K_{r+1}$-free graph with $n$ vertices and $m$ edges undergoes a phase transition when $m$ crosses an explicit (sharp) threshold function $m_r(n)$. They conjectured that a similar threshold phenomenon occurs when $K_{r+1}$ is replaced by any strictly $2$-balanced, edge-critical graph $H$.  In this paper, we resolve this conjecture.  In fact, we prove that the structure of a typical $H$-free graph undergoes an analogous phase transition for every $H$ in a family of vertex-critical graphs that includes all edge-critical graphs.

\end{abstract}

\maketitle

\setcounter{tocdepth}{1}
\tableofcontents
\setcounter{tocdepth}{2}

\section{Introduction}
\label{sec:introduction}

\subsection{Background and motivation}
\label{sec:backgr-motiv}

Given a graph $H$, let $\Freen$ be the family of all graphs with vertex set $\br{n} = \{1, \dotsc, n\}$ that are \emph{$H$-free}, that is, graphs which do not contain a (not necessarily induced) subgraph isomorphic to $H$. A basic question in extremal graph theory is to determine $\exnH$, the largest number of edges in a graph from $\Freen$.  The classical result of Tur{\'a}n~\cite{Tu41} determines $\mathrm{ex}(n,K_{r+1})$ for every $r \ge 2$ and also characterises the extremal graphs. The works of Erd\H{o}s, Simonovits, and Stone~\cite{ErdSim66, ErSt46} extend this to an arbitrary non-bipartite graph $H$, showing that
\[
  \exnH = \left(1 - \frac{1}{\chi(H)-1} + o(1)\right)\binom{n}{2} ,
\]
and, moreover, that every $H$-free graph with at least $\exnH - o(n^2)$ edges may be made $(\chi(H)-1)$-partite by removing from it some $o(n^2)$ edges.

Here, we are interested in the structure of a \emph{typical} $H$-free graph.  This problem was first considered by Erd\H{o}s, Kleitman, and Rothschild~\cite{ErKlRo76}, who proved that almost all triangle-free graphs are bipartite.  Formally, if $F_n$ is a uniformly chosen random element of $\cF_n(K_3)$, then  
\[
  \lim_{n\to\infty} \Pr(F_n\text{ is bipartite}) = 1.
\]
This result was later generalised by Kolaitis, Pr\"omel, and Rothschild~\cite{KoPrRo87}, who showed that, for every fixed $r \ge 2$, almost all graphs in $\cF_n(K_{r+1})$ are $r$-partite. Very recently, Balogh and the second named author~\cite{BalSam19} proved that this remains true as long as $r \le c \log n / \log \log n$ for some small positive constant $c$, see also~\cite{BalBusColLiuMorSha17, MouNenSte14}.

The result of~\cite{KoPrRo87} was further generalised from cliques to the much wider class of edge-critical graphs.  We say that a graph $H$ is \emph{edge-critical} if it contains an edge whose removal reduces the chromatic number, that is, if $\chi(H \setminus e) = \chi(H)-1$ for some $e \in E(H)$; in particular, every clique is edge-critical and so is every odd cycle.  Simonovits~\cite{Si74} showed that, for every edge-critical graph $H$ and all large enough $n$, not only $\exnH = \ex(n, K_{\chi(H)})$ but also that the only $H$-free graphs with $\exnH$ edges are complete $(\chi(H)-1)$-partite graphs, as in the case $H = K_{\chi(H)}$.  Pr\"omel and Steger~\cite{PrSt92} showed that, if $H$ is edge-critical, then almost every $H$-free graph is $(\chi(H)-1)$-partite.

One drawback of the structural characterisations of typical $H$-free graphs mentioned above is that they do not say anything about sparse graphs, that is, $n$-vertex graphs with $o(n^2)$ edges. Indeed, for every non-bipartite $H$, the family $\Freen$ contains all bipartite graphs and there are at least $2^{\lfloor n^2/4 \rfloor}$ of them; this is much more than the number of \emph{all} graphs with $n$ vertices and at most $n^2/20$ edges. In view of this, it is natural to ask the following refined question, first considered by Pr\"omel and Steger~\cite{PrSt96}: Fix some $m$ with $0 \le m \le \exnH$. What can be said about the structure of a uniformly selected random element of $\Freen$ with exactly $m$ edges? In particular, for what $m$ does this graph admit a similar description as a uniformly random element of $\Freen$?

Let $\Gnm$ be the family of all graphs with vertex set $\br{n}$ and precisely $m$ edges and let $\Free = \Gnm \cap \Freen$ be the subfamily of $\Gnm$ that comprises all $H$-free graphs. Osthus, Pr\"omel, and Taraz~\cite{OsPrTa03} showed that, for every odd integer $\ell \ge 3$, there exists an explicit constant $c_\ell$ such that, letting $m_\ell = m_\ell(n) = c_\ell n^{\frac{\ell}{\ell-1}}(\log n)^{\frac{1}{\ell-1}}$, a uniformly random graph $F_{n,m} \in \cF_{n,m}(C_\ell)$ satisfies, for every $\eps > 0$,
\[
  \lim_{n\to\infty}\Pr(F_{n,m} \text{ is bipartite}) = 
  \begin{cases}
    0 & \text{if } n/2 \le m \le (1-\eps)m_\ell, \\
    1 & \text{if } m \ge (1+\eps) m_\ell.
  \end{cases}
\]
This result was extended by Balogh, Morris, Samotij, and Warnke~\cite{BaMoSaWa16} to the case where $H$ is a clique of an arbitrary order.  They showed that, for every $r\ge 3$, there is an explicit positive constant $c_r'$ such that, letting $m_r' = m_r'(n)= c_r' n^{2-\frac{2}{r+2}}(\log n)^{\frac{1}{\binom{r+1}{2}-1}}$, a uniformly chosen random graph $F_{n,m} \in \cF_{n,m}(K_{r+1})$ satisfies, for every $\eps > 0$,
\[
  \lim_{n\to\infty}\Pr(F_{n,m} \text{ is $r$-partite}) = 
  \begin{cases}
     0 & \text{if } n \ll m \le (1-\eps)m_r', \\
     1 & \text{if } m \ge (1+\eps)m_r'.
  \end{cases}
\]

\subsection{Our result -- edge-critical graphs}
\label{sec:our-results-edge}

Aiming towards a common generalisation of the results of~\cite{BaMoSaWa16,OsPrTa03,PrSt92}, the authors of~\cite{BaMoSaWa16} made the following conjecture. Recall that the \emph{$2$-density} of a graph $H$ is defined by
\[
  m_2(H) = \max\left\{\frac{e_K-1}{v_K-2} : K \subseteq H, \, v_K \ge 3\right\} ,
\]
and that $H$ is called \emph{strictly $2$-balanced} if the maximum above is attained only when $K = H$, that is, if $m_2(K) < m_2(H)$ for every proper subgraph $K \subsetneq H$.

\begin{conj}[{\cite[Conjecture~1.3]{BaMoSaWa16}}]
  For every strictly $2$-balanced, non-bipartite, edge-critical graph~$H$, there exists a constant $C$ such that the following holds. If
  \[
    m \ge C \twoDensTS (\log n)^{\frac{1}{e_H-1}},
  \]
  then almost all graphs in $\Free$ are $(\chi(H)-1)$-partite.
\end{conj}

In this paper, we resolve this conjecture and show that the assumption on $m$ is best possible.

\begin{thm}
  \label{thm:edge-critical}
  For every strictly $2$-balanced, non-bipartite, edge-critical graph~$H$, there exist positive constants $c_H$ and $C_H$ such that, letting
  \[
    m_H = m_H(n) = \twoDensTS (\log n)^{\frac{1}{e_H-1}},
  \]
  the following holds for a uniformly chosen random graph $F_{n,m} \in \Free$:
  \[
    \lim_{n \to \infty} \Pr\big(\text{$F_{n,m}$ is $(\chi(H)-1)$-partite}\big) =
    \begin{cases}
      0 & \text{if } n \ll m \le c_H m_H, \\
      1 & \text{if } m \ge C_Hm_H.
    \end{cases}
  \]
\end{thm}

In fact, Theorem~\ref{thm:edge-critical} is only a special case of a much more general result, Theorem~\ref{thm:full} below, which we present in the next subsection.

\subsection{Our result -- vertex-critical graphs}
\label{sec:our-results-vertex}

We say that a graph $H$ is \emph{vertex-critical} if it contains a vertex whose deletion reduces the chromatic number, that is, if $\chi(H - v) = \chi(H)-1$ for some $v \in V(H)$; we call every such $v$ a \emph{critical vertex} of $H$. A star $S \subseteq H$ centred at a critical vertex is called a \emph{critical star} if $\chi(H\setminus S)=\chi(H)-1$ and if no proper subgraph $S'\subsetneq S$ has this property. For a critical vertex $v$, we define $\crit(v)$, the \emph{criticality of $v$}, to be the smallest number of edges incident to $v$ whose removal decreases the chromatic number, that is,
\[
  \crit(v) = \min\left\{e_S : \text{$S$ is a critical star centred at $v$}\right\},
\]
and define $\crit(H)$, the \emph{criticality of H}, to be the smallest criticality of a vertex, that is,
\[
  \crit(H) = \min\{\crit(v) : \text{$v$ is a critical vertex}\}.
\]
Note that every edge-critical graph is also vertex-critical. Conversely, a vertex-critical graph is edge-critical precisely when its criticality is equal to one.

The motivation for our investigation of the typical structure of graphs not containing a fixed vertex-critical graph is a result of Hundack, Pr\"omel, and Steger~\cite{HuPrSt93} which states that, for every vertex-critical $H$, almost all $H$-free graphs are `almost' $(\chi(H)-1)$-partite in the following precise sense. Given integers $r \ge 1$ and $k \ge 0$, we will denote by $\cG(r,k)$ the class of all graphs $G$ that admit an $r$-colouring of $V(G)$ for which the subgraph of $G$ induced by each of the $r$ colour classes has maximum degree at most $k$. In particular, $\cG(r,0)$ is the class of all $r$-colourable graphs and the following theorem generalises the main result of~\cite{PrSt92}. 

\begin{thm}[\cite{HuPrSt93}]
  \label{thm:HuPrSt}
  If $H$ is a vertex-critical graph of criticality $k+1$ and $\chi(H) = r+1 \ge 3$, then, for some positive $c$,
  \[
    |\cF_n(H)| = \left(1+O(2^{-cn})\right)\cdot|\cF_n(H)\cap\cG(r,k)|.
  \]
\end{thm}

\begin{remark}
  A less accurate description of the structure of a typical member of $\cF_n(H)$, but valid for \emph{every} non-bipartite $H$, was given by Balogh, Bollob\'as, and Simonovits~\cite{BaBoSi09}.
\end{remark}

Our main result is a sparse analogue of Theorem~\ref{thm:HuPrSt} that is valid for a subclass of vertex-critical graphs that includes all edge-critical graphs. In order to state it, we need several additional definitions.

\begin{dfn}
  A vertex-critical graph $H$ will be called \emph{simple vertex-critical} if every colouring of~$H$ with $\chi(H)-1$ colours admits a monochromatic star with $\crit(H)$ edges or a~monochromatic cycle. Further, a vertex-critical graph will be called \emph{plain vertex-critical} if, for every colouring of $H$ with $\chi(H)-1$ colours, the monochromatic graph $B$ satisfies at least one of the following:
  \begin{enumerate}[label={(\roman*)}]
  \item
    $B$ contains a cycle,
  \item
    $B$ is the star $K_{1, \crit(H)}$,
  \item
    $B$ has a vertex with degree larger than $\crit(H)$, or
  \item
    $B$ has two nonadjacent vertices with degree $\crit(H)$.
  \end{enumerate}
\end{dfn}

It is not hard to see that every edge-critical graph is plain vertex-critical and every plain vertex-critical graph is simple vertex-critical.

\begin{remark}
  Another family of plain vertex-critical graphs are the complete multipartite graphs $K_{1, k_1, \dotsc, k_r}$ with $1 \le k_1 < k_2 \le \dotsb \le k_r$. To see this, denote the $r+1$ colour classes of this graph by $V_0, \dotsc, V_r$, so that $|V_0| = 1$ and $|V_i| = k_i$ for each $i \in \br{r}$. Consider an arbitrary $r$-colouring $W_1 \cup \dotsb \cup W_r$ of the vertices of $K_{1,k_1, \dotsc, k_r}$. If, for some $j \in \br{r}$, we have $|W_j \setminus V_i| \ge 2$ for every $i$, then $W_j$ must contain a cycle; indeed, in this case $W_j$ intersects three different $V_i$ or it intersects some two $V_i$ in at least two vertices each. We may therefore assume that, for every $j \in \br{r}$, there is an $i(j)$ such that $\delta_j = |W_j \setminus V_{i(j)}| \le 1$.  This assumption guarantees that each $W_j$ induces a star (if $\delta_j = 1$) or an empty graph (if $\delta_j = 0$). Let $J = \{ j \in \br{r} : \delta_j = 1\}$ and observe that
  \[
    \sum_{j \in J} |W_j| = 1 + k_1 + \dotsb + k_r - \sum_{j \notin J} |W_j| \ge 1+k_1+\dotsb+k_{|J|} \ge |J| \cdot (k_1+1).
  \]
  In particular, either each $W_j$ with $j \in J$ induces a copy of $K_{1,k_1}$ or one of them induces a graph with maximum degree strictly larger than $k_1$.
\end{remark}

For an integer $k \ge 2$ and a graph $F$ with $v_F \ge k+1$, we let
\[
  d_k(F) = \frac{e_F-k+1}{v_F-k},
\]
cf.~the definition of $2$-density given at the start of Section~\ref{sec:our-results-edge}.  
Suppose that $H$ is a non-bipartite, vertex-critical graph.  
Let $k \ge 0$ and $r \ge 2$ be the integers such that $\crit(H) = k +1$ and $\chi(H)=r+1$. 
Let $S_1,\dotsc, S_t$ be all the smallest critical stars of $H$, i.e., all its critical stars with~$k+1$ edges. 
Denote, for each $i \in \br{t}$,
\[
  \eta_i(H) = \max \big\{ d_{k+2}(F) : S_i \subsetneq F\subseteq H \big\}
  \qquad \text{and} \qquad 
  \eta(H) = \min_{1\le i \le t} \eta_i (H) ,
\]
and, further,
\[
  \zeta_i(H) = \min \big\{ e_F : S_i \subsetneq F\subseteq H, \; d_{k+2}(F) = \eta_i(H) \big\}
  \qquad \text{and} \qquad
  \zeta(H) = \max_{\substack{1 \le i \le t \\ \mathclap{\eta_i(H) = \eta(H)}}} \zeta_i(H).
\]
We are now ready to define the threshold function:
\begin{equation}
  \label{eq:mH-def}
  m_H = m_H(n) = 
  \begin{cases}
    \twoDensTS & \text{if } m_2(H) > \eta(H), \\
    \criticalityTS & \text{otherwise}.
  \end{cases}
\end{equation}

\begin{remark}
  If $H$ is edge-critical, then the smallest critical stars $S_1, \dotsc, S_t$ are the critical edges of $H$, that is, edges $S$ satisfying $\chi(H \setminus S) = \chi(H)-1$. If, additionally, $H$ is strictly $2$-balanced, then, for every $i \in \br{t}$, the maximum in the definition of $\eta_i(H)$ is uniquely attained at $F = H$ and thus $\eta_i(H) = d_2(H) = m_2(H)$ and $\zeta_i(H) = e_H$; consequently, $\eta(H) = m_2(H)$ and $\zeta(H) = e_H$. This shows that definition~\eqref{eq:mH-def} extends the definition of $m_H$ given in the statement of Theorem~\ref{thm:edge-critical}.
\end{remark}

The following generalisation of Theorem~\ref{thm:edge-critical} is the main result of this work.

\begin{thm}
  \label{thm:full}
  Let $H$ be a simple vertex-critical graph with $\chi(H) = r+1 \ge 3$ and criticality $k+1$,	and let $F_{n,m}$ denote a uniformly chosen random graph of~$\Free$. There exists a positive constant $C_H$ such that, for every $m \ge C_Hm_H$,
  \[
    \lim_{n \to \infty} \Pr\big(F_{n,m} \in \cG(r,k)\big) = 1.
  \]
  Furthermore, if $H$ is plain vertex-critical, then there exists a positive constant $c_H$ such that, for every $n \ll m \le c_Hm_H$,
  \[
    \lim_{n \to \infty} \Pr\big(F_{n,m} \in \cG(r,k)\big) = 0.
  \]
\end{thm}

It may be worth pointing out that there are plain vertex-critical graphs $H$ for which $\eta(H)$ is strictly larger than $m_2(H)$. (One such graph is $H = K_{1,2,3}$, which is plain vertex-critical with criticality two, has exactly one critical star, and satisfies $\eta(H) = 3 > 5/2 = m_2(H)$.)  Why is it interesting? Let $H$ be such a graph and let $F_{n,m}$ be a uniformly chosen random element of $\Free$. As soon as $m \gg \twoDensTS$, with probability close to one, $F_{n,m}$ can be made $(\chi(H)-1)$-partite by removing from it $o(m)$ edges, see Theorem~\ref{thm:approx-struct} below.  However, it is only when $m \gg \criticalityTS$, polynomially above the $2$-density threshold, that the `exact' structure emerges.  What can one say about the typical structure of $F_{n,m}$ between these two thresholds? 

Unfortunately, our techniques are too weak to extend Theorem~\ref{thm:full} to arbitrary vertex-critical graphs.  Still, they are sufficient to prove an approximate version of the $1$-statement.  We refrain ourselves from stating this result here; instead, we refer the interested reader to~\cite{Eng}.  Having said that, we have no good reason to believe that $m_H$ is the threshold for a general vertex-critical graph $H$.  We believe that it would be extremely interesting to find the threshold for a generic vertex-critical graph $H$ and extend Theorem~\ref{thm:full} to all such $H$.

\subsection{Acknowledgements}
We thank J\'ozsef Balogh, Michael Krivelevich, Rob Morris, and Angelika Steger for stimulating discussions about various aspects of this work, and Anita Liebenau for bringing~\cite[Theorem~2]{Bol80} to our attention. 
We are also grateful to the referees for helpful suggestions concerning the presentation.

\section{Outline of the proof}

\subsection{Why is $m_H$ the threshold?}

The location of the threshold at which a typical graph in $\Free$ `enters' $\cG(r,k)$ can be guessed by comparing the number of graphs in $\Free \cap \cG(r,k)$ to the number of graphs in $\Free$ that are `one edge away' from $\cG(r,k)$.  It is relatively straightforward to estimate the former:  For a constant proportion of graphs $G \in \Gnm \cap \cG(r,k)$, the monochromatic subgraph of $G$ (under its optimal colouring, i.e., the colouring that witnesses $G \in \cG(r,k)$) has girth larger than the number of vertices of $H$ (this is true for all $m$); moreover, the assumption that $H$ is \emph{simple} vertex-critical guarantees that each such graph is $H$-free.  As for the latter quantity, the number of graphs in $\Gnm$ that are `one edge away' from $\cG(r,k)$ is $\Theta(m)$ times larger, but the proportion of them that are $H$-free is a decreasing function of $m$.  The reason for this is that every such graph $G$ has at least one copy of $K_{1,k+1}$ in its monochromatic subgraph (under every $r$-colouring) and, since $\chi(H) = r+1$ and $\crit(H) = k+1$, this copy of $K_{1,k+1}$ can be extended to many copies of $H$ in $K_n$ that use only edges that are properly coloured.  In particular, if $G$ is $H$-free, then it must avoid all such copies of $H \setminus K_{1,k+1}$.  Furthermore, if $G$ is \emph{plain} vertex-critical, then this implication can be reversed---$G$ is $H$-free \emph{if and only if} it avoids all such copies of $H \setminus K_{1,k+1}$---under a weak assumption on the monochromatic graph (girth larger than $v_H$) that is satisfied by a constant proportion of all such graphs.

Finally, if we fix both the optimal $r$-colouring and the monochromatic graph (which is `one edge away' from having maximum degree $k$), the proportion $P_m$ of graphs in $\Gnm$ (among those that contain our fixed monochromatic graph) that avoid all copies of $H \setminus K_{1,k+1}$ of the above type can be bounded using the inequalities of Janson (from above) and Harris (from below) as follows:
\[
  - \log P_m = \Theta\left(\max_{i \in \br{t}} \; \min\left\{n^{v_F-k-2}  \cdot (m/n^2)^{e_F-k-1} : S_i \subsetneq F \subseteq H \right\}\right),
\]
where $S_1, \dotsc, S_t$ are the (smallest) critical stars of $H$.

The threshold $m_H$ is then the smallest $m \ge \twoDensTS$ for which $-\log P_m \ge \log m$, that is, for which the number of $H$-free graphs that are `one edge away' from $\cG(r,k)$ is of the same order of magnitude as $|\Free \cap \cG(r,k)|$.  We note that the additional requirement $m \ge \twoDensTS$ is needed because $\twoDensTS$ is the threshold for approximate $r$-colourability of a random element of $\Free$ and comparing $|\Free \cap \cG(r,k)|$ only to the number of graphs in $\Free$ that are `one edge away' from $\cG(r,k)$---as opposed to $|\Free|$---cannot be justified below this threshold.

\subsection{The 0-statement}

Our proof of the $0$-statement (the second assertion of Theorem~\ref{thm:full}), presented in Section~\ref{sec:zeroState}, is a formalisation of the above discussion.  (Having said that, in the range $n \ll m \ll \twoDensTS$, we give a separate, elementary counting argument that exploits the fact that a typical graph in $\Gnm$ can be made $H$-free by removing from it some $o(m)$ edges, see Section~\ref{sec:below-2-density}.)  One of the key ideas here is to reduce the problem of comparing $|\Free \cap \cG(r,k)|$ and the number of graphs in $\Free$ that are `one edge away' from $\cG(r,k)$ to the analogous problem for a \emph{fixed} $r$-colouring that is moreover balanced (in the sense that each of its $r$ colour classes has approximately $n/r$ vertices).  Various assertions and estimates that justify this reduction are proved in Section~\ref{sec:almostRcolour}, where we also show that the number of graphs in $\Gnm$ that are `one edge away' from $\cG(r,k)$ and whose monochromatic graph has girth larger than $v_H$ is indeed $\Theta(m)$ times bigger than $|\Gnm \cap \cG(r,k)|$.

\subsection{The 1-statement}

The proof of the $1$-statement (the first assertion of Theorem~\ref{thm:full}) is significantly harder.  A first, nowadays standard, step is to show that, when $m \gg \twoDensTS$, almost every graph in $\Free$ admits an $r$-colouring such that:
\begin{enumerate}[label=(\roman*)]
\item
  \label{item:monoch-edges}
  there are only $o(m)$ monochromatic edges,
\item
  \label{item:balanced-colouring}
  each colour class comprises approximately $n/r$ vertices,
\item
  \label{item:unfriendly}
  every vertex has at most as many neighbours in its own colour class as it has in any other colour class.
\end{enumerate}
We derive this approximate version of the $1$-statement, stated as Theorem~\ref{thm:aa-balanced-unfriendly} below, from~\cite[Theorem~1.7]{BaMoSa15}.  Using several properties of graphs in $\Gnm \cap \cG(r,k)$ established in Section~\ref{sec:almostRcolour}, we further reduce the $1$-statement to showing that, for every \emph{fixed} $r$-colouring $\Pi$\footnote{We identify every $r$-colouring with a partition of $\br{n}$ into $r$ sets as well as the complete $r$-partite graph with these partite sets.} satisfying~\ref{item:balanced-colouring} above, the number of graphs $G \in \Free$ that satisfy~\ref{item:monoch-edges} and~\ref{item:unfriendly} for this particular $\Pi$ but the maximum degree of $G \setminus \Pi$ (the monochromatic subgraph of $G$) exceeds $k$ is much smaller than the number of graphs $G \in \Gnm$ such that the maximum degree of $G \setminus \Pi$ is at most $k$.  This reduction is formalised in Section~\ref{sec:sufficient-condition}.

Fix an $r$-colouring $\Pi$ satisfying~\ref{item:balanced-colouring} and let $\Frees$ denote the family of all graphs $G \in \Free$ that satisfy~\ref{item:monoch-edges} and~\ref{item:unfriendly} and $\Delta(G \setminus \Pi) > k$.  The methods of bounding the number of graphs in $\Frees$ will vary with $m$ and the distribution of the edges of $G \setminus \Pi$.  In Section~\ref{sec:dense}, we give a somewhat ad-hoc argument to separately treat the case where $m > \ex(n,H) - \xi n^2$  for some small positive $\xi$ (the \emph{dense case}); we will not discuss it in detail here.  We deal with the main, complementary case $m \le \ex(n,H) - \xi n^2$, which we term the \emph{sparse case} in Section~\ref{sec:sparse}.

Let $\TT$ denote the collection of all possible monochromatic graphs $G \setminus \Pi$ as $G$ ranges over $\Frees$.  For every $T \in \TT$, we arbitrarily choose a maximal subgraph $B_T \subseteq T$ with $\Delta(B_T) = k$.  Since we will separately estimate the number of graphs $G \in \Frees$ such that $B_{G \setminus \Pi} = B$ for every $B$ with $\Delta(B) = k$, we may further fix one such $B$.  The possible monochromatic graphs $T \in \TT$ with $B_T = B$ are divided into two classes, denoted $\TTlow$ and $\TThigh$, depending on what proportion of edges of $T$ are incident to vertices whose degrees are larger than $\rho m / n$, where $\rho$ is a small positive constant, see Section~\ref{sec:low-high-degree}.  We separately enumerate graphs $G \in \Frees$ such that $G \setminus \Pi \in \TTlow$ and those satisfying $G \setminus \Pi \in \TThigh$.  We term these two parts of the argument the \emph{low-degree case} and the \emph{high-degree case}, respectively.  We outline these two cases in Sections~\ref{sec:low-degree-summary} and~\ref{sec:high}, respectively.

In the low-degree case, for each $T \in \TTlow$, we give an upper bound on the number of graphs $G \in \Frees$ such that $G \setminus \Pi = T$ and then sum this bound over all $T$.  This upper bound, stated in Proposition~\ref{prop:oneStateLDProb}, lies at the heart the low-degree case.  We briefly describe the main idea:  By construction, every edge of $T \setminus B$ belongs to a copy of $K_{1,k+1}$ in $T$ and, since $\chi(H) = r+1$ and $\crit(H) = k+1$, this copy of $K_{1,k+1}$ can be extended to $\Omega(n^{v_H-k-2})$ copies of $H$ in $K_n$ that use only the edges of $\Pi$.  Consequently, each $G \in \Frees$ with $G \setminus \Pi = T$ must avoid all such copies of $H \setminus K_{1, k+1}$, for every copy of $K_{1,k+1}$ in $T$.  The number of graphs $G$ with this property is bounded from above with the use of the Hypergeometric Janson Inequality (Lemma~\ref{lem:HJI}) applied to a carefully chosen subfamily of copies of $K_{1,k+1}$ in $T$ with `nice' properties that enable us to control the correlation term $\Delta$ in Lemma~\ref{lem:HJI}.  Finally, the size of the summation over all $T$ is controlled by Lemma~\ref{lem:enumeratingT}, which in turn utilises bounds on the number of graphs in $\TT$ with given values of certain key parameters that are obtained in Section~\ref{subsec:enumGraphs}.

In the high-degree case, we crucially use property~\ref{item:unfriendly} to argue that, for every $T \in \TThigh$, all high-degree vertices (vertices whose degree is larger than $\rho m /n$) in every $G \in \Frees$ with $G \setminus \Pi = T$ must have at least $\rho m / n$ neighbours in each colour class of $\Pi$.  This allows us to enumerate all such $G$ in two steps as follows:  First, we specify a graph $Z$ that includes $T$ and the edges of $G \cap \Pi$ incident to a carefully chosen set $Y$ of high-degree vertices.  Second, we specify the remaining $m - e(Z)$ edges of $G \cap \Pi$.  Since $H$ is vertex-critical, each vertex of $Y$ belongs to many copies of $H$ in $Z \cup \Pi$.  This allows us to bound the number of ways to choose the $m-e(Z)$ edges of $G \cap \Pi$ in the second step from above with another application of the Hypergeometric Janson Inequality.  For the vast majority of $Z$, this upper bound will be sufficiently strong to enable a naive union bound over the choice of $Z$; we shall say that such $Z$ fall into the \emph{regular case}.  Unfortunately, there will be a small family of exceptional graphs $Z$ for which this upper bound is too weak (this can happen when there are unusual overlaps between the neighbourhoods of the vertices in $Y$); we shall term it the \emph{irregular case}.  However, we may use Lemma~\ref{lemma:d-sets} to show that the number of such exceptional graphs is so small that even a trivial upper bound of $\binom{e(\Pi)}{m-e(Z)}$ for the number of choices in the second step will be sufficient to show that the number of graphs that fall into the irregular case is tiny.

\section{Preliminaries}
\label{sec:pre}

\subsection{Probabilistic inequalities}

In this section, we present four probabilistic inequalities that will be used in the proof of Theorem~\ref{thm:full}. The first three results presented in this section were proved in~\cite[Lemmas~3.1,~3.2,~3.6]{BaMoSaWa16} and the fourth result is a standard bound on the tail probabilities of hypergeometric distributions~\cite[Theorem~2.10]{JaLuRu00}. We begin with a version of Janson's inequality~\cite{Jan90} for the hypergeometric distribution, which is an essential ingredient in the proof of the $1$-statement in Theorem~\ref{thm:full}.

\begin{lemma}[Hypergeometric Janson Inequality]
  \label{lem:HJI}
  Suppose that $\{B_i\}_{i \in I}$ is a family of subsets of an $n$-element set $\Omega$, let $m \in \{0, \dotsc, n\}$, and let $p = m/n$. Let
  \[
  \mu = \sum_{i \in I} p^{|B_i|} \qquad \text{and} \qquad \Delta = \sum_{i \sim j} p^{|B_i \cup B_j|},
  \]
  where the second sum is over all ordered pairs $(i,j) \in I^2$ such that $i \neq j$ and $B_i \cap B_j \neq \emptyset$. Let $R$ be the uniformly chosen random $m$-element subset of $\Omega$ and let $\cB$ denote the event that $B_i \nsubseteq R$ for all $i \in I$. Then, for every $q \in [0,1]$,
  \[
  \Pr(\cB) \le 2 \cdot \exp\left(-q\mu + q^2 \Delta / 2\right).
  \]
\end{lemma}

The main tool in the proof of the $0$-statement in Theorem~\ref{thm:full} will be the following version of the Harris Inequality~\cite{Har60} for the hypergeometric distribution; it gives a lower bound on the probability $\Pr(\cB)$ from the statement of Lemma~\ref{lem:HJI}.

\begin{lemma}[Hypergeometric Harris Inequality]
  \label{lem:HFKG}
  Suppose that $\{B_i\}_{i \in I}$ is a family of subsets of an $n$-element set $\Omega$. Let $m \in \{0, \dotsc, \lfloor n/2 \rfloor\}$, let $R$ be the uniformly chosen random $m$-element subset of $\Omega$, and let $\cB$ denote the event that $B_i \nsubseteq R$ for all $i \in I$. Then, for every $\eta \in (0,1)$,
  \[
  \Pr(\cB) \ge \prod_{i \in I} \left(1 - \left(\frac{(1+\eta)m}{n}\right)^{|B_i|}\right) - \exp\big( -\eta^2m / 4 \big).
  \]
\end{lemma}

Another key tool in the proof of the $1$-statement in Theorem~\ref{thm:full} is an estimate of the upper tail of the distribution of the number of edges in a random induced subhypergraph of a sparse $z$-uniform $z$-partite hypergraph. It formalises the following statement: If $\cM \subseteq U_1 \times \dotsb \times U_z$ contains only a tiny proportion of all the $z$-tuples in $U_1 \times \dotsb \times U_z$, then the probability that, for a random choice of $d$-elements sets $W_1 \subseteq U_1, \dotsc, W_z \subseteq U_z$, a much larger proportion of $W_1 \times \dotsb \times W_z$ falls in $\cM$ decays exponentially in $d$.

\begin{lemma}
  \label{lemma:d-sets}
  For every integer $z$ and all positive $\alpha$ and $\lambda$, there exists a positive $\tau$ such that the following holds. Let $U_1,\dots,U_z$ be finite sets and let $d$ be an integer satisfying $2\le d \le \min\{|U_1|,\dots, |U_z|\}$. Suppose that $\cM\subseteq U_1\times\dots \times U_z$ satisfies
\[
|\cM| \le \tau\prod_{i=1}^{z}|U_i| ,
\]
and that $W_1,\dots,W_z$ are $d$-element subsets of $U_1,\dots,U_z$, respectively. Then, there are at most $\alpha^d \cdot \prod_{i=1}^{z}\binom{|U_i|}{d}$ choices of $(W_i)_{i\in\br{z}}$~for which
\[
|\cM\cap(W_1\times\dots\times W_z)| > \lambda d^z.
\]
\end{lemma}

Finally, we will need the following simple bound on lower tails of hypergeometric distributions.

\begin{lemma}
  \label{lemma:hyper-lower-tail}
   Let $R$ be the uniformly chosen random $m$-element subset of an $N$-element set $\Omega$ and let $A \subseteq \Omega$ be a $k$-element set. Then, for every $t \ge 0$,
   \[
     \Pr\left(|R \cap A| \le \frac{km}{N} - t\right) \le \exp\left(-\frac{t^2}{2 \cdot km/N}\right).
   \]
\end{lemma}

\subsection{Two-density related bounds}

In this short section, we present a useful inequality that will be invoked several times in the proof of the $1$-statement of Theorem~\ref{thm:full}.

\begin{lemma}
  \label{lemma:m2H}
  Suppose that $H$ is a graph with at least three vertices. If $p \ge C\twoDensTS$ for some $C \ge 0$, then, for every nonempty $F \subseteq H$,
  \[
    n^{v_F} p^{e_F} \ge C^{e_F-1}n^2p.
  \]
\end{lemma}
\begin{proof}
  Let $F$ be a nonempty subgraph of $H$. If $F$ has two vertices, then $F = K_2$ and the assertion is trivially true. Suppose now that $v_F \ge 3$ and observe that
  \[
    n^{v_F-2}p^{e_F-1} \ge n^{v_F-2}\left(C\twoDensTS\right)^{e_F-1} = C^{e_F-1} n^{v_F-2 - (e_F-1)/m_2(H)}.
  \]
  The claimed bound follows as $m_2(H) \ge \frac{e_F-1}{v_F-2}$, by the definition of $2$-density.
\end{proof}

\subsection{The Tur\'an problem for $r$-partite graphs}
\label{sec:turan-problem}

The case $m \ge \exnH - o(n^2)$ in the proof of Theorem~\ref{thm:full} will require the following folklore result in extremal graph theory. For integers $r \ge 2$ and $n \ge 1$, we denote by $K_r(n)$ the balanced complete $r$-partite graph with $r \cdot n$ vertices.

\begin{lemma}
  \label{lemma:Turan}
  For all integers $r$, $s$, and $n$ satisfying $r \ge 2$ and $n \ge s \ge 1$,
  \[
    \ex\big(K_r(n), K_r(s)\big) \le e\big(K_r(n)\big) - n^2/s^2.
  \]
\end{lemma}
\begin{proof}
  Denote the $r$ colour classes of $K_r(n)$ by $V_1, \dotsc, V_r$ and, for each $i \in \br{r}$, let~$R_i$ be a uniformly chosen random $s$-element subset of $V_i$. Suppose that $G \subseteq K_r(n)$ is $K_r(s)$-free and let $G'$ be the subgraph of $G$ induced by $R_1 \cup \dotsb \cup R_r$. Since $G'$ may be viewed as a subgraph of $K_r(s)$, we have $e(G') \le e(K_r(s)) - 1$. On the other hand, $\Ex[e(G')] = e(G) \cdot (s/n)^2$. We conclude that
  \[
    e(G) \le (n/s)^2 \cdot \left(e\big(K_r(s)\big)-1\right) = e\big(K_r(n)\big) - n^2/s^2,
  \]
  as claimed.
\end{proof}

\subsection{Estimates for binomial coefficients}

We will use the following trivial inequalities that hold for all positive integers $a>b>c$:
\begin{align}
  \binom{a}{b-c} &\le \binom{a}{b} \cdot \left(\frac{b}{a-b}\right)^c,\label{eq:binCoef1} \\
  \binom{b}{c}\binom{a}{c}^{-1} &\le \left(\frac{b}{a}\right)^c,\label{eq:binCoef2} \\
  \binom{a}{c}\binom{b}{c}^{-1} &\le \left(\frac{a-c}{b-c}\right)^c,\label{eq:binCoef3} \\
  \sum_{i = 0}^b \binom{a}{i} & \le \left(\frac{ea}{b}\right)^b. \label{eq:binCoef4}
\end{align}

\section{On almost $r$-colourable graphs}\label{sec:almostRcolour}

In this section, we establish several properties of almost $r$-colourable graphs, that is, graphs belonging to the family $\cG(r,k)$, defined in Section~\ref{sec:our-results-vertex}; these properties will come in handy in our proof of Theorem~\ref{thm:full}. It will be convenient to denote by $\Grk = \Gnm \cap \cG(r,k)$  the family of graphs with vertex set $\br{n}$ and precisely $m$ edges that admit an $r$-colouring whose induced monochromatic graph has maximum degree at most~$k$.

Let $\Part$ be the family of all $r$-colourings of $\br{n}$, that is, all partitions of $\br{n}$ into $r$ parts. For the sake of brevity, we shall often identify a partition $\Pi \in \Part$ with the complete $r$-partite graph with vertex set $\br{n}$ whose colour classes are the $r$ parts of $\Pi$. In particular, if $G$ is a graph on the vertex set $\br{n}$, then $G \subseteq \Pi$ means that $G$ is a subgraph of the complete $r$-partite graph $\Pi$ or, in other words, that the partition $\Pi$ is a proper colouring of $G$. Exploiting this convention, we will also write $\Pic$ to denote the complement of the graph $\Pi$, that is, the union of $r$ complete graphs with vertex sets $V_1, \dotsc, V_r$. 

\subsection{Balanced $r$-colourings}

We will be interested in balanced $r$-colourings, that is, partitions of $\br{n}$ whose all parts have approximately $n/r$ elements. More precisely, given a positive $\gamma$, we let $\Part(\gamma)$ be the family of all partitions of $\br{n}$ into $r$ parts $V_1, \dotsc, V_r$ such that 
\begin{equation}
  \label{eq:Part-gamma}
  \left(\frac{1}{r} - \gamma\right)n \le |V_i| \le \left(\frac{1}{r} + \gamma\right)n \quad \text{for all $i \in \br{r}$}.
\end{equation}
That is,
\[
\Part(\gamma) = \big\{\{V_1, \dotsc, V_r\} \in \Part : \text{\eqref{eq:Part-gamma} holds}\big\}.
\]
The following easy proposition establishes useful bounds for the number of edges in the complete $r$-partite graphs defined by balanced and unbalanced $r$-colourings.

\begin{prop}
  \label{prop:bal-unbal-part-size}
  The following holds for every integer $r \ge 2$, every $\gamma > 0$, and all sufficiently large $n$:
  \begin{enumerate}[label=(\roman*)]
  \item
    \label{item:Partg-lower}
    If $\Pi \in \Partg$, then
    \[
      e(\Pi) \ge (1-2r\gamma) \cdot \left(1 - \frac{1}{r}\right)\frac{n^2}{2}.
    \]
    In particular, if $\gamma \le \frac{1}{20r}$, then $e(\Pi) \ge n^2/5$.
  \item
    \label{item:Part-Partg-upper}
    If $\Pi \in \Part \setminus \Partg$, then
    \[
      e(\Pi) \le \left(1 - \frac{\gamma^2}{3}\right) \cdot \ex(n, K_{r+1}).
    \]
  \end{enumerate}
\end{prop}
\begin{proof}
  Note that every $\Pi = \{V_1, \dotsc, V_r\} \in \Part(\gamma)$ satisfies
  \[
    e(\Pi) = \sum_{1 \le i < j \le r} |V_i||V_j| \ge \binom{r}{2} \cdot \left[\left(\frac{1}{r}-\gamma\right)n\right]^2 = (1-r\gamma)^2 \cdot \left(1 - \frac{1}{r}\right)\frac{n^2}{2},
  \]
  proving~\ref{item:Partg-lower}. To see that~\ref{item:Part-Partg-upper} holds as well, fix an arbitrary partition $\Pi$ that does not satisfy~\eqref{eq:Part-gamma} and let $V$ and $W$ be two parts of $\Pi$ with the smallest and the largest size, respectively. Let
  \[
    d = \left\lfloor \frac{|W| - |V|}{2} \right\rfloor,
  \]
  let $\Pi'$ be a partition obtained from $\Pi$ by moving some $d$ vertices from $W$ to $V$, and note that
  \[
    e(\Pi') - e(\Pi) = (|W| - d)(|V|+d) = |V||W| + (|W| - |V|)d - d^2 \ge d^2.
  \]
  Since $\Pi$ does not satisfy~\eqref{eq:Part-gamma}, it must be that $d \ge \lfloor \gamma n / 2\rfloor$ and, since $\ex(n, K_{r+1})$ is the largest number of edges in an $r$-partite graph with $n$ vertices,
  \begin{equation}
    \label{eq:ePi-gamma-upper}
    e(\Pi) \le e(\Pi') - \frac{\gamma^2n^2}{5} \le \ex(n, K_{r+1}) - \frac{\gamma^2n^2}{5} \le \left(1 - \frac{\gamma^2}{3}\right) \ex(n, K_{r+1}),
  \end{equation}
  provided that $n$ is sufficiently large.
\end{proof}

\subsection{Monochromatic graphs with small maximum degree}
\label{sec:mono-graphs-small-max-deg}

For every $\Pi\in\Part$, define $\BPk$ to be the family of all subgraphs of $\Pic$ with maximum degree at most $k$. Now, for every $\Pi \in \Part$ and $B\in \BPk$, define
\[
\GPB = \{G \in \Gnm : G \cap\Pic = B\},
\]
the family of all graphs in $\Gnm$ that, when coloured by $\Pi$, have precisely the edges of $B$ monochromatic. Then
\[
  |\GPB| = \binom{e(\Pi)}{m-e(B)}
\]
and, since $e(B) \le \Delta(B)n \le kn$ for every $B \in \BPk$,
\begin{equation}
  \label{eq:BPkBound}
  |\BPk| \le \sum_{b=0}^{kn} \binom{e(\Pic)}{b} \leByRef{eq:binCoef4} \left(\frac{ en^2}{kn}\right)^{kn} \le e^{2kn\log n},
\end{equation}
provided that $n$ is sufficiently large. We also have
\[
  \Grk = \bigcup_{\Pi \in \Part}\bigcup_{B \in \BPk} \GPB.
\]

\subsection{The number of graphs with an unbalanced colouring}

The following proposition shows that if $m \gg n\log n$, then almost every graph in $\Grk$ cannot be coloured by an unbalanced partition $\Pi$ in such a way that the monochromatic graph has maximum degree at most $k$. In other words, for almost every $G \in \Grk$, all $r$-colourings of $G$ that yield a monochromatic subgraph with maximum degree $k$ are balanced.

\begin{prop}
  \label{prop:unbalanced-graphs}
  For all integers $k \ge 0$ and $r \ge 2$ and every positive $\gamma$, there exists a constant $C>0$ such that, if $m \ge Cn\log n$,
  \[
    \sum_{\Pi \notin \Part(\gamma)}\sum_{B\in\BPk} |\GPB| \ll \binom{\ex(n, K_{r+1})}{m} \le |\Grk|.
  \]
\end{prop}
\begin{proof}
  First note that, for an equipartition $\tilde{\Pi}$ of $\br{n}$ into $r$ parts, we have $|\cG_m(\tilde{\Pi},\emptyset)| = \binom{\ex(n, K_{r+1})}{m}$, so it is enough to check the first inequality.
  Assume that $m \ge Cn \log n$ for some sufficiently large constant $C$. We have
  \begin{equation}
    \label{eq:BadPiBound}
    \begin{aligned}
      \binom{\ex(n, K_{r+1})}{m}^{-1} & \sum_{B\in\BPk} |\GPB| 
      \leBy{\eqref{eq:binCoef1}, \eqref{eq:BPkBound}}
      \binom{\ex(n, K_{r+1})}{m}^{-1}  \binom{e(\Pi)}{m} e^{2kn\log n}m^{kn} \\
      &\leByRef{eq:ePi-gamma-upper}
      \left(1-\frac{\gamma^2}{3}\right)^m e^{4kn\log n} \le e^{-\gamma^2 m/3 +4kn\log n} \le e^{-\gamma^2 m/4}.
    \end{aligned}
  \end{equation}
  To complete the proof, note that there are at most $r^n$ different $r$-colourings and that $r^n \le e^{\gamma^2 m/5}$ when $n$ is sufficiently large. Consequently, summing~\eqref{eq:BadPiBound} over all $\Pi\notin\Part(\gamma)$ one gets the assertion of the proposition.
\end{proof}

\subsection{The number of graphs with many colourings}

Even though the collections $\GPB$ are generally not pairwise disjoint, there is not too much overlap between them. In other words, for all $\Pi\in\Partg$ and $B\in\BPk$, the pair $(\Pi,B)$ is the unique pair which covers $G$ for almost all $G \in \GPB$. More precisely, let $\UPB$ be the family of all graphs in $\GPB$ for which $(\Pi,B)$ is the unique pair which covers them. The following result is based on a result implicit in the work of Pr\"omel and Steger~\cite{PrSt95}. 

\begin{prop}
  \label{prop:UP}
  For all integers $k \ge 0$ and $r \ge 2$ and real number $a$, there exists a constant $c$ such that the following holds for all $\Pi \in \Part(\frac{1}{2r})$ and $B\in\BPk$. If $m \ge c n \log n$, then
  \[
  |\GPB \setminus \UPB| \le n^{-a} \cdot |\GPB|.
  \]
\end{prop}
\begin{proof}
  Fix some $\Pi \in \Part(\frac{1}{2r})$ and $\Pi' \in \Part(\frac{1}{2r}) \setminus \{\Pi\}$. Suppose that $\Pi = \{V_1, \dotsc, V_r\}$ and $\Pi' = \{V_1', \dotsc, V_r'\}$ and, for all $i, j \in \br{r}$, let $V_{i,j} = V_i \cap V_j'$. We will say that the vertices in $V_{i,j}$ are moved from $V_i$ to $V_j'$. For every $i \in \br{r}$, define $L_i$ and $S_i$ as the largest and the second largest subclasses of $V_i$, respectively. Note that $|V_i| \ge \frac{n}{2r}$ implies that $|L_i| \ge \frac{n}{2r^2}$. Set $s = \max_{j \in \br{r}} |S_j|$ and let $S = S_j$ for the smallest $j$ for which the maximum in the definition of $s$ is achieved. Note that $1 \le s \le n/2$, as $s = 0$ would imply that $(V_1', \dotsc, V_r')$ is a permutation of $(V_1, \dotsc, V_r)$, and therefore $\Pi = \Pi'$, which will imply also that $B=B'$ if $\GPB \cap \GPpBp \neq \emptyset$.

  Observe that either some pair $\{L_i, L_j\}$ of largest subclasses, or some largest subclass $L_i$ and $S$, where $S \nsubseteq V_i$, are moved to the same vertex class $V_z'$. Denote these sets $L_i$ and $L_j$ or $L_i$ and $S$ by $C$ and $D$. Since, for every $G\in\GPpBp$, the subgraph of $G$ induced by $V_z'$ has maximum degree at most $k$, it follows that, for every $G \in \GPB \cap \GPpBp$, the bipartite subgraph of $G$ induced between $C$ and $D$ also has maximum degree at most~$k$. In particular,
  \[
    e(C,D) \le k\cdot\min\{|C|,|D|\}.
  \]
  It follows that
  \begin{align*}
    (\star) 
		&= \Bigl|\GPB \cap \bigcup_{B'\in\BPpk}\GPpBp\Bigr| \\
                                                  &\le \sum_{t=0}^{k\cdot\min\{|C|,|D|\}}\binom{e(\Pi) - |C|\cdot|D|}{m-e(B)-t} \binom{|C|\cdot|D|}{t},
  \end{align*}
  since every $G \in \GPB$ contains $B$ and we need to specify its remaining $m-e(B)$ edges (by the definition of $C$ and $D$, no edge of $B$ connects these two sets). Consequently,
  \begin{align*}
    (\star) & \leByRef{eq:binCoef1} \binom{e(\Pi) - |C|\cdot|D|}{m-e(B)}\cdot\sum_{t=0}^{k\cdot\min\{|C|,|D|\}} \left(m-e(B)\right)^t\cdot\binom{|C|\cdot|D|}{t} \\
            &\leByRef{eq:binCoef4} \binom{e(\Pi) - |C|\cdot|D|}{m-e(B)} \cdot \left(\frac{(m-e(B)) \cdot e\max\{|C|,|D|\}}{k}\right)^{k\cdot\min\{|C|,|D|\}} \\
            &\leByRef{eq:binCoef2} \left(1 - \frac{|C|\cdot|D|}{n^2}\right)^{m-kn} \cdot \binom{e(\Pi)}{m-e(B)} \cdot e^{4k\cdot\min\{|C|,|D|\}\cdot\log n} \\
            & \le \exp\left(-\frac{|C| \cdot |D| \cdot (m-kn)}{n^2} + 4k \cdot \min\{|C|, |D|\} \cdot \log n\right) \cdot \binom{e(\Pi)}{m-e(B)}.
  \end{align*}
  If $m \ge c n \log n$ for a sufficiently large constant $c = c(k, r, a)$, then the simple bounds $\max\{|C|,|D|\} \ge\frac{n}{2r^2}$ and $\min\{|C|,|D|\} \ge\frac{s}{2r^2}$ imply that
  \begin{align*}
    (\star) & \le \exp\left(\left(-\frac{ m/2}{2r^2n}+4k\log n\right)\cdot\min\{|C|,|D|\}\right)\cdot\binom{e(\Pi)}{m-e(B)} \\
            &  \le n^{-(a+3)sr^2}\cdot \binom{e(\Pi)}{m-e(B)} = n^{-(a+3)sr^2}\cdot |\GPB|.
  \end{align*}
  Finally, observe that, given a $\Pi$, we can describe any $\Pi' \neq \Pi $ by first picking the partitions $\{V_{i,j}\}_{j \in \br{r}}$ for every $i$ and then setting $V_j' = \bigcup_{i \in \br{r}} V_{i,j}$. We claim that, for every $s$, there are at most $n^{r^2} \cdot n^{sr^2}=n^{(s+1)r^2}$ ways to choose  all $V_{i,j}$ so that $\max_{i \in \br{r}} |S_i| = s$.  Indeed, one may first specify the sequence $\big(|V_{i,j}|\big)_{i,j \in \br{r}}$ and then specify, for each $i \in \br{r}$, the elements of each $V_{i,j}$ with $j \in \br{r}$, apart from $L_i$ (which will comprise all the remaining, unspecified elements of $V_i$).  Therefore, by the above computation,
  \begin{align*}
    |\GPB \setminus \UPB| 
		& \le \sum_{\substack{\Pi'\in\Part(\frac{1}{2r})\\ \Pi'\ne\Pi}}\Bigl|\GPB \cap \bigcup_{B'\in\BPpk}\GPpBp\Bigr|\\
		& \le \sum_{\substack{\Pi'\in\Part(\frac{1}{2r})\\ \Pi'\ne\Pi}} n^{-(a+3)sr^2}\cdot |\GPB|\\
    &\le \sum_{s \ge 1} \left( n^{(s+1)r^2} \cdot n^{-(a+3)sr^2} \right) \cdot |\GPB| \le n^{-a} \cdot |\GPB|,
  \end{align*}
  as claimed.
\end{proof}

\subsection{Typical degrees in almost $r$-colourable graphs}

We shall now show that most vertices of almost every graph in $\Grk$ have degree exactly $k$ in the monochromatic graph. To make this informal statement precise, given a positive number~$\kappa$, denote by $\BPkk$ the family of all $B \in \BPk$ such that
\[
  \left|\big\{ v \in \br{n} : \deg_B(v) = k\big\}\right| \ge (1-\kappa)  n.
\]

\begin{prop}
  \label{prop:max-degree-exactly-k}
  For all integers $k \ge 0$ and $r \ge 2$, every positive $\kappa$, all $\Pi \in \Part$, and every $m$ satisfying $m \gg n$,
  \begin{equation}
    \label{eq:max-degree-exactly-k}
    \sum_{B\in\BPk \setminus \BPkk} |\GPB| \ll \sum_{B \in \BPk} |\GPB|.
  \end{equation}
\end{prop}
\begin{proof}
  Since $\BPzk = \BPz$, we may assume that $k \ge 1$. The left-hand and the right-hand sides of~\eqref{eq:max-degree-exactly-k} are cardinalities of the (disjoint) unions of families $\GPB$ over all $B \in \BPk \setminus \BPkk$ and all $B \in \BPk$, respectively; denote these two families of graphs by $\cF_L$ and $\cF_R$. We will compare the sizes of $\cF_L$ and $\cF_R$ by counting edges in a bipartite graph $\cH \subseteq \cF_L \times \cF_R$ defined as follows: A pair $(G_L, G_R) \in \cF_L \times \cF_R$ belongs to $\cH$ if and only if $G_R \setminus G_L$ is a single edge of $\Pic \setminus G_L$ and $G_L \setminus G_R$ is a single edge of $\Pi \cap G_L$.

  On the one hand, for every $G_R \in \cF_R$,
  \[
    \deg_{\cH}(G_R) \le e(\Pi) \cdot e(\Pic \cap G_R) \le n^2 \cdot kn.
  \]
  On the other hand, since every $B \in \BPk \setminus \BPkk$ contains more than $\kappa n$ vertices of degree smaller than $k$, at least $r \cdot \binom{\kappa n/r}{2}$ pairs of such vertices belong to the same colour class of $\Pi$. Consequently, for every $G_L \in \cF_L$,
  \[
    \begin{split}
      \deg_{\cH}(G_L) & \ge \left( r \cdot \binom{\kappa n/r}{2} - e(\Pic \cap G_L) \right) \cdot e(\Pi \cap G_L) \\
      & \ge \left(\frac{\kappa^2n^2}{3r} - kn\right) \cdot (m-kn) \ge \frac{\kappa^2 n^2}{4r} \cdot \frac{m}{2}.
  \end{split}
  \]
  We conclude that
  \[
    |\cF_L| \cdot \frac{\kappa^2n^2m}{8r} \le e(\cH) \le |\cF_R| \cdot kn^3,
  \]
  which implies that $    |\cF_L| \le (8kr/\kappa^2) \cdot (n/m) \cdot |\cF_R| \ll |\cF_R|$, as claimed.
\end{proof}

\subsection{Almost $r$-colourable graphs with large monochromatic girth}

We shall now show that in a constant proportion of graphs in $\Grk$, the monochromatic graph has large girth. To make this informal statement precise, given an integer $g \ge 3$, denote by $\BPkg$ the family of all graphs in $\BPk$ that do not contain any cycles of length at most~$g$. The following statement is a key ingredient in our proof of Theorem~\ref{thm:full}.

\begin{prop}
  \label{prop:large-girth}
  For all integers $k \ge 0$, $r \ge 2$, and $g \ge 3$, there exists a positive constant $c$ such that, for all $\Pi \in \Part(\frac{1}{2r})$ and every $m$ satisfying $m \gg n$,
  \[
    \sum_{B \in \BPkg} |\GPB| \ge c \cdot \sum_{B \in \BPk} |\GPB|.
  \]
\end{prop}

The proof of this proposition is a relatively straightforward corollary of Proposition~\ref{prop:max-degree-exactly-k} and the following classical result of Bollob\'as~\cite{Bol80} and Wormald~\cite{Wor81}.

\begin{thm}[{\cite[Theorem~2]{Bol80}}]
  \label{thm:Bollobas}
  Suppose that $k \ge 2$ and $g \ge 3$ are integers and let $0 \le d_1 \le \dotsb \le d_n \le k$ be such that $\sum_{i=1}^n d_i \eqqcolon 2m$ is even and $2m - n \to \infty$ as $n \to \infty$. Let $G$ be a graph chosen uniformly at random from the family of all graphs with vertex set $\br{n}$ such that $\deg_G(i) = d_i$ for every $i \in \br{n}$ and, for each $\ell \ge 3$, denote by $X_\ell$ the number of cycles of length $\ell$ in $G$. Denote by $(Z_3, \dotsc, Z_g)$ the vector of independent Poisson random variables with
  \[
    \Ex[Z_\ell] = \frac{1}{2\ell} \left(\frac{1}{m}\sum_{i=1}^n \binom{d_i}{2}\right)^\ell
  \]
  for each $\ell$. Then
  \[
    \lim_{n \to \infty} \dTV\big((X_3, \dotsc, X_g), (Z_3, \dotsc, Z_g)\big) = 0,
  \]
  where $\dTV$ is the total variation distance.
\end{thm}

\begin{proof}[{Proof of Proposition~\ref{prop:large-girth}}]
  We may assume that $k \ge 2$, since otherwise no graph in $\BPk$ can contain a cycle and thus $\BPkg = \BPk$. Suppose that $\Pi = \{V_1, \dotsc, V_r\} \in \Part(\frac{1}{2r})$ and let $G$ be a uniformly random element of $\bigcup_{B \in \BPk} \GPB$. Conditioned on $G \cap \Pi$ and the degree sequence of $G \cap \Pic$, the graphs $G[V_1], \dotsc, G[V_r]$ become independent, uniformly chosen random graphs with respective degree sequences. By Proposition~\ref{prop:max-degree-exactly-k}, invoked with $\kappa = 1/(6r)$, with probability $1-o(1)$,
  \begin{equation}
    \label{eq:G-Vi-many-edge}
    \sum_{v \in V_i} \deg_{G[V_i]}(v) \ge (|V_i| - \kappa n) \cdot k \ge \frac{2|V_i|}{3} \cdot k \ge \frac{4|V_i|}{3}
  \end{equation}
  for each $i \in \br{r}$, as $\min_i |V_i| \ge n/(2r) = 3\kappa n$. Since, for every $i \in \br{r}$,
  \[
    \frac{1}{e(G[V_i])} \sum_{v \in V_i} \binom{\deg_{G[V_i]}(v)}{2} \le \frac{1}{e(G[V_i])} \sum_{v \in V_i} \frac{\deg_{G[V_i]}(v) \cdot (k-1)}{2} = k-1,
  \]
  Theorem~\ref{thm:Bollobas} implies that, if the degree sequence of $G \cap \Pic$ satisfies~\eqref{eq:G-Vi-many-edge} for every $i \in \br{r}$, which happens with probability $1-o(1)$,
  \begin{multline*}
    \Pr(G \cap \Pic \in \BPkg \mid G \cap \Pi, \text{degree sequence of $G \cap \Pic$}) \\
    \ge \left(\frac{1}{2} \cdot \inf\left\{\prod_{\ell=3}^g \Pr\big(\Pois(\lambda) = 0\big) : 0 \le \lambda \le \frac{(k-1)^\ell}{2\ell}\right\}\right)^{r},
  \end{multline*}
  where $\Pois(\lambda)$ denotes the Poisson random variable with mean $\lambda$. The assertion of the proposition follows as $\Pr(\Pois(\lambda) = 0) = e^{-\lambda}$ and $g,k,r = O(1)$.
\end{proof}

\section{The 0-statement}
\label{sec:zeroState}

In this section, we treat the $0$-statement of Theorem~\ref{thm:full}.  First, using an elementary counting argument, we show that, for every graph $H$ with maximum degree at least two, if $m \ll \twoDensTS$, then the family $\Free$ constitutes an $e^{-o(m)}$-proportion of all graphs with $n$ vertices and $m$ edges.  Using a standard estimate on the lower tails of hypergeometric distributions, it will be fairly straightforward to deduce that, when $n \ll m \ll \twoDensTS$ and both $r$ and $k$ are bounded, the family $\Grk$ is far smaller than $\Free$.  The details are presented in Section~\ref{sec:below-2-density}.

Second, using a much more subtle argument, we show that, for every plain vertex-critical graph $H$ with criticality $k+1$ and chromatic number $r+1$, if $\Omega\big(\twoDensTS\big) \le m \le c \criticalityTS$ for a sufficiently small positive $c$, the number of graphs in $\Free$ that are `one edge away' from being in $\Grk$ is far greater than the number of graphs in $\Grk$.  Our argument, which relies on the Hypergeometric Harris Inequality as well as several crucial properties of graphs in $\Grk$ that we have established in Section~\ref{sec:almostRcolour}, is presented in Section~\ref{sec:above-2-density}.

\subsection{Below the 2-density}
\label{sec:below-2-density}

We first give a simple lower bound on $|\Free|$, valid for every graph $H$ with maximum degree at least two, that exploits the fact that, if $m \ll \twoDensTS$, a typical graph in $\Gnm$ can be made $H$-free by removing from it some $o(m)$ edges.

\begin{prop}
  \label{prop:below2den}
  Let $H$ be an arbitrary graph with maximum degree at least two. For every $\eps>0$, there is a $\delta>0$ such that, for all sufficiently large $m$ and every $m\le \delta\twoDensTS$,
  \[
  |\Free| \ge e^{-\eps m} \cdot \binom{\binom{n}{2}}{m}.
\]
\end{prop}
\begin{proof}
  Suppose that $H$ is a graph with maximum degree at least two. This means that $K_{1,2} \subseteq H$ and hence $m_2(H) \ge m_2(K_{1,2}) \ge 1$. Suppose that $\eps$ is a positive number. Let $F$ be an arbitrary subgraph of $H$ such that $d_2(F) = m_2(H)$ and note that $e_F \ge 2$, as $m_2(H) \ge 1$. Finally, let $\delta$ be a small positive number satisfying
  \begin{equation}
    \label{eq:below2den-delta-choice}
    (6\delta)^{e_F -2} \le \frac{1}{72} \qquad \text{and} \qquad \left(\frac{\delta}{e(1+2\delta)}\right)^{2\delta} \ge e^{-\eps/2}.
  \end{equation}

  Let $m$ be a positive integer satisfying $m\le \delta \twoDensTS$. If $m \le n^{1/3}$, we let $G$ be a~uniformly chosen random graph in $\Gnm$ and note that
  \[
    \Pr(H \subseteq G) \le \Pr(K_{1,2} \subseteq G) \le n^3 \cdot \left(\frac{m}{\binom{n}{2}}\right)^2 \le 5n^{-1/3} \le 1-e^{-\eps},
  \]
  provided that $n$ is sufficiently large. Consequently,
  \[
    |\Free| = \Pr(H \nsubseteq G) \cdot \binom{\binom{n}{2}}{m} \ge e^{-\eps m} \cdot \binom{\binom{n}{2}}{m},
  \]
  as desired. We may thus assume from now on that $m > n^{1/3}$.
  
  Set $m' = \lceil (1+\delta)m \rceil$ and note that
  \[
    m' \le (1+\delta)\delta \twoDensTS + 1\le 2\delta \twoDensTS = 2\delta n^{2-\frac{1}{d_2(F)}}
  \]
  provided that $n$ is sufficiently large. Now, let $G$ be a uniformly chosen random graph in $\Gnmp$, and let $X$ denote the number of copies of $F$ in $G$. Recalling that $d_2(F) = (e_F-1)/(v_F-2)$, we have
  \[
    \begin{split}
      \Ex[X] & \le n^{v_F} \cdot \left(\frac{m'}{\binom{n}{2}}\right)^{e_F} \le n^{v_F} \cdot \left(\frac{3m'}{n^2}\right)^{e_F} \le n^{v_F} \cdot \frac{3m'}{n^2} \cdot \left(6 \delta n^{-\frac{1}{d_2(F)}}\right)^{e_F-1} \\
      & = (6\delta)^{e_F-1} \cdot 3m' \leByRef{eq:below2den-delta-choice} \frac{\delta m'}{4} \le \frac{\delta m}{2} ,
    \end{split}
  \]
  and consequently, by Markov's inequality,
  \[
    \Pr(X\ge m'-m) = \Pr(X\ge \delta m) \le \frac{1}{2}.
  \]
  We conclude that at least half of the graphs in $\Gnmp$ contain a subgraph with $m$ edges that is $F$-free and thus also $H$-free. (Indeed, we may delete an arbitrary edge from each of the at most $m' -m$ copies of $F$ in the original graph). By double counting,
  \[
    |\Free| \cdot \binom{\binom{n}{2}-m}{m'-m} \ge \frac{1}{2} \cdot \binom{\binom{n}{2}}{m'}.
  \]
  It follows that, denoting $N = \binom{n}{2}$,
  \[
    \frac{|\Free|}{\binom{N}{m}} \ge \frac{1}{2}\cdot\frac{\binom{N}{m'}}{\binom{N}{m}\binom{N-m}{m'-m}} = \frac{1}{2 \cdot\binom{m'}{m'-m}} \geByRef{eq:binCoef4} \frac{1}{2} \cdot \left(\frac{em'}{m'-m}\right)^{m-m'}.
  \]
  Finally, since $(1+\delta) m \le m' \le (1+\delta)m + 1 \le (1+2\delta) m$,  we conclude that
  \[
    \frac{|\Free|}{\binom{\binom{n}{2}}{m}} \ge \frac{1}{2} \cdot \left(\frac{\delta}{e(1+2\delta)}\right)^{2\delta m} \geByRef{eq:below2den-delta-choice} \frac{1}{2} \cdot e^{-\eps m/2} \ge e^{-\eps m},
  \]
  provided that $n$ is sufficiently large.
\end{proof}

In order to bound the number of graphs in $\Grk$ from above, we use the simple observation that every graph in $\Grk$ contains a set of at least $n/r$ vertices that induces a graph with average degree at most $k$, which is much less than the expected average degree of a graph that such a set would induce in a uniformly chosen random graph from $\Gnm$.

\begin{prop}
  \label{prop:Grk-upper}
  For all positive integers $k$, $r$, $n$, and $m$ satisfying $m \ge 6r^2(k+2)n$, we have
  \[
    |\Grk| \le  \exp\left(-\frac{m}{4r^2}\right) \cdot \binom{\binom{n}{2}}{m},
  \]
  provided that $n$ is sufficiently large.
\end{prop}
\begin{proof}
  Observe that, for every graph $G \in \Grk$, there is a set $W \subseteq \br{n}$ with at least $n/r$ elements such that $e(G[W]) \le k|W|/2$. In particular, if $G$ is a uniformly chosen random graph from $\Gnm$,
  \[
    |\Grk| \le \sum_{\substack{W \subseteq \br{n} \\ |W| \ge n/r}} \Pr\big(e(G[W]) \le k|W|/2 \big) \cdot \binom{\binom{n}{2}}{m}.
  \]
  We may bound each term in the above sum using Lemma~\ref{lemma:hyper-lower-tail}, invoked with $\Omega = \binom{\br{n}}{2}$ and $A = \binom{W}{2}$. Indeed, letting
  \[
    t = \frac{m\binom{|W|}{2}}{\binom{n}{2}} - \frac{k|W|}{2},
  \]
  we have
  \[
    \Pr\big(e(G[W]) \le k|W|/2 \big) \le \exp\left(-\frac{t^2}{2m\binom{|W|}{2}/\binom{n}{2}}\right) \le \exp\left(-\frac{m \binom{|W|}{2}}{2\binom{n}{2}} + \frac{k|W|}{2}\right).
  \]
  If $n$ is sufficiently large, then, for every $W$ with $n/r \le |W| \le n$,
  \[
    \frac{m \binom{|W|}{2}}{2\binom{n}{2}} - \frac{k|W|}{2} \ge \frac{m}{2} \cdot \frac{n/r \cdot (n/r-1)}{n \cdot (n-1)} - \frac{kn}{2} \ge \frac{m}{3r^2} - \frac{kn}{2},
  \]
  and, consequently,
  \[
    |\Grk| \le 2^n \cdot \exp\left(-\frac{m}{3r^2} + \frac{kn}{2}\right) \cdot \binom{\binom{n}{2}}{m}.
  \]
  The claimed bound now follows from our assumption that $m \ge 6r^2(k+2)n$.
\end{proof}

Propositions~\ref{prop:below2den} and~\ref{prop:Grk-upper} immediately yield the following corollary.

\begin{cor}
  \label{cor:0-statement-below-2-density}
  Let $H$ be an arbitrary graph with maximum degree at least two and let $k$ and $r$ be positive integers. There exists a positive constant $c$ such that, if $n \ll m \le c \twoDensTS$,
  \[
    |\Free| \gg |\Grk|.
  \]
\end{cor}

\subsection{Above the 2-density}
\label{sec:above-2-density}

In this section, we show that, if $H$ is a plain vertex-critical graph with criticality $k+1$ and chromatic number $r+1 \ge 3$, then there exists a positive constant $c_H$ such that $|\Free| \gg |\Grk|$ for every $m$ satisfying $\Omega\big(\twoDensTS\big) \le m \le c_Hm_H$. More precisely, we will show that the number of graphs in $\Free$ that are `one edge away' from being in $\Grk$, i.e., graphs $G \in \Free$ such that $G \setminus e \in \Grk$ for some $e \in G$, is far greater than the number of graphs in $\Grk$.

\begin{prop}
  \label{prop:zeroState}
  Suppose that $H$ is a plain vertex-critical graph with criticality $k+1$ and chromatic number $r+1 \ge 3$. There is a positive constant $c_H$ such that, if $m$ satisfies
  \[
    \Omega\Big(\twoDensTS\Big)\le  m \le c_H m_H,
  \]
  then $|\Free|\gg|\Grk|$.
\end{prop}

The main ingredient in our proof of this proposition is the following lower bound on the number of $H$-free graphs that are `one edge away' from $\GPB$ for given balanced $r$-colouring $\Pi$ and $B \in \BPk$ with large girth.

\begin{lemma}
  \label{lemma:0-statement-above-2-density}
  Suppose that $H$ is a plain vertex-critical graph with criticality $k+1$ and chromatic number $r+1 \ge 3$. For every $\eps > 0$, there exists a positive constant $c$ such that the following holds for every $m$ that satisfies
  \[
    n\log n \ll m \le c \criticalityTS.
  \]
  For every $\Pi \in \Part(\frac{1}{20r})$, all $B \in \BPk$, and each $e \in \Pic \setminus B$ such that $B \cup e$ has girth larger than $v_H$,
  \[
    |\UPBe \cap \Free| \ge \frac{m}{n^{2+\eps}} \cdot |\GPB|.
  \]
\end{lemma}

\begin{proof}
  Note first that only the two endpoints of $e$ can have degree larger than $k$ in the graph $B \cup e$ and that, by assumption, the girth of $B \cup e$ is larger than $v_H$. The definition of plain vertex-critical graphs guarantees that, for every embedding $\varphi$ of $H$ into $\Pi \cup B \cup e$, there must be a critical star $S \subseteq H$ with~$k+1$ edges such that $\varphi(H) \cap (B \cup e) = \varphi(S)$ and $e \in \varphi(S)$; in particular, for every $S \subseteq F \subseteq H$, the map $\varphi$, restricted to $V(F)$, is also an embedding of $F$ into $\Pi \cup B \cup e$ that maps $S$ to $B \cup e$ and $F \setminus S$ to $\Pi$.

  Let $S_1, \dotsc, S_t$ be all the smallest critical stars of $H$ (i.e., with~$k+1$ edges) and, for each $i \in \br{t}$, let $F_i$ be a subgraph satisfying
  \[
    S_i \subsetneq F_i \subseteq H, \qquad d_{k+2}(F_i) = \eta_i(H), \qquad \text{and} \qquad e_{F_i} = \zeta_i(H).
  \]
  Let $G$ be a uniformly chosen random element of $\GPBe$, let $G' = G \cap \Pi$, and observe that $G'$ is a uniformly random subgraph of $\Pi$ with $m - e(B) - 1$ edges. For every $i$ and every injection $\varphi \colon V(F_i) \to \br{n}$, we let $S_{i,\varphi} = \varphi(S_i)$ and $K_{i,\varphi} = \varphi(F_i \setminus S_i)$ be the labelled graphs that are the images of $S_i$ and $F_i \setminus S_i$ via the embedding $\varphi$. Define
  \[
    \Phi_i = \{\varphi : S_{i,\varphi} \subseteq B \cup e \text{ and } K_{i,\varphi} \subseteq \Pi\};
  \]
  in other words, $\Phi_i$ comprises all those embeddings of $F_i$ into $\Pi \cup B \cup e$ that embed $S_i$ into $B \cup e$ and map the remaining edges of $F_i$ to $\Pi$. Since $B \cup e$ contains at most two copies of $S_i$, one for each endpoint of $e$, we have $|\Phi_i| \le 2n^{v_{F_i}-k-2}$. More importantly, the above discussion implies that
  \[
    |\GPBe \cap \Free| \ge \underbrace{\Pr(G' \nsupseteq K_{i,\varphi} \text{ for all $i$ and $\varphi \in \Phi_i$})}_{P} \cdot \binom{e(\Pi)}{m-e(B)-1}.
  \]

  Assume that $m \le c \criticalityTS$, where
  \begin{equation}
    \label{eq:cH-0-statement}
    c = \frac{\eps}{64t}.
  \end{equation}
  We shall bound $P$ from below using the Hypergeometric Harris Inequality (Lemma~\ref{lem:HFKG}). To this end, let
  \[
    p = \frac{3}{2} \cdot \frac{m-e(B)-1}{e(\Pi)}
  \]
  and note that $p \le \frac{8m}{n^2}$, by part~\ref{item:Partg-lower} of Proposition~\ref{prop:bal-unbal-part-size}, as $\Pi \in \Part(\frac{1}{20r})$. It follows from Lemma~\ref{lem:HFKG} that
  \[
    \begin{split}
      P + \exp(-m/16)  & \ge \prod_{i=1}^t \left(1 - p^{e_{F_i\setminus S_i}}\right)^{|\Phi_i|} \ge \exp\left(-\sum_{i=1}^t |\Phi_i| \cdot 2p^{e_{F_i \setminus S_i}}\right) \\
      & \ge \exp\left(-\sum_{i=1}^t 4n^{v_{F_i}-k-2} p^{e_{F_i}-k-1}\right).
    \end{split}
  \]
  \begin{claim}
    \label{claim:criticality-0-statement}
    For every $i \in \br{t}$,
    \[
      n^{v_{F_i}-k-2}p^{e_{F_i}-k-1} \le 8c\log n.
    \]
  \end{claim}
  \begin{proof}
    Since $e_{F_i} > e_{S_i} = k+1$, we have
    \[
      \begin{split}
        n^{v_{F_i}-k-2}p^{e_{F_i}-k-1} & \le  n^{v_{F_i}-k-2} \cdot \left(\frac{8cm}{n^2}\right)^{e_{F_i}-k-1} \\
        & \le 8c \cdot n^{v_{F_i}-k-2} \cdot \left(\frac{m}{n^2}\right)^{e_{F_i}-k-1} \\
        & \le 8c \cdot \left(n^{\frac{1}{d_{k+2}(F_i)}} \cdot \frac{m}{n^2}\right)^{e_{F_i}-k-1} \\
        & \le 8c \cdot \left(n^{\frac{1}{\eta_i(H)} - \frac{1}{\eta(H)}} \cdot (\log n)^{\frac{1}{\zeta(H)-k-1}}\right)^{\zeta_i(H)-k-1}.
      \end{split}
    \]
    The claimed upper bound follows since $\eta_i(H) \ge \eta(H)$ and $\zeta_i(H) \le \zeta(H)$ whenever $\eta_i(H) = \eta(H)$.
  \end{proof}

  In particular, assuming that $n$ is large, we have
  \[
    P \ge \exp\left(-32tc\log n\right ) - \exp(-m/16) \geByRef{eq:cH-0-statement} n^{-\eps/2} - \exp(-n) .
  \]
  Since $B \cup e \in \BPkpo$, we may now invoke Proposition~\ref{prop:UP} with $a = \eps$ to obtain
  \[
    \begin{split}
      |\UPBe \cap \Free| & \ge |\GPBe \cap \Free| - |\GPBe \setminus \UPBe| \\
      & \ge (P - n^{-\eps}) \cdot \binom{e(\Pi)}{m-e(B)-1} \ge n^{-\eps} \cdot \binom{e(\Pi)}{m-e(B)-1}.
  \end{split}
  \]
  Since $m-e(B) \ge m-kn \ge m/2$ and $e(\Pi) \le n^2/2$, we may conclude that
  \[
    |\UPBe \cap \Free| \ge \frac{m}{n^{2+\eps}} \cdot \binom{e(\Pi)}{m-e(B)} = \frac{m}{n^{2+\eps}} \cdot |\GPB|,
  \]
  as claimed.
\end{proof}

\begin{proof}[{Proof of Proposition \ref{prop:zeroState}}]
  If $\eta(H) < m_2(H)$, then $m_H = \twoDensTS$ and we may simply invoke Corollary~\ref{cor:0-statement-below-2-density} and let $c_H = c_{\ref{cor:0-statement-below-2-density}}$. If this is not the case, then $m_H = \criticalityTS$ and we let $c_H = c_{\ref{lemma:0-statement-above-2-density}}(\eps)$, where $2\eps = 1-1/m_2(H) > 0$.

  Since, for all $\Pi \in \Part$, every $B' \subseteq \Pic$ can be written as $B' = B \cup e$ with $B \in \BPk$ and $e \notin B$ in at most $e(B)+1 \le kn$ different ways, we have
  \begin{equation}
    \label{eq:Free-0-statement-lower}
    |\Free| \ge \frac{1}{kn} \sum_{\Pi \in \Part} \sum_{B \in \BPk} \sum_{e \in \Pic \setminus B} |\UPBe \cap \Free|.
  \end{equation}
  Further, observe that, for every $B \in \BPkv$, the number of edges $e \in \Pic \setminus B$ such that $B \cup e$ has girth larger than $v_H$ is at least
  \[
    e(\Pic) - e(B) - n \cdot \sum_{\ell=2}^{v_H-1} k(k-1)^{\ell-1} \ge \frac{n^2}{4r}.
  \]
  Let $\gamma = 1/(20r)$. Since $m = \Omega\big(\twoDensTS\big) = \Omega\big(n^{1+2\eps}\big)$, we may conclude that
  \[
    \begin{split}
      |\Free| & \geBy{L~\ref{lemma:0-statement-above-2-density}} \frac{1}{kn} \sum_{\Pi \in \Partg} \sum_{B \in \BPkv} \frac{n^2}{4r} \cdot \frac{m}{n^{2+\eps}} \cdot |\GPB| \\
      & = \frac{m}{4krn^{1+\eps}} \sum_{\Pi \in \Partg} \sum_{B \in \BPkv} |\GPB| \\
      & \geBy{P~\ref{prop:large-girth}} \frac{c_{~\ref{prop:large-girth}}m}{4krn^{1+\eps}} \sum_{\Pi \in \Partg} \sum_{B \in \BPk} |\GPB| \\
      & \gg \sum_{\Pi \in \Partg} \sum_{B \in \BPk} |\GPB|.
    \end{split}
  \]
  On the other hand,
  \[
    |\Grk| \le \sum_{\Pi \in \Part} \sum_{B \in \BPk} |\GPB| \leBy{P~\ref{prop:unbalanced-graphs}} 2 \sum_{\Pi \in \Partg} \sum_{B \in \BPk} |\GPB|.
  \]
  These two estimates imply the assertion of the proposition.
\end{proof}

\section{Approximate 1-statement}
\label{sec:approx-1-stat}

In this section, we show that, for every graph $H$ with $\chi(H) = r+1 \ge 3$, then, as soon as $m \gg \twoDensTS$, most graphs in $\Free$ admit a balanced, unfriendly $r$-colouring that leaves only $o(m)$ edges monochromatic. 
Given a graph~$G$, one of its vertices~$v \in V(G)$ and a set~$U \subseteq V(G)$, we denote the number of neighbours of~$v$ in the set~$U$ by~$\deg_G(v,U)$.

\begin{thm}
  \label{thm:aa-balanced-unfriendly}
  Suppose that a graph $H$ satisfies $\chi(H) = r+1 \ge 3$. For all positive $\delta$ and $\gamma$, there is a positive $C$ such that the following holds. If $m\ge C\twoDensTS$, then almost every graph $G$ in $\Free$ admits a partition $\Pi\in\Partg$ such that
  \begin{equation}\label{eq:almostRcolour}
    e(G\setminus \Pi) \le \delta m
  \end{equation}
  and, letting $\Pi = \{V_1, \dotsc, V_r\}$,
  \begin{equation} \label{eq:Pi-unfriendly}
    \deg_G(v,V_i) \le \min_{j \neq i} \deg_G(v, V_j) \quad \text{for all $i \in \br{r}$ and $v \in V_i$}.
  \end{equation}
\end{thm}

Our proof of Theorem~\ref{thm:aa-balanced-unfriendly} relies on the following result established in~\cite[Theorem~1.7]{BaMoSa15}, which states that, for every $H$ with $\chi(H) = r+1 \ge 3$, when $m \gg \twoDensTS$, then most graphs $G \in \Free$ admit an $r$-partition $\Pi \in \Part$ such that $e(G \setminus \Pi) = o(m)$. With little extra work, we will show that, for most such $G$, one such partition $\Pi$ is balanced (i.e., it belongs to $\Partg$ for some small $\gamma$) and unfriendly (i.e., it satisfies~\eqref{eq:Pi-unfriendly}).

\begin{thm}
  \label{thm:approx-struct}
  For every graph $H$ with $\chi(H) \ge 3$ and every positive $\delta$, there exists a positive constant~$C$ such that the following holds. If $m \ge C\twoDensTS$, then almost every graph in $\Free$ can be made $\left(\chi(H)-1\right)$-partite by removing from it at most $\delta m$ edges.
\end{thm}

As a next step towards establishing Theorem~\ref{thm:aa-balanced-unfriendly}, we now show that very few $G \in \Free$ admit a non-balanced partition $\Pi$ that satisfies $e(G \setminus \Pi) \le \delta m$.

\begin{prop}
  \label{prop:approx-struct}
  Suppose that a graph $H$ satisfies $\chi(H) = r+1 \ge 3$. For all positive~$\delta$ and $\gamma$ and all $m \gg n$, almost every $G \in \Free$ does not admit a partition $\Pi\in\Part\setminus\Partg$ that satisfies $e(G \setminus \Pi) \le \delta m$.
\end{prop}
\begin{proof}
  Since making $\delta$ smaller only strengthens the assertion of the theorem, we may assume without loss of generality that $\delta \le \gamma^2 6$ and that
  \begin{equation}
    \label{eq:gamma-delta}
    \delta \cdot \big( 4-\log (\gamma^2\delta) \big) - (1-\delta) \cdot \frac{\gamma^2}{3} < - \frac{\gamma^2}{4};
  \end{equation}
  indeed, as $\delta \to 0$, the left-hand side of~\eqref{eq:gamma-delta} converges to $-\gamma^2/3$.

  Fix an arbitrary partition $\Pi \in \Part \setminus \Partg$, that is, a $\Pi \in \Part$ that does not satisfy~\eqref{eq:Part-gamma} and recall from the proof of Proposition~\ref{prop:unbalanced-graphs}, see~\eqref{eq:ePi-gamma-upper}, that
  \[
    e(\Pi) \le \left(1-\frac{\gamma^2}{3}\right) \cdot \ex(n, K_{r+1}) \le \left(1 - \frac{\gamma^2}{3}\right) \cdot \exnH.
  \]
  Denote $N=\binom{n}{2}$ and $N' = \exnH$. The number $X_\Pi$ of graphs $G \in \Free$ for which $e(G \setminus \Pi) \le \delta m$ satisfies
  \[
    (\star) =\binom{N'}{m}^{-1} \cdot X_\Pi 
    \le \sum_{t = 0}^{\delta m} \frac{\binom{N}{t} \binom{e(\Pi)}{m-t}}{\binom{N'}{m}}
    = \sum_{t = 0}^{\delta m} \frac{\binom{N}{t} \binom{e(\Pi)}{m-t}\binom{m}{t}}{\binom{N'}{m-t}\binom{N'-m+t}{t}}.
  \]
  Note that, for every $t$, either $m-t \le e(\Pi)$ or the corresponding summand is equal  to zero. This observation and the above bound on $e(\Pi)$ imply that
  \[
    \binom{e(\Pi)}{m-t} \binom{N'}{m-t}^{-1} \leByRef{eq:binCoef2} \left(\frac{e(\Pi)}{N'}\right)^{m-t} \le \left(1 - \frac{\gamma^2}{3}\right)^{m-t}
  \]
  and that
  \[
    \begin{split}
      \binom{N}{t} \binom{N'-m+t}{t}^{-1} & \le \binom{N}{t} \binom{N'-e(\Pi)}{t}^{-1} \le \binom{N}{t} \binom{\gamma^2N'/3}{t}^{-1} \\
      & \leByRef{eq:binCoef3} \left(\frac{N}{\gamma^2N'/3-t}\right)^t \le \left(\frac{6N}{\gamma^2N'}\right)^t \le \left(\frac{12}{\gamma^2}\right)^t,
    \end{split}
  \]
  as $t \le \delta m \le \delta N' \le \gamma^2N'/6$ and $N' \ge \ex(n,K_3) \ge N/2$. Consequently,
  \[
    (\star) \le \sum_{t = 0}^{\delta m} \left(1-\frac{\gamma^2}{3}\right)^{m-t} \left(\frac{12}{\gamma^2}\right)^t \binom{m}{t} \le \left(1-\frac{\gamma^2}{3}\right)^{(1-\delta)m} \left(\frac{12}{\gamma^2}\right)^{\delta m} \cdot \sum_{t=0}^{\delta m} \binom{m}{t}.
  \]
  Since $\delta\le \frac{1}{2}$, inequalities~\eqref{eq:binCoef4} and $\log(12e) \le 4$ further imply that
  \[
    \begin{split}
      (\star) &\le \left(1-\frac{\gamma^2}{3}\right)^{(1-\delta)m} \left(\frac{12}{\gamma^2} \cdot \frac{e}{\delta}\right)^{\delta m}
      \le \exp\left(\left(\delta \cdot (4-\log(\gamma^2\delta)) - (1-\delta) \cdot \frac{\gamma^2}{3}\right) \cdot m\right) \\
      & \leByRef{eq:gamma-delta} \exp\left(-\frac{\gamma^2m}{4}\right).
    \end{split}
  \]

  Finally, since there are at most $r^n$ partitions $\Pi \in \Part$ and at least $\binom{N'}{m}$ graphs in $\Free$ and since $m \gg n$, we have
  \[
    \sum_{\Pi \in \Part \setminus \Partg} X_\Pi \le r^n \cdot e^{-\gamma^2m/4} \cdot \binom{N'}{m} \le e^{-\gamma^2m/5} \cdot |\Free|,
  \]
  which implies the assertion of the proposition.
\end{proof}

\begin{proof}[{Proof of Theorem~\ref{thm:aa-balanced-unfriendly}}]
  Let $\Freedg$ be the collection of all graphs $G \in \Free$ that satisfy~\eqref{eq:almostRcolour} for some $\Pi \in \Part(\gamma)$ but no $\Pi \in \Part \setminus \Part(\gamma)$. Let $C = C_{\ref{thm:approx-struct}}(\delta)$ and assume that $m \ge C \twoDensTS$. Since Theorem~\ref{thm:approx-struct} and Proposition~\ref{prop:approx-struct} imply that almost all graphs in $\Free$ belong to $\Freedg$, it is enough to show that every $G \in \Freedg$ admits a partition $\Pi  = \{V_1, \dotsc, V_r\} \in \Partg$ that satisfies both~\eqref{eq:almostRcolour} and~\eqref{eq:Pi-unfriendly}.

  To see this, given an arbitrary $G \in \Freedg$, let $\Pi \in \Part$ be a partition that minimises $e(G \setminus \Pi)$ over all $r$-partitions of $\br{n}$. Since $e(G \setminus \Pi) \le \delta m$, by the definition of $\Freedg$ and the minimality of $\Pi$, then $\Pi \in \Part(\gamma)$, again by the definition of $\Freedg$. Suppose that $\Pi = \{V_1, \dotsc, V_r\}$. If there were $i, j \in \br{r}$ and $v \in V_i$ such that $\deg_G(v, V_i) > \deg_G(v, V_j)$, then the partition $\Pi'$ obtained from $\Pi$ by moving the vertex $v$ from $V_i$ to $V_j$ would satisfy $e(G \setminus \Pi') < e(G \setminus \Pi)$, contradicting the minimality of $\Pi$.
\end{proof}

\section{The 1-statement}
\label{sec:oneState}

In this section, we prepare for the proof of the $1$-statement of Theorem~\ref{thm:full}. Our goal is to show that, if $H$ is a simple vertex-critical graph with criticality $k+1$ and chromatic number $r+1 \ge 3$, then there is a positive constant $C_H$ such that, if $m \ge C_H m_H$, then almost every graph from $\Free$ belongs to $\Grk$; recall that $m_H$ is the threshold function defined in~\eqref{eq:mH-def}. Note that it suffices to prove this statement only for graphs $H$ that have no isolated vertices.

\subsection{A sufficient condition}
\label{sec:sufficient-condition}

Given a positive constant $\delta$ and a balanced $r$-partition $\Pi = \{V_1, \dotsc, V_r\} \in \Part(\gamma)$, let $\Freedp$ be the family of all $G \in \Free$ for which $\Pi$ is an unfriendly partition that leaves at most $\delta m$ edges of $G$ monochromatic, that is,
\[
  \Freedp = \big\{G \in \Free : \text{$(G, \Pi)$ satisfy~\eqref{eq:almostRcolour} and~\eqref{eq:Pi-unfriendly}} \big\}
\]
and let
\[
  \FreesdP = \big\{G \in \Freedp : G \setminus \Pi \notin\BPk\big\}.
\]
In other words, $\FreesdP$ comprises all those graphs $G \in \Freedp$ for which the monochromatic subgraph of $G$ induced by the $r$-colouring $\Pi$ has maximum degree larger than $k$. The following proposition gives a sufficient condition for the assertion of the $1$-statement of Theorem~\ref{thm:full} to hold true, that is, a sufficient condition for the asymptotic inequality $|\Free \setminus \Grk| \ll |\Free|$.

\begin{prop}
  \label{prop:sufficient-1-statement}
  Suppose that $H$ is a simple vertex-critical graph with $\chi(H) = r+1 \ge 3$ and criticality $k+1$. For all positive $\delta$ and $\gamma$, there exists a constant $C$ such that the following holds when $m \ge C \twoDensTS$: Suppose that there is a function $\omega \colon \NN \to (0, \infty)$ satisfying $\omega(n) \to \infty$ as $n \to \infty$ such that, for every $\Pi\in\Partg$, there exists a map $\cM \colon \FreesdP \to \BPk$ that satisfies
  \begin{equation}
    \label{eq:Freesdgp}
    | \cM^{-1}(B) | \le \frac{1}{\omega(n)} \cdot \binom{e(\Pi)}{m-e(B)}
  \end{equation}
  for every $B \in \BPk$. Then
  \[
    |\Free \setminus \Grk| \ll |\Free|.
  \]
\end{prop}
\begin{proof}
  Set $C = C_{\ref{thm:aa-balanced-unfriendly}}(\delta, \gamma)$ and suppose that $m \ge C \twoDensTS$. We claim that
  \[
    \Free \setminus \Grk \subseteq 
		\biggl(\Free \setminus  \underbrace{\bigcup_{\Pi \in \Partg} \Freedp}_{=: \Freedgp}\biggr) \cup \bigcup_{\Pi \in \Partg} \FreesdP .
  \]
  Indeed, if $G \in \Freedgp \setminus \Grk$, then, on the one hand, $G \in \Freedp$ for some $\Pi \in \Partg$ but, on the other hand, $G \setminus \Pi \notin \BPk$ and hence $G \in \FreesdP$. Since Theorem~\ref{thm:aa-balanced-unfriendly} implies that
  \[
    |\Free \setminus \Freedgp| \ll |\Free|,
  \]
  it suffices if we show that our assumptions imply that
  \begin{equation}
    \label{eq:FreesdP-union}
    \sum_{\Pi \in \Partg} |\FreesdP| \ll |\Free|.
  \end{equation}
  To this end, note first that the assumption that $H$ is \emph{simple} vertex-critical implies that, for all $\Pi \in \Part$ and $B \in \BPkv$, the graph $B \cup \Pi$ is $H$-free and, consequently,
  \[
    \UPB \subseteq \GPB \subseteq \Free.
  \]
  Since the families $\UPB$ are pairwise-disjoint and
  \[
    |\UPB| \geBy{P~\ref{prop:UP}} \frac{1}{2} \cdot |\GPB| = \frac{1}{2}\binom{e(\Pi)}{m-e(B)},
  \]
  we have
  \begin{equation}
    \label{eq:Free-union}
    \begin{split}
      |\Free| & \ge \sum_{\Pi \in \Partg}  \sum_{B \in \BPkv} |\UPB| \\
      & \ge \frac{1}{2}\sum_{\Pi \in \Partg} \sum_{B \in \BPkv} \binom{e(\Pi)}{m-e(B)} \\
      & \geBy{P~\ref{prop:large-girth}} \frac{c_{\ref{prop:large-girth}}}{2} \sum_{\Pi \in \Partg}\sum_{B \in \BPk} \binom{e(\Pi)}{m-e(B)}.
    \end{split}
  \end{equation}
  Fix an arbitrary $\Pi \in \Partg$, let $\cM$ be the map satisfying~\eqref{eq:Freesdgp} for every $B \in \BPk$, and observe that
  \begin{equation}
    \label{eq:FreesdP-bound}
    |\FreesdP| = \sum_{B \in \BPk} |\cM^{-1}(B)| \le \frac{1}{\omega(n)} \sum_{B \in \BPk} \binom{e(\Pi)}{m-e(B)}.
  \end{equation}
  Summing~\eqref{eq:FreesdP-bound} over all $\Pi \in \Partg$ and substituting it into~\eqref{eq:Free-union} yields~\eqref{eq:FreesdP-union}.
\end{proof}

\subsection{Splitting into the sparse and the dense cases}

In the remainder of this paper, we will define, for some sufficiently small positive constants $\delta$ and $\gamma$ and every $\Pi \in \Partg$, a map $\cM \colon \FreesdP \to \BPk$ and show that these maps satisfy the assumptions of Proposition~\ref{prop:sufficient-1-statement}. Unfortunately, our main argument, presented in Section~\ref{sec:sparse}, will work only under the assumption that $m \le \ex(n, H) - \Omega(n^2)$; the (much easier) complementary case $m \ge \ex(n, H) - o(n^2)$ will be treated in Section~\ref{sec:dense}.

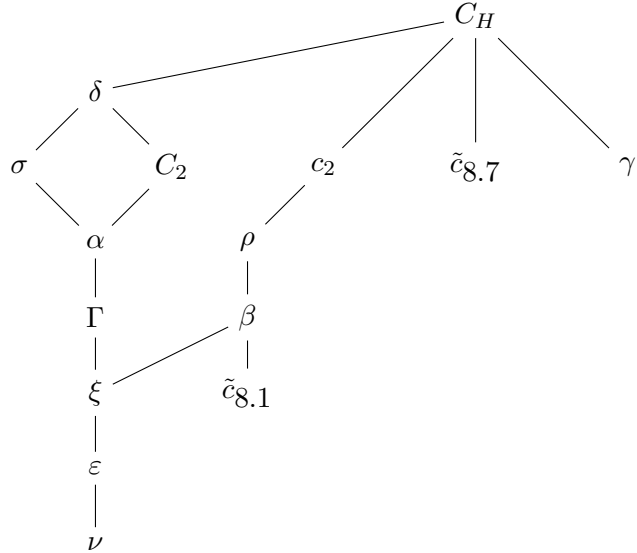
\begin{figure}[h]
  \centering
  \begin{tikzpicture}
    \node (nu) at (1,0) {$\nu$};
    \node (eps) at (1,1) {$\eps$};
    \node (xi) at (1,2) {$\xi$};
    \node (cLD) at (4,2) {$\tilc_{\ref{prop:oneStateLDProb}}$};
    \node (Gamma) at (1,3) {$\Gamma$};
    \node (beta) at (4,3) {$\beta$};
    \node (alpha) at (1,4) {$\alpha$};
    \node (rho) at (4,4) {$\rho$};
    \node (sigma) at (0,5) {$\sigma$};
    \node (C2) at (2,5) {$C_2$};
    \node (c2) at (6,5) {$c_2$};
    \node (cHDreg) at (4,5) {$\tilc_{\ref{lemma:regCase}}$};
    \node (delta) at (1,6) {$\delta$};
    \node (gamma) at (8,5) {$\gamma$};
    \node (CH) at (6,7) {$C_H$};
    \draw (nu) -- (eps) -- (xi) -- (Gamma) -- (alpha) -- (sigma) -- (delta) -- (CH)
    (alpha) -- (C2) -- (delta)
    (cLD) -- (beta) -- (rho) -- (c2) -- (CH)
    (gamma) -- (CH)
    (xi) -- (beta)
    (cHDreg) -- (CH)
		(rho) -- (delta)
		(cHDreg) -- (delta);
  \end{tikzpicture}  
  \caption{The Hasse diagram depicting dependence between the various constants in the proof}
  \label{fig:constants}
\end{figure}

In order to formally define the split between the two cases, we need to introduce several additional parameters (cf.~Figure~\ref{fig:constants}). First, let $\gamma$ be any positive constant satisfying
\begin{equation}
  \label{eq:gamma}
  \gamma \le \frac{1}{20r}.
\end{equation}
Second, let $\xi$ be a positive constant that satisfies inequalities~\eqref{eq:xi-dense} and the first inequality~\eqref{eq:xi-delta-eps-nu-gamma}, which involve absolute constants $\eps$ and $\nu$ that are defined in~\eqref{eq:eps-nu}. Third, let $\delta$ be a small positive constants that also satisfies the first inequality in~\eqref{eq:xi-delta-eps-nu-gamma} and, moreover, the inequalities
\begin{equation}
  \label{eq:delta}
  \delta \le \frac{1}{20r} \qquad \text{and} \qquad \delta \le \frac{\xi \rho \tilc_{\ref{lemma:regCase}}}{70} \cdot \min\left\{\sigma, \frac{1}{C_2}\right\},
\end{equation}
where $\tilc_{\ref{lemma:regCase}}$ is an absolute positive constant implicit in the statement of Lemma~\ref{lemma:regCase}, $\rho$ is a constant that depends on $\xi$ and on $\tilc_{\ref{lemma:regCase}}$ and is defined at the beginning of Section~\ref{sec:sparse}, and $\sigma$ and $C_2$ are constant that depend on $\xi$ and the function $(z, \alpha, \lambda) \mapsto \tau$ implicit in the statement of Lemma~\ref{lemma:d-sets} and are defined in Section~\ref{sec:high}. Finally, define
\begin{equation}
  \label{eq:CH}
  C_H = \max\left\{C_{\ref{prop:sufficient-1-statement}}(\delta, \gamma),\; \frac{1}{\beta}, \; \frac{1}{c_2 \cdot \tilc_{\ref{lemma:regCase}}} \cdot \frac{35r}{\xi}, \; 2\right\},
\end{equation}
where $C_{\ref{prop:sufficient-1-statement}}(\delta, \gamma)$ is a constant that depends on $\delta$ and $\gamma$ and is implicitly defined in the statement of Proposition~\ref{prop:sufficient-1-statement}, $\beta$ is a constant that depends on $\xi$ and on $\tilc_{\ref{prop:oneStateLDProb}}$ and is defined at the beginning of Section~\ref{sec:sparse}, and $c_2$ is a constant that depends on $\rho$ (see above) and is defined in Section~\ref{sec:high}.

Fix an arbitrary $\Pi \in \Partg$. Our definition of the map $\cM \colon \FreesdP \to \BPk$ and the arguments we will use to show that $\cM$ satisfies the assumptions of Proposition~\ref{prop:sufficient-1-statement} with $\omega(n) = 2/n$ will vary depending on whether
\begin{equation}
  \label{eq:sparse-dense}
  C_Hm_H \le m \le e(\Pi) - \xi n^2 \qquad \text{or} \qquad e(\Pi) - \xi n^2 < m \le \exnH.  
\end{equation}
Our analysis under the assumption that $m$ satisfies the first and the second pair of inequalities in~\eqref{eq:sparse-dense} will be referred to as the \emph{sparse case} and the \emph{dense case}, respectively. These two cases will be treated in Sections~\ref{sec:sparse} and~\ref{sec:dense}, respectively.

\section{The 1-statement: the sparse case}
\label{sec:sparse}

Fix a partition $\Pi \in \Partg$. In this section, we verify the assumptions of Proposition~\ref{prop:sufficient-1-statement} in the case where 
\[
  C_Hm_H \le m \le e(\Pi) - \xi n^2.
\]
In order to show that the assumptions of Proposition~\ref{prop:sufficient-1-statement} are satisfied, we will first define a natural map $\cM \colon \FreesdP \to \BPk$ by letting $\cM(G)$ be an arbitrarily chosen maximal subgraph of $G \setminus \Pi$ with maximum degree $k$. We will estimate the left-hand side of~\eqref{eq:Freesdgp} using two different arguments, depending on the distribution of edges in the monochromatic graph $G \setminus \Pi$: the \emph{low-degree case} and the \emph{high-degree case}.

Let $\TT_\Pi$ denote the family of all $T \subseteq \Pic$ that are the monochromatic subgraph of some $G \in \FreesdP$, that is,
\[
  \TT_\Pi = \big\{G \setminus \Pi : G \in \FreesdP\big\};
\]
our definitions imply that every $T\in\TT_\Pi$ satisfies $e(T) \le \delta m$ and $\Delta(T)>k$. Define further, for every $T\in\TT_\Pi$,
\[
  \FreeT = \big\{G\in\FreesdP : G\setminus\Pi = T\big\} ,
\]
and observe that
\[
  |\FreesdP| = \sum_{T\in\TT_\Pi} |\FreeT|.
\]

In order to describe the split between the low-degree and the high-degree cases, let
\begin{equation}
  \label{eq:beta}
  \beta = \min\left\{\frac{e}{\xi(r-1)},\frac{\tilc_{\ref{prop:oneStateLDProb}}}{22}\right\} \qquad\text{and}\qquad D =\floor{\beta\frac{m}{n\log n}},
\end{equation}
where $\tilc_{\ref{prop:oneStateLDProb}}$ is an absolute positive constant that is implicit in the statement of Proposition~\ref{prop:oneStateLDProb}, and choose a $\rho>0$ which satisfies
\begin{equation}
  \label{eq:rhobeta}
  \left(\frac{e}{\xi\rho}\right)^{\rho}\le e^{\beta/2} \qquad \text{and} \quad \rho \le \frac{1}{4r};
\end{equation}
it is possible to choose such $\rho$, since the left-hand side of the first inequality in \eqref{eq:rhobeta} converges to $1$ as $\rho\to 0$. 

\subsection{Decomposing the monochromatic graphs}
\label{sec:decomp-monochr-graph}

For every $T \in \TT_\Pi$, we define the following graphs and sets:
\begin{itemize}
\item
  Let $B_T$ be an arbitrarily chosen maximal subgraph of $T$ with $\Delta(B_T) = k$;  note that $B_T \in\BPk$, as defined in Section~\ref{sec:mono-graphs-small-max-deg}.
\item
  Let $U_T$ be an arbitrarily chosen maximal subgraph of $T$ that extends $B_T$ and satisfies $\Delta(U_T) \le D$.
\item
  Let $X_T$ be the set of vertices whose degrees in $U_T$ are exactly $D$.
\item
  Let $H_T$ be the set of all vertices whose degrees in $T$ are larger than $\rho m / n$; note that $|H_T| \le 2\delta n/\rho$.
\end{itemize}
Finally, for every $B \in \BPk$, let $\TTsubB$ denote the subfamily of $\TT_\Pi$ comprising all $T$ with
\[
  B_T = B, \qquad e(T) = t, \qquad e(U_T) = e(B) + \ell, \qquad \text{and} \qquad |H_T| = h.
\]
The map $\cM$ that we will supply to Proposition~\ref{prop:sufficient-1-statement} is the map defined by $\cM(G) = B_T$, where $T = G \setminus \Pi$.

\subsection{The low-degree and the high-degree cases}
\label{sec:low-high-degree}

We may now define the partition into the low-degree and the high-degree cases. Suppose that $T \in \TT_\Pi$. We place $T$ in $\TTL$ when
\begin{equation}
  \label{eq:LowDegCon}
  e(U_T \setminus B_T) \log n \ge \frac{m |H_T|}{\xi n};
\end{equation}
otherwise, we place $T$ in $\TTH$. Since $\TTL$ and $\TTH$ form a partition of $\TT_\Pi$, we have, for every $B \in \BPk$,
\begin{equation}
  \label{eq:FreeT-hi-lo}
  |\cM^{-1}(B)| = \sum_{\substack{T\in\TT_\Pi \\ B_T = B}}|\FreeT|= \sum_{\substack{T\in\TTL \\ B_T = B}} |\FreeT| + \sum_{\substack{T\in\TTH \\ B_T = B}} |\FreeT|.
\end{equation}
The low-degree and the high-degree cases are estimates of the first and the second sums in the right hand side of~\eqref{eq:FreeT-hi-lo}, respectively.

\subsection{The low-degree case -- summary}
\label{sec:low-degree-summary}

In the low-degree case, we will rely on the following upper bound on $|\FreeT|$, which is established in Section~\ref{sec:low} with the use of the Hypergeometric Janson Inequality (Lemma~\ref{lem:HJI}).

\begin{prop}
  \label{prop:oneStateLDProb}
  There exists a positive constant $\tilc$ that depends only on $H$ such that the following holds. If~$m\ge\tilC m_H$ for some $\tilC \ge 2$, then, for every $\Pi \in \Partg$, every $B \in \BPk$, all $t \le m/2$, $\ell$, and $h$, and every $T\in\TTsubB$,
  \[
    |\FreeT| \le \exp\left(-\frac{\tilc}{\beta + \tilC^{-1}} \cdot \ell \log n\right) \cdot \binom{e(\Pi)}{m-t}.
  \]
\end{prop}

This upper bound on $|\FreeT|$ will be combined with the following estimate on the size of the sum over all $T \in \TTL$, which is derived in Section~\ref{subsec:enumGraphs}.

\begin{lemma}
  \label{lem:enumeratingT}
  Suppose that $n \log n \ll m \le e(\Pi) - \xi n^2$. For every $B \in \BPk$ and all $t$, $\ell$, and $h$,
  \[
    |\TTsubB| \cdot \binom{e(\Pi)}{m-t} \le \exp\left(14\ell \log n+\frac{2mh}{\xi n}\right) \cdot \binom{e(\Pi)}{m-e(B)}.
  \]
\end{lemma}

Before we close this section, we show how these two lemmas can be used to estimate the first sum in the right-hand side of~\eqref{eq:FreeT-hi-lo}. Let $\cL$ be the family of all triples $(t, \ell, h)$ that satisfy $t \ge \ell \ge 1$ and $\ell \log n \ge mh/(\xi n)$, cf.~\eqref{eq:LowDegCon}, and observe that, for every $B \in \BPk$,
\[
  \sum_{\substack{T\in\TTL \\ B_T = B}}|\FreeT| = \sum_{(t, \ell, h) \in \cL} \; \underbrace{\sum_{T\in\TTsubB}|\FreeT|}_{X_{t,\ell,h}}.
\]
Since $m \ge C_Hm_H$, Proposition~\ref{prop:oneStateLDProb} and Lemma~\ref{lem:enumeratingT} imply that, for every $(t, \ell, h) \in \cL$,
\[
  \begin{split}
    X_{t,\ell,h} & \le |\TTsubB| \cdot \exp\left(-\frac{\tilc}{\beta + C_H^{-1}} \cdot \ell \log n\right) \cdot \binom{e(\Pi)}{m-t} \\
    & \le \exp\left(14\ell \log n+\frac{2mh}{\xi n}-\frac{\tilc}{\beta + C_H^{-1}} \cdot \ell \log n\right) \cdot \binom{e(\Pi)}{m-e(B)} \\
    & \leByRef{eq:CH} \exp\left(\left(16 - \frac{\tilc}{2\beta}\right) \cdot \ell \log n\right) \cdot \binom{e(\Pi)}{m - e(B)} \leByRef{eq:beta} n^{-6\ell} \cdot \binom{e(\Pi)}{m - e(B)}.
  \end{split}
\]
Since $\ell \ge 1$ for every $(t, \ell, h) \in \cL$, we may conclude that
\[
  \sum_{\substack{T\in\TTL \\ B_T = B}}|\FreeT| \le |\cL| \cdot n^{-6} \cdot \binom{e(\Pi)}{m-e(B)} \le \frac{1}{n} \cdot \binom{e(\Pi)}{m-e(B)}.
\]

\subsection{Enumerating the monochromatic graphs}
\label{subsec:enumGraphs}

In this short section, we enumerate graphs in $\TTsubB$, proving Lemma~\ref{lem:enumeratingT}.

\begin{proof}[{Proof of Lemma~\ref{lem:enumeratingT}}]
  We will count the number of ways to construct any $T\in\TTsubB$ in several steps. We first record the following inequality, which holds for all integers $y\le m'\le m$:
  \begin{equation}
    \label{eq:usefulEnumBound}
    \frac{\binom{e(\Pi)}{m'-y}}{\binom{e(\Pi)}{m'}} = \frac{\binom{m'}{y}}{\binom{e(\Pi)-m'+y}{y}} \le \frac{\binom{m}{y}}{\binom{\xi n^2+y}{y}} \leByRef{eq:binCoef3} \left(\frac{m}{\xi n^2}\right)^y.
  \end{equation}
  Since $B = B_T \subseteq T$ for every $T \in \TTsubB$, we only need to choose which $t - e(B)$ edges of $\Pic$ form the graph $T \setminus B$. For every $T \in \TTsubB$, let $U_T'$ be the subgraph of $U_T \setminus B$ obtained by removing all edges touching $X_T$. Since every edge of $U_T\setminus U_T'$ has at least one endpoint in $X_T$ and $\Delta(B) \le k$, we have
  \[
    \ell = e(U_T \setminus B) \ge e(U_T') +|X_T| \cdot (D-k)/2 \ge e(U_T') +|X_T| \cdot D/3.
  \]
  
 We choose the edges of $T \setminus B$ in three steps:
  \begin{enumerate}[label={(S\arabic*)}]
  \item
    \label{item:enumeratingT-1}
    We choose the edges of $U_T'$.
  \item
    \label{item:enumeratingT-2}
    We choose the edges of $T \setminus B$ that touch $X_T \setminus H_T$.
  \item
    \label{item:enumeratingT-3}
    We choose the remaining edges of $T \setminus B$; they all touch $H_T$.
  \end{enumerate}

  We count the number of ways to build a graph $T \in \TTsubB$ with $u'$, $t_X$, and $t_H$ edges chosen in steps~\ref{item:enumeratingT-1}, \ref{item:enumeratingT-2}, and~\ref{item:enumeratingT-3}, respectively. An upper bound on $|\TTsubB|$ will be obtained by summing over all choices for $u'$, $t_X$, and $t_H$. There are at most $\binom{e(\Pic)}{u'}$ ways to choose $u'$ edges of $U_T'$. Since $e(\Pic) \le n^2$, we have
  \[
    \binom{e(\Pic)}{u'} \cdot \frac{\binom{e(\Pi)}{m-e(B)-u'}}{\binom{e(\Pi)}{m-e(B)}} \leByRef{eq:usefulEnumBound}
    n^{2u'} \left(\frac{m}{\xi n^2}\right)^{u'} = \left(\frac{m}{\xi}\right)^{u'} \le m^{2u'}.
  \]
  Next, we bound the number of ways to choose the $t_X$ edges that touch $X_T\setminus H_T$. To this end, we arbitrarily order the vertices of $X_T \setminus H_T$ as $v_1, \dotsc, v_s$ and then, for each $i \in \br{s}$, we choose the edges incident to $v_i$ and not to any of $v_1, \dotsc, v_{i-1}$; denote the number of such edges by $d_i$. Since we are considering only vertices of $X_T \setminus H_T$, we have $d_i \le \rho m / n$; moreover, $d_1 + \dotsb + d_s = t_X$. Let $N_2 = N_2(t_X, s)$ denote the total number of ways to choose the $t_X$ edges when $|X_T \setminus H_T| = s$. We have
  \begin{align*}
    N_2 \cdot \frac{\binom{e(\Pi)}{m-e(B)-u'-t_X}}{\binom{e(\Pi)}{m-e(B)-u'}}
    & \le \binom{n}{s} \sum_{\substack{d_1, \dotsc, d_s \le \rho m/n \\ d_1 + \dotsb +d_s = t_X}} \prod_{i=1}^s
    \binom{n}{d_i} \cdot\frac{\binom{e(\Pi)}{m-e(B)-u'-(d_1+\dotsb+d_i)}}{\binom{e(\Pi)}{m-e(B)-u'-(d_1+\dotsb+d_{i-1})}} \\
    &\le \left(n \cdot \sum_{d=0}^{\rho m/n} \binom{n}{d} \cdot \max_{m' \le m} \frac{\binom{e(\Pi)}{m'-d}}{\binom{e(\Pi)}{m'}}\right)^s \\
    & \leByRef{eq:usefulEnumBound} \left(n + n \cdot \sum_{d=1}^{\rho m/n} \left(\frac{en}{d} \cdot \frac{m}{\xi n^2}\right)^d\right)^s.
  \end{align*}
  Since, for every positive $a$, the function $x \mapsto (ea/x)^x$ is increasing on the interval $(0,a]$ and $\rho < e/\xi$, we conclude that
  \[
    \left(N_2 \cdot \frac{\binom{e(\Pi)}{m-e(B)-u'-t_X}}{\binom{e(\Pi)}{m-e(B)-u'}}\right)^{1/s}
    \le n + \rho m \cdot \left(\frac{e}{\xi \rho}\right)^{\rho m /n} \le \left(\frac{e}{\xi\rho}\right)^{2\rho m /n} \leByRef{eq:rhobeta} e^{\beta m/n} \le m^{D}.
  \]
  Finally, let $N_3 = N_3(t_X, h)$ denote the number of ways to choose the remaining $t_H$ edges of $T \setminus B$. Recalling that $e(B) + u' + t_X + t_H = t$ and arguing similarly as above, we obtain
  \[
    N_3 \cdot \frac{\binom{e(\Pi)}{m-t}}{\binom{e(\Pi)}{m-e(B)-u'-t_X}} \le \left(n + n \cdot \sum_{d=1}^{n-1} \left(\frac{en}{d} \cdot \frac{m}{\xi n^2}\right)^d\right)^h.
  \]
  Using again the fact that $(ea/x)^x \le e^a$ for all $x \in (0, \infty)$, we conclude that
  \[
    \left(N_3 \cdot \frac{\binom{e(\Pi)}{m-t}}{\binom{e(\Pi)}{m-e(B)-u'-t_X}}\right)^{1/h}
    \le n+n^2 \cdot \exp\left(\frac{m}{\xi n}\right) \le \exp\left(\frac{2m}{\xi n}\right).
  \]
  Combining the above bounds, we obtain
  \begin{align*}
    |\TTsubB| \cdot \frac{\binom{e(\Pi)}{m-t}}{\binom{e(\Pi)}{m-e(B)}}
    & \le \sum_{\substack{u',t_X,t_H,s \\ u'+t_X+t_H = t \\ u' + sD/3 \le \ell}} m^{2u'} \cdot m^{Ds} \cdot \exp\left(\frac{2mh}{\xi n}\right) \\
    & \le nm^3 \cdot m^{3\ell} \cdot \exp\left(\frac{2mh}{\xi n}\right)
      \le \exp\left(14\ell \log n+\frac{2mh}{\xi n}\right),
  \end{align*}
  where the final inequality follows as $nm^3 \le m^4 \le m^{4\ell}$ and $m \le n^2$.
\end{proof}

\subsection {The low-degree case}
\label{sec:low}

In this section, we prove Proposition~\ref{prop:oneStateLDProb}, that is, for a given $T \in \TTsubB$, we give an upper bound on the number of graphs in $\FreeT$ in terms of $t$ and $\ell$. To this end, fix an arbitrary critical star $\Siz$ in $H$ that satisfies
\[
  \eta_{i_0}(H) = \eta(H) \qquad \text{and}  \qquad \zeta_{i_0}(H)=\zeta(H),
\]
where $\eta(H)$ and $\zeta(H)$ are the quantities defined above~\eqref{eq:mH-def}. Fix some $T \in \TT_\Pi$. For every injection $\varphi\colon V(H) \to \br{n}$, we let $\Sphi = \varphi(\Siz)$ and $\Kphi=\varphi(H\setminus \Siz)$ be the labelled graphs that are the images of $\Siz$ and $H \setminus \Siz$ via the embedding $\varphi$. Define
\[
\Phi_T = \{\varphi : \Sphi\subseteq T\text{ and } \Kphi\subseteq\Pi\};
\]
in other words, $\Phi_T$ comprises all those embeddings of $H$ into $\Pi\cup T$ that embed $\Siz$ into $T$ and map the remaining edges of $H$ to $\Pi$. Since $T \subseteq G$ for every $G \in \FreeT$, the graph $G \cap \Pi$ does not contain any of the~$K_\varphi$ with $\varphi \in \Phi_T$. In particular, letting $G'$ be a uniformly chosen random subgraph of $\Pi$ with $m - t$ edges, we have
\begin{equation}
  \label{eq:FreeT-Pr}
  |\FreeT| \le \Pr\big(K_{\varphi} \nsubseteq G' \text{ for each } \varphi \in \Phi_T\big) \cdot \binom{e(\Pi)}{m-t}.  
\end{equation}

Proposition~\ref{prop:oneStateLDProb} is derived from~\eqref{eq:FreeT-Pr} and the Hypergeometric Janson Inequality. In order to get a strong bound on the probability in the right-hand side of~\eqref{eq:FreeT-Pr}, we will carefully construct a sub-family of $\Phi_T$ that satisfies some `nice' properties and apply Janson's inequality with $\Phi_T$ replaced by this sub-family.

\begin{lemma}
  \label{lem:goodStars}
  Suppose that $\Pi = \partition$ and let $T \in \TTsubB$. There are an $i \in \br{r}$ and a family $\SS$ of edge-disjoint copies of $K_{1,k+1}$ in $T[V_i]$ that satisfy the following properties for some positive constants $c_1$ and $C_1$ that depend only on $r$ and $k$:
  \begin{enumerate}[label=(GS\arabic*)]
  \item
    \label{item:Ssize}
    We have $c_1\ell \le|\SS| \le \ell$.
  \item
    \label{item:oneVtxBound}
    For every $v\in\br{n}$, we have $|\{S\in\SS : v\in V(S)\}| \le D =\floor{\beta \frac{m}{n\log n}}$.
  \item
    \label{item:twoVtxBound}
    For every two different vertices $v,u\in\br{n}$, let $A(u,v)$ be the set of all pairs of stars $S,S'\in\SS$, each containing both $u$ and $v$ as leaves. Then, $\sum_{u,v}|A(u,v)| \le C_1\ell$.
  \end{enumerate}
\end{lemma}

We will first derive Proposition~\ref{prop:oneStateLDProb} from Lemma~\ref{lem:goodStars} and and then prove the lemma.

\begin{proof}[{Proof of Proposition~\ref{prop:oneStateLDProb}}]
  Suppose that $\Pi = \{V_1, \dotsc, V_r\}$ and let $T \in \TTsubB$ be a graph with at most $m/2$ edges. Let $i$, $\SS$, $c_1$, and $C_1$ be the colour class, the family of stars, and the two constants from the statement of Lemma~\ref{lem:goodStars}, respectively. Fix an arbitrary colouring $\psi \colon V(H) \to \br{r}$ that leaves only the edges of $\Siz$ monochromatic and such that the vertices of $\Siz$ are coloured~$i$; such a colouring exists because $\Siz$ is a critical star of $H$. For every $j \in \br{r}$, randomly choose an equipartition $\{V_{j,w}\}_{w \in V(H)}$ of $V_j$ into $v_H$ parts. We let $\PhipT$ be the family of all embeddings $\varphi \in \Phi_T$ that satisfy
  \[
    S_\varphi \in \SS \qquad \text{and} \qquad \varphi(w)\in V_{\psi(w),w} \text{ for every } w \in V(H).
  \]
  Let $n' = \min\{|V| : V \in \Pi\} \ge n/(2r)$. Since there are at least $|\SS| \cdot (n'-v_H)^{v_H - (k+2)}$ embeddings $\varphi \in \Phi_T$ such that $S_\varphi \in \SS$ and $\varphi(w) \in V_{\psi(w)}$ for every $w \in V(H)$ and, for each such $\varphi$, the probability that $\varphi \in \Phi_T'$ is at least $v_H^{-v_H}$, there is a positive constant $c$ that depends only on $H$ such that
  \[
    \Ex\big[|\PhipT|\big] \ge c \ell n^{v_H-k-2}.
  \]
   
  We now fix some partitions $\{V_{j,w}\}_{w \in V(H)}$ for which $|\PhipT|$ is at least as large as its expectation and we let
  \[
    \SS' = \big\{ \Sphi : \varphi \in \PhipT \big\}
    \qquad \text{and} \qquad
    \cKp = \big\{ \Kphi : \varphi \in \PhipT \big\}.
  \]
  We claim that $\Kphi \neq \Kphip$ for each pair of distinct $\varphi,\varphi'\in\PhipT$. To see this, note first that, since $\Siz$ is a critical star, every vertex in $V(\Siz)$ must have a neighbour in $\psi(j)^{-1}$, for each $j \in \br{r} \setminus \{i\}$. Since $H$ has no isolated vertices, this means that each vertex of $H$ is incident to an edge of $H \setminus \Siz$. Therefore, since each $\varphi \in \PhipT$ maps every $w \in V(H)$ to its dedicated set $V_{\psi(w),w}$, one can recover $\varphi$ from the graph $\Kphi$. This means, in particular, that
  \begin{equation}
    \label{eq:KTsize}
    |\cKp| = |\PhipT| \ge c \ell n^{v_H-k-2}.
  \end{equation}
  
  Suppose that $m \ge \tilC m_H$ for some $\tilC \ge 2$ and let $G'$ be a uniformly chosen random subgraph of $\Pi$ with $m - t$ edges. The definition of $\cKp$ and~\eqref{eq:FreeT-Pr} imply that
  \[
    |\FreeT| \le \Pr(K \nsubseteq G'\text{ for every }K\in \cKp) \cdot \binom{e(\Pi)}{m-t}.
  \]
  We shall bound this probability from above using the Hypergeometric Janson Inequality. To this end, let $p=\frac{m-e(T)}{e(\Pi)}$ and note that
  \begin{equation}
    \label{eq:p-bounds-Janson}
    \frac{m}{2n^2} \le p \le \frac{5m}{n^2},
  \end{equation}
  where the first inequality holds because $e(T) \le m/2$ and the last inequality follows from part~\ref{item:Partg-lower} of Proposition~\ref{prop:bal-unbal-part-size}, as $\Pi \in \Partg$ and $\gamma \le \frac{1}{20r}$. For any $K, K' \in \cKp$, we write $K \sim K'$ if $K$ and $K'$ share an edge but $K \neq K'$. Let $\mu$ and $\Delta$ be the quantities defined in the statement of the Hypergeometric Janson Inequality (Lemma~\ref{lem:HJI}), that is,
  \[
    \mu = \sum_{K\in\cKp}p^{e_K} \qquad \text{and} \qquad \Delta = \sum_{\substack{K, K' \in \cKp \\ K \sim K'}}p^{e_{K \cup K'}}.
  \]
  Since $e_K = e(H \setminus \Siz) = e_H - k - 1$ for every $K \in \cKp$, we have, by~\eqref{eq:KTsize},
  \begin{equation}
    \label{eq:mu-lower}
    \mu = |\cKp| \cdot p^{e_H - k - 1} \ge c\ell n^{v_H-k-2}p^{e_H-k-1}.
  \end{equation}
  
  We now bound $\Delta$ from above. In order to do this, we shall classify the pairs $(K, K') \in (\cKp)^2$ with $K \sim K'$ according to their intersection. To this end, for each $J \subseteq V(\Siz)$, define $\SSpJ$ to be the set of all pairs of stars from $\SSp$ which agree exactly on (the image of) $J$, that is,
  \[
    \SSpJ = \big\{(\Sphi,\Sphip)\in\SSp \times\SSp : \Sphi\cap\Sphip = \varphi(H[J]) = \varphi'(H[J])\big\}.
  \]
  Further, given $J \subseteq V(\Siz)$ and $I\subseteq V(H) \setminus V(\Siz)$, let $\FIJ = H[I\cup J] \setminus \Siz$, that is, $\FIJ$ is a graph with vertex set $I \cup J$ that comprises the edges of $H \setminus \Siz$ with both endpoints in $I \cup J$. (Let us note here that $\FIJ$ may have some isolated vertices.) Finally, for $J \subseteq V(\Siz)$, $I \subseteq V(H) \setminus V(\Siz)$, and $S, S' \in \SSpJ$, define $\cKIJS$ to be the set of all pairs $K, K' \in \cKp$ which extend the stars $S,S'$, respectively, and agree exactly on (the image of) $I \cup J$. In other words,
  \[
    \cKIJS = \big\{ (\Kphi, \Kphi' )\in (\cKp)^2 : \Kphi\cap\Kphip = \varphi(\FIJ) = \varphi'(\FIJ) , \Sphi= S, \Sphip=S'\big\}.
  \]
  For brevity, set
  \[
    v' = |V(H) \setminus V(\Siz)| = v_H - k - 2
    \qquad \text{and} \qquad
    e' = e(H \setminus \Siz) = e_H - k - 1.
  \]
  These definitions were made in such a way that
  \begin{equation}
    \label{eq:Delta-upp}
    \begin{split}
      \Delta
      & = \sum_{J\subseteq V(\Siz)} \sum_{(S,S')\in\SSpJ} \sum_{\substack{I\subseteq V(H)\setminus V(\Siz) \\ e(\FIJ)>0}} \sum_{(K,K')\in\cKIJS}p^{2e'-e(\FIJ)} \\
      & \le \sum_{J \subseteq V(\Siz)} \sum_{(S,S') \in \SSpJ} \sum_{\substack{I \subseteq V(H) \setminus V(\Siz) \\ e(\FIJ) > 0}}  n^{2v'-|I|}p^{2e'-e(\FIJ)}.
    \end{split}
  \end{equation}
  
  Denote by $\Delta_0$, $\Delta_1$, and $\Delta_2$ the contributions to the sum in the right-hand side of~\eqref{eq:Delta-upp} corresponding to $J = \emptyset$, $|J| = 1$, and $|J| \ge 2$, respectively, so that $\Delta \le \Delta_0 +\Delta_1 +\Delta_2$. Since $F_{I,\emptyset} = H[I] \subseteq H$, we have
  \begin{align*}
    \frac{\Delta_0}{n^{2v'}p^{2e'}}
    & = \sum_{(S,S')\in\SSp(\emptyset)} \sum_{\substack{I\subseteq V(H)\setminus V(\Siz)\\ e(F_{I,\emptyset})>0}}  \frac{1}{n^{|I|} p^{e(F_{I,\emptyset})}} \le |\SSp(\emptyset)| \cdot \sum_{\emptyset \ne F \subseteq H} \frac{1}{n^{v_F}p^{e_F}} \\
    & \le |\SS|^2 \cdot \frac{2^{e_H}}{\min_{\emptyset \neq F \subseteq H} n^{v_F} p^{e_F}} \leBy{L.~\ref{lemma:m2H}}  |\SS|^2\cdot \frac{2^{e_H}}{n^2p} \le \ell^2 \cdot \frac{2^{e_H}}{n^2p},
  \end{align*}
  where the last inequality follows from~\ref{item:Ssize} in Lemma~\ref{lem:goodStars}.  Further, as $v_{\FIJ} = |I| + |J|$,
  \begin{align*}
    \frac{\Delta_1}{n^{2v'}p^{2e'}}
    & = \sum_{\substack{J\subseteq V(\Siz)\\|J|=1}} \sum_{(S,S')\in\SSpJ} \sum_{\substack{I\subseteq V(H)\setminus V(\Siz)\\ e(\FIJ)>0}}  \frac{n}{n^{v_{\FIJ}}p^{e_{\FIJ}}} \\
    & \le \sum_{S \in \SS} \sum_{v \in V(S)} \sum_{\substack{S' \in \SS \\ v \in V(S')}} \sum_{\emptyset \neq F \subseteq H} \frac{n}{n^{v_F} p^{e_F}} \\
    & \le |\SS| \cdot (k+2) \cdot \max_{v} |\{S' \in \SS : v \in V(S')\}| \cdot \frac{2^{e_H} \cdot n}{\min_{\emptyset \neq F \subseteq H} n^{v_F}p^{e_F}} \\
    & \le \ell \cdot (k+2) \cdot D \cdot \frac{2^{e_H}}{np},
  \end{align*}
  where the last inequality follows from \ref{item:Ssize} and \ref{item:oneVtxBound} in Lemma~\ref{lem:goodStars} and from Lemma~\ref{lemma:m2H}. Finally,
  \begin{equation}
    \label{eq:Delta-ge2-upp}
    \begin{split}
      \frac{\Delta_2}{n^{2v'}p^{2e'}}
      & = \sum_{\substack{J\subseteq V(\Siz)\\|J|\ge2}} \sum_{(S,S')\in\SSpJ} \sum_{\substack{I\subseteq V(H)\setminus V(\Siz)\\ e(\FIJ)>0}}  \frac{n^{|J|}}{n^{v_{\FIJ}}p^{e_{\FIJ}}} \\
      & \le \sum_{\substack{u, v \in V_i \\ u \neq v}} \sum_{\substack{S, S' \in \SSp \\ u,v\in V(S) \cap V(S')}} \sum_{\substack{J \subseteq V(\Siz) \\ |J| \ge 2}} \sum_{\substack{I \subseteq V(H) \setminus V(\Siz) \\ e(\FIJ) > 0}} \frac{n^{|J|}}{n^{v_{\FIJ}}p^{e_{\FIJ}}} \\
      & \le C_1 \ell \cdot 2^{v_H} \cdot \max\left\{\frac{n^{|V(F) \cap V(\Siz)|}}{n^{v_F} p^{e_F}} : \emptyset \neq F \subseteq H \setminus \Siz\right\},
    \end{split}
  \end{equation}
  where the last inequality follows from~\ref{item:twoVtxBound} in Lemma~\ref{lem:goodStars} (since the stars in $\SS \supseteq \SSp$ are edge-disjoint, two different $S, S' \in \SSp$ that intersect in more than one vertex have to intersect only in leaf vertices). In order to bound the maximum in the right-hand side of \eqref{eq:Delta-ge2-upp}, given an arbitrary nonempty $F \subseteq H \setminus \Siz$, we let $F' =F \cup \Siz$, so that
  \begin{equation}
    \label{eq:Delta-ge2-max-upp}
    \frac{n^{|V(F) \cap V(\Siz)|}}{n^{v_F} p^{e_F}} = n^{-v_{F'}+k+2}p^{-e_{F'}+k+1}.
  \end{equation}
  
  \begin{claim}
    \label{claim:criticality}
    For every $F'$ satisfying $\Siz \subsetneq F' \subseteq H$, we have
    \[
      n^{-v_{F'}+k+2}p^{-e_{F'}+k+1}\le \frac{2}{\tilC\log n}.
    \]
  \end{claim}
  \begin{proof}
    Since $e_{F'} > e_{\Siz} = k+1$, we have
    \[
      \begin{split}
        n^{-v_{F'}+k+2}p^{-e_{F'}+k+1}
        & \leByRef{eq:p-bounds-Janson} n^{-v_{F'}+k+2} \cdot \left(\frac{\tilC}{2}\cdot\frac{m_H}{n^2}\right)^{-e_{F'}+k+1} \\
        & \le \frac{2}{\tilC} \cdot n^{-v_{F'}+k+2} \cdot \left(\frac{m_H}{n^2}\right)^{-e_{F'}+k+1} \\
        & = \frac{2}{\tilC} \cdot \left(n^{\frac{1}{d_{k+2}(F')}} \cdot \frac{m_H}{n^2}\right)^{-e_{F'}+k+1}.
      \end{split}
    \]
    Regardless of which case holds true in the definition of $m_H$ given in~\eqref{eq:mH-def}, we have
    \[
      \frac{m_H}{n^2} \le n^{-\frac{1}{\eta(H)}} (\log n)^{\frac{1}{\zeta(H)-k-1}} = n^{-\frac{1}{\eta_{i_0}(H)}}(\log n)^{\frac{1}{\zeta_{i_0}(H)-k-1}}
    \]
    and, consequently,
    \[
      n^{-v_{F'}+k+2}p^{-e_{F'}+k+1} \le \frac{2}{\tilC} \cdot n^{\left(\frac{1}{d_{k+2}(F')} - \frac{1}{\eta_{i_0}(H)}\right)(-e_{F'}+k+1)} \cdot (\log{n})^{-\frac{e_{F'}-k-1}{\zeta_{i_0}(H) -k-1}}.
    \]
    The claimed upper bound follows since $d_{k+2}(F') \le \eta_{i_0}(H)$ and $e_{F'} \ge \zeta_{i_0}(H)$ whenever $d_{k+2}(F') = \eta_{i_0}(H)$.
  \end{proof}
  
  Substituting~\eqref{eq:Delta-ge2-max-upp} into~\eqref{eq:Delta-ge2-upp} and invoking Claim~\ref{claim:criticality} yields
  \[
    \Delta_2 \le \frac{C_1 \ell \cdot 2^{v_H+1}}{\tilC \log n} \cdot n^{2v'}p^{2e'}.
  \]
  Recalling~\eqref{eq:mu-lower} and the definitions of $v'$ and $e'$, we thus obtain
  \[
    \frac{\Delta}{\mu^2} \le \frac{\Delta_0 + \Delta_1 + \Delta_2}{\mu^2} \le \frac{1}{c^2\ell} \cdot \left(\frac{2^{e_H}\ell}{n^2p} + \frac{2^{e_H}(k+2)D}{np} + \frac{C_1 2^{v_H+1}}{\tilC \log n}\right).
  \]
  Since $\ell \le Dn$, or otherwise $\TTsubB$ is empty (see Section~\ref{sec:decomp-monochr-graph}), and
  \[
    \frac{D}{np} \le \frac{\beta m}{n^2p \log n} \leByRef{eq:p-bounds-Janson} \frac{2\beta}{\log n},
  \]
  we conclude that
  \begin{equation}
    \label{eq:Delta-mu-2}
    \frac{\Delta}{\mu^2} \le \frac{C'(\beta + \tilC^{-1})}{\ell\log n} ,
  \end{equation}
  where $C'$ is some constant that depends only on $H$. On the other hand, \eqref{eq:mu-lower} and Claim~\ref{claim:criticality} with $F' = H$ imply that
  \begin{equation}
    \label{eq:mu-lower-final}
    \mu \ge\frac{c\tilC\ell\log n}{2}.
  \end{equation}
  
  Finally, we invoke Lemma~\ref{lem:HJI} with $q = \frac{\mu}{\mu + \Delta} \le 1$ to conclude that
  \[
    \begin{split}
      \frac{|\FreeT|}{\binom{e(\Pi)}{m-t}}
      & \le \Pr(K \nsubseteq G'\text{ for every }K\in \cKp) \le \exp\left(-\frac{\mu^2}{\mu+\Delta} + \frac{\mu^2\Delta}{2(\mu+\Delta)^2}\right) \\
      & \le \exp\left(-\frac{\mu^2}{2(\mu + \Delta)}\right) \le \exp\left(-\min\left\{\frac{\mu}{4}, \frac{\mu^2}{4\Delta}\right\}\right).
    \end{split}
  \]
  Substituting inequalities~\eqref{eq:Delta-mu-2} and~\eqref{eq:mu-lower-final} into this bound, we obtain the assertion of the proposition with $\tilc = \min\{1/(4C'), c/8\}$.
\end{proof}

\begin{proof}[{Proof of Lemma~\ref{lem:goodStars}}]
  Suppose that $\Pi = \partition$ and let $T \in \TTsubB$ for some $B \in \BPk$. Recall from Section~\ref{sec:decomp-monochr-graph} that $U_T$ is a canonically chosen maximal subgraph of $T$ that extends $B$ and satisfies $\Delta(U_T) \le D$.
  \begin{claim}
    \label{claim:orientation}
    There are $U' \subseteq U_T$ and an orientation $\Tz$ of a subgraph of $U'$ that satisfy
    \begin{enumerate}[label=(\roman*)]
    \item\label{Ezer:Gpp}
      We have $B \subseteq U'$ and $e(U' \setminus B) \ge \ell/2$.
    \item\label{Ezer:NeighboursBound}
      For every $(u,v)\in \Tz$, we have $\deg_{U'}(u) \le \max\{\deg_{U'}(v),4(k+1)\}$.
    \item\label{Ezer:DegBound}
      For every $v\in \br{n}$, either $\deg_{\Tz}^-(v)=0$ or $\deg_{\Tz}^-(v)\ge \max\{\deg_{U'}(v)/4, k+1\}$.
    \item\label{Ezer:EdgesBound}
      We have $e(\Tz) \ge \ell/(8k+8)$.
    \end{enumerate} 
  \end{claim}
  \begin{proof}
    Let $Q=\{v : \deg_{U_T}(v)\ge 4(k+1)\}$. We split the proof into two cases, depending on how many edges of $U_T \setminus B$ have an endpoint in $Q$.
    
    \medskip
    \noindent
    \textit{Case 1.} Fewer than half  the edges of $U_T \setminus B$ touch $Q$.
    
    \medskip
    \noindent
    Let $U'$ be the graph obtained from $U_T$ by removing all edges of $U_T \setminus B$ that touch $Q$. As $\ell = e(U_T \setminus B)$, the graph $U'$ satisfies~\ref{Ezer:Gpp}; moreover, as $\Delta(B) \le k$, then $\Delta(U') < 4(k+1)$. Let $W = \{w \in U' : \deg_{U'}(w) \ge k+1\}$, let $W'$ be a largest $U'$-independent subset of $W$, and let
    \[
      \Tz = \big\{(u,v) : \{u, v\} \in U' \text{ and } v \in W'\}.
    \]
    Since $\Delta(U') < 4(k+1)$, property~\ref{Ezer:NeighboursBound} clearly holds. To see that~\ref{Ezer:DegBound} holds, choose an arbitrary $v \in \br{n}$ and note that $\deg_{\Tz}^-(v)=0$ if $v \notin W'$; if $v \in W' \subseteq W$, then
    \[
      \deg_{\Tz}^-(v) =\deg_{U'}(v) \ge k+1 = \max\{\deg_{U'}(v)/4, k+1\}.
    \]
    Finally, we argue that~\ref{Ezer:EdgesBound} holds as well. Since $B$ is a maximal subgraph of $U_T$ with maximum degree at most $k$ and $U' \supseteq B$, every edge of $U' \setminus B$ must have an endpoint with degree larger than $k$. Therefore, by~\ref{Ezer:Gpp},
    \[
      \ell/2 \le e(U' \setminus B) \le \sum_{w \in W} \deg_{U'}(w).
    \]
    As every vertex in $W \setminus W'$ has a $U'$-neighbour in $W'$ (since $W'$ is a maximal $U'$-independent subset of $W$), we further have
    \[
      \sum_{w \in W \setminus W'} \deg_{U'}(w) \le \sum_{v \in W'} \sum_{w \in N_{U'}(v)} \deg_{U'}(w) \le \sum_{v \in W'} \deg_{U'}(v) \cdot \Delta(U')
    \]
    Recalling that $\Delta(U') \le 4k+3$, that $(u,w) \in \Tz$ for every $\{u,w\} \in U'$ such that $w \in W'$, and that $W'$ is an independent set in $U'$, we conclude that
    \[
      \ell/2 \le (4k+4) \sum_{w \in W'} \deg_{U'}(w) = (4k+4) e(\Tz).
    \]
    
    \medskip
    \noindent
    \textit{Case 2.} At least half the edges of $U_T \setminus B$ touch $Q$.
    
    \medskip
    \noindent     
    In this case we just take $U'=U_T$, so that \ref{Ezer:Gpp} clearly holds. We first let $\Tpd$ be an arbitrary orientation of $U'$ such that $\deg_{U'}(u) \le \deg_{U'}(v)$ for all $(u, v) \in \Tpd$. We then obtain $\Tz$ from $\Tpd$ by removing all edges directed to a vertex $v$ that satisfies $\deg^-_{\Tpd}(v)<\max\{k+1, \deg_{U'}(v)/4\}$. The construction of $\Tz$ guarantees that both~\ref{Ezer:NeighboursBound} and~\ref{Ezer:DegBound} are satisfied. Since every edge of $U'$ between $Q$ and $Q^c$ is directed (in $\Tpd$) towards its $Q$-endpoint, we have
    \[
      \sum_{v\in Q} \deg_{\Tpd}^-(v) \ge \frac{1}{2}\sum_{v\in Q}\deg_{U'}(v).
    \]
    Consequently,
    \begin{align*}
      e(\Tz)
      &\ge \sum_{v\in Q}\deg_{\Tz}^-(v) = \sum_{v\in Q} \deg_{\Tpd}^-(v) - \sum_{\substack{v\in Q \\ \deg_{\Tz}^-(v)=0}} \deg_{\Tpd}^-(v)  \\
      &\ge \frac{1}{2}\sum_{v\in Q}\deg_{U'}(v)  - \sum_{v\in Q} \max\{k+1, \deg_{U'}(v)/4\} = \frac{1}{4}\sum_{v\in Q}\deg_{U'}(v),
    \end{align*}
    since $\deg_{U'}(v) \ge 4(k+1)$ for every $v \in Q$. Finally, as at least half the edges of $U' \setminus B$ touch $Q$, we have $\sum_{v \in Q} \deg_{U'}(v) \ge \ell/2$ and we may conclude that $e(\Tz) \ge \ell/8$.
  \end{proof}
  
  Let $U' \subseteq U_T$ and an orientation $\Tz$ of a subgraph of $U'$ be as in Claim~\ref{claim:orientation}. For each vertex $v$, denote $\dd_v = \deg^-_{\Tz}(v)$ and let $u_1^v, \dotsc, u_{\dd_v}^v$ be a uniformly chosen random ordering of the set of the in-neighbours of $v$ in $\Tz$. Given $v \in \br{n}$ and $A \subseteq \br{n} \setminus \{v\}$, denote by $S_v(A)$ the $|A|$-star centred at $v$ whose leaves are all elements of $A$. Define
  \[
    \SSp =\big\{ S_v(\{u_i^v,\dots,u_{i+k}^v\}) : v\in \br{n}, \, i\in \br{\dd_v-k}, \text{ and } (k+1)|(i-1)\big\}.
  \]
  In other words, $\SSp$ is a (random) collection of $K_{1,k+1}$s in $U'$ created by taking, for every vertex $v$ with positive in-degree in $\Tz$, the $\lfloor \dd_v/(k+1) \rfloor$ stars centred at $v$ whose leaves are $v$'s first $k+1$ in-neighbours (in the random ordering defined above), $v$'s next $k+1$ in-neighbours, etc. By construction, the stars in $\SSp$ are edge-disjoint and $|\SSp| \le e(U' \setminus B) \le \ell$, as each star must contain an edge of $U \setminus B$ (since $\Delta(B) \le k$). On the other hand, since $\dd_v \ge k+1$ for every $v$ such that $\dd_v > 0$,
  \[
    |\SSp| = \sum_{v} \left\lfloor \frac{\dd_v}{k+1} \right\rfloor \ge \sum_{v} \frac{\dd_v}{2(k+1)} = \frac{e(\Tz)}{2(k+1)} \ge \frac{\ell}{16(k+1)^2},
  \]
  by~\ref{Ezer:EdgesBound} in Claim~\ref{claim:orientation}. Finally, since, for every $S \in \SSp$, there is an index $i_S \in \br{r}$ such that $S \subseteq U[V_{i_S}]$, by the pigeonhole principle, there must be an $i \in \br{r}$ such that the set
  \[
    \SS = \{S \in \SSp : S \subseteq T[V_i]\}
  \]
  has size at least $|\SSp|/r$. This family satisfies~\ref{item:Ssize} with $c_1 = \big(16(k+1)^2r\big)^{-1}$. To see that~\ref{item:oneVtxBound} holds as well, recall that the stars in $\SS$ are edge-disjoint, contained in $U'$, and $\Delta(U') \le D$.
  
  In the remainder of the proof we show that, with nonzero probability, our collection~$\SS$ satisfies also~\ref{item:twoVtxBound}. To this end, recall that
  \[
    A(u,v) = \big\{(S,S')\in(\SSp)^2 : \text{$u$  and $v$ are leaves of both $S$ and $S'$}\big\}.
  \]
  Since $\SS \subseteq \SS'$, it will suffice to show that, with nonzero probability,
  \begin{equation}
    \label{eq:Auv-sum-bound}
    \sum_{\substack{u,v\in \br{n}\\ u\ne v}} |A(u,v)| \le C_1\ell.
  \end{equation}
  for some $C_1$ that depends only on $r$ and $k$. For each pair of distinct $u,v \in \br{n}$, define
  \begin{align*}
    \Suv &= \{S\in\SSp : \text{$u$ and $v$ are leaves of }S\}, \\
    \Duv&=\{w\in \br{n} : u, v \in N_{\Tz}^-(w) \text{ and }  \dd_w\ge k+1\}.
  \end{align*}
  Since the stars in $\SSp$ are edge-disjoint, for every $w \in \Duv$, there is at most one $S \in \Suv$ whose $w$ is the centre. Moreover, if $\Suv$ contains such a star, then both $u$ and $v$ must fall into one of the $\lfloor \dd_w / (k+1) \rfloor$ intervals of length $k+1$ in the random ordering $u_1^w, \dotsc, u_{\dd_w}^w$ of~$N_{\Tz}^-(w)$. In particular, for every $w \in \Duv$,
  \begin{equation}\label{eq:starPrBound}
    \Pr\big(\text{$\Suv$ contains a star centred at $w$}\big) \le \frac{k}{\dd_w-1} \le \frac{k+1}{\dd_w},
  \end{equation}
  as $\dd_w \ge k+1$. Moreover, if $w \in \Duv$, then $(u,w) \in \Tz$ and hence, by~\ref{Ezer:NeighboursBound} and~\ref{Ezer:DegBound} in Claim~\ref{claim:orientation},
  \begin{equation}\label{eq:DuvBound}
    \dd_w \ge \frac{\max\{\deg_{U'}(w), 4(k+1)\}}{4} \ge \frac{\deg_{U'}(u)}{4} \ge \frac{|\Duv|}{4}.
  \end{equation}
  We conclude that
  \begin{align*}
    \sum_{\substack{u,v\in \br{n} \\u\ne v}} \Ex\big[|\Suv|\big]
    & = \Ex\big[|\SSp|\big] + \sum_{u,v} \sum_{\substack{w_1,w_2 \in \Duv \\ w_1 \neq w_2}} \prod_{i=1}^2 \Pr(\text{$S_{u,v}$ contains a star centred at $w_i$}) \\
    & \leByRef{eq:starPrBound} \ell + \sum_{u, v} \left(\sum_{\substack{w \in \Duv}} \frac{k+1}{\dd_w}\right)^2
      \leByRef{eq:DuvBound} \ell + \sum_{u,v} \sum_{w \in \Duv} \frac{4(k+1)^2}{\dd_w} \\
    & \le \ell + \sum_{w : \dd_w \ge k+1}\frac{4(k+1)^2}{\dd_w} \cdot \binom{\dd_w}{2} \le \ell + 2(k+1)^2\sum_{w : \dd_w \ge k+1}\dd_w.
  \end{align*}
  Finally, since $\dd_w \ge k+1$ implies that
  \[
    \dd_w \le \deg_{U'}(w) \le \deg_{U' \setminus B}(w) + k \le (k+1) \deg_{U' \setminus B}(w),
  \]
  we have
  \[
    \begin{split}
      \sum_{\substack{u,v \in \br{n} \\ u \neq v}} \Ex\big[|A(u,v)|\big] & = \sum_{u,v} \Ex\big[|\Suv|\big] \le \ell + 2(k+1)^3 \sum_{w} \deg_{U' \setminus B}(w) \\
      & = \ell + 2(k+1)^3 \cdot 2e(U' \setminus B) \le \big(4(k+1)^3+1\big)\ell.
    \end{split}
  \]
  In particular, taking $C_1 = 4(k+1)^3+1$, inequality \eqref{eq:Auv-sum-bound} must hold with nonzero probability.
\end{proof}

\subsection{The high-degree case -- introduction}
\label{sec:high}
 
Recall from Section~\ref{sec:decomp-monochr-graph} that, for $T\in\TT_\Pi$, we defined subgraphs $B_T$ and $U_T$ satisfying $B_T \subseteq U_T \subseteq T$ and we denoted by $H_T$ the set of all vertices of $T$ whose degree is larger than $\rho m/n$. Then, $\TTH$ was the family of all $T\in\TT_\Pi$ that satisfy (cf.~\eqref{eq:LowDegCon})
\begin{equation}\label{eq:HighDegCon}
  e(U_T \setminus B_T) \log n < \frac{m |H_T|}{\xi n},
\end{equation}
Our argument in the high-degree case will analyse the distribution of edges incident to a subset of the set $H_T$ of high-degree vertices that has convenient properties specified by our next lemma.

\begin{lemma}
  \label{lemma:YT}
  Suppose that $\Pi = \partition$. For every $T \in \TT_\Pi$, there exist $i \in \br{r}$ and $Y \subseteq V_i$ with $|Y| \ge |H_T| / (2r)$ such that, for every $v \in Y$,
  \[
    \deg_{T \setminus U_T}(v, V_i \setminus Y) \ge \frac{\rho m}{3n}.
  \]
\end{lemma}
\begin{proof}
  By the pigeonhole principle, there is an $i \in \br{r}$ such that $|H_T\cap V_i|\ge |H_T|/r$. Fix any such $i$ and let $V_i = V_i' \cup V_i''$ be an arbitrary partition that maximises the number of edges of $T \setminus B_T$ incident to $H_T \cap V_i$ that cross the partition. Then, for every $v \in V_i'$, we have $\deg_T(v,V_i'') \ge \deg_T(v,V_i')$ and vice-versa. We let $Y$ be the larger of the two sets $H_T \cap V_i'$ and $H_T \cap V_i''$, so that $|Y| \ge |H_T \cap V_i| /2 \ge |H_T|/(2r)$. Without loss of generality, $Y = H_T \cap V_i'$. Writing $U = U_T$, we have, for every $v \in Y \subseteq H_T$,
  \[
    \begin{split}
      \deg_{T \setminus U}(v, V_i \setminus Y) & \ge \deg_{T \setminus U}(v, V_i'') \ge \frac{\deg_{T \setminus U}(v, V_i)}{2} = \frac{\deg_Tv - \deg_Uv}{2} \\
      & \ge \frac{\rho m}{2n} - \frac{D}{2} \ge \frac{\rho m}{2n} - \frac{\beta m}{2n \log n} \ge \frac{\rho m}{3n},
    \end{split}
  \]
  as claimed.
\end{proof}

Fix some $\Pi = \partition$ and $T\in\TTH$. Let $i_T \in \br{r}$ and $Y_T \subseteq V_{i_T}$ be the index and the set from the statement of Lemma~\ref{lemma:YT}. Let
\[
  \Dh = \left\lceil \frac{\rho m}{3n} \right\rceil ,
\]
and define $\ZT$ to be the family of all graphs that are obtained from $T$ by adding to it edges connecting each $v \in Y$ to some $\Dh$ vertices in each $V_i$ with $i \neq i_T$. Note that, for every $Z\in\ZT$,
\[
  e(Z) = e(T) + |Y_T| \cdot (r-1) \cdot \Dh.
\]
Recall from~\eqref{eq:Pi-unfriendly} that, for every $G \in \FreeT$ and every $v \in H_T$, we have $\deg_{G}(v,V_i) \ge \rho m/n \ge \Dh$ for every $i\in\br{r}$. This means, in particular, that for each $G \in \FreeT$, there is some $Z \in \ZT$ such that $Z \subseteq G$. In other words, defining, for each $Z \in \ZT$,
\[
  \FreeZ = \{G\in\FreeT : Z\subseteq G\},
\]
we have
\begin{equation}
  \label{eq:FreeT-FreeZ}
  \FreeT = \bigcup_{Z \in \ZT} \FreeZ.
\end{equation}

We now turn to bounding $|\FreeZ|$ from above.  To this end, fix some $T \in \TTH$ and $Z\in\ZT$. For every $v \in Y_T$ and every $i \in \br{r}$, let $N_i(v)$ be an arbitrary subset of $N_{Z \setminus U_T}(v) \cap (V_i \setminus Y)$ with $\Dh$ elements (and note that $N_i(v) = N_Z(v) \cap V_i$ when $i \neq i_T$).

Let $\vcrit$ be the centre of any critical star of $H$ and let $\Hm$ be the subgraph of $H$ obtained by removing $\vcrit$ and all the vertices whose only neighbour in $H$ is $\vcrit$. (As $H$ has no isolated vertices, neither does $\Hm$.) Let $W_1 = N_H(\vcrit) \cap V(\Hm)$ and $W_2 = V(\Hm) \setminus W_1$; denote $v_1 = |W_1|$ and $v_2 = |W_2|$. Since $\Hm$ is obtained from $H$ by removing the critical vertex $\vcrit$ (and possibly some additional vertices), it is $r$-colourable; let us fix an arbitrary proper colouring $\psi \colon V(\Hm) \to \br{r}$.

Define a $v_1$-partite $v_1$-uniform hypergraph $\cH_Z$ as follows:
\begin{align*}
  V(\cH_Z) & = \bigsqcup_{w \in W_1}V_{\psi(w)}, \\
  E(\cH_Z) & = \bigcup_{v \in Y_T} \left\{(v_w)_{w \in W_1} : v_w \in N_{\psi(w)}(v) \text{ for all $w \in W_1$, all distinct}\right\}.
\end{align*}
For every injection $\varphi\colon V(\Hm) \to \br{n}$, let $\Kphi$ be the labelled graph that is the image of $\Hm$ via the embedding $\varphi$. Define
\[
\Phi_Z = \left\{\varphi : \Kphi \subseteq \Pi - Y_T \text{ and } \big(\varphi(w)\big)_{w \in W_1} \in \cH_Z \right\};
\]
in other words, $\Phi_Z$ comprises all embeddings of $\Hm$ into $\Pi$ that avoid the set $Y_T$ and such that $W_1$ is mapped into $N_1(v) \cup \dotsb \cup N_r(v)$ for some $v \in Y_T$, accordingly with the colouring $\psi$.

Choose an arbitrary $G \in \FreeZ$. We claim that $G \cap \Pi$ cannot contain any of the~$K_\varphi$ with $\varphi \in \Phi_Z$. Suppose to the contrary that $K_\varphi \subseteq G \cap \Pi$ for some $\varphi \in \Phi_Z$. By the definitions of $\cH_Z$ and $\Phi_Z$, there is a vertex $v \in Y_T$ such that $\varphi(w) \in N_{\psi(w)}(v)$ for all $w \in W_1$. Since $N_i(v) \subseteq N_Z(v) \subseteq N_G(v)$ for all $i \in \br{r}$, extending $\varphi$ to $V(H)$ by first letting $\varphi(\vcrit) = v$ and then choosing $\varphi(w) \in N_Z(v)$ arbitrarily\footnote{One can keep $\varphi$ injective since $v_H \ll \rho m / n \le \deg_Z(v)$.} for all $w \in N_H(\vcrit) \setminus V(\Hm)$ would give an embedding of $H$ into $G$. In particular, letting $G'$ be a uniformly chosen random subgraph of $\Pi \setminus Z$ with $m - e(Z)$ edges, we have
\begin{equation}
  \label{eq:FreeZ-Pr}
  |\FreeZ| \le \Pr\big(K_{\varphi} \nsubseteq G' \text{ for each } \varphi \in \Phi_Z\big) \cdot \binom{e(\Pi)}{m-e(Z)}.  
\end{equation}

The probability in~\eqref{eq:FreeZ-Pr} can vary greatly with the distribution of the edges of the associated hypergraph $\cH_Z$. For a vast majority of $Z \in \ZT$, an upper bound on this probability that we will obtain using the Hypergeometric Janson Inequality will be sufficient to survive a naive union bound argument; we shall refer to this as the \emph{regular case}. There will be, however, a family of exceptional graphs $Z \in \ZT$ for which the distribution of the edges of the associated hypergraph $\cH_Z$ precludes obtaining a strong upper bound on the probability in~\eqref{eq:FreeZ-Pr}. We shall prove (using Lemma~\ref{lemma:d-sets}) that the number of such exceptional graphs $Z$ is extremely small; we shall refer to this as the \emph{irregular case}.

To make the above discussion precise, given a hypergraph $\cH$ on $\bigsqcup_{w \in W_1}V_{\psi(w)}$, a set $I\subseteq W_1$, and an $L \in\prod_{w\in I}V_{\psi(w)}$, the degree $\deg_{\cH}(L)$ of $L$ in $\cH$ is defined by
\[
  \deg_{\cH}(L) = |\{K\in\cH\colon L\subseteq K\}|,
\]
where we write $L \subseteq K$ to mean that $K$ agrees with $L$ on the coordinates indexed by $I$, and the maximal $I$-degree of $\cH$, denoted by $\Delta_I(\cH)$, is defined by
\[
  \Delta_I(\cH) = \max\big\{\deg_{\cH}(L) \colon L\in\prod_{w \in I}V_{\psi(w)}\big\};
\]
in particular $\Delta_\emptyset(\cH) = e(\cH)$.

In order to describe the split between the regular and the irregular cases, we need to introduce several additional parameters. First, let $\Gamma$ be a constant satisfying
\begin{equation}
  \label{eq:Gamma}
  \Gamma \ge \frac{21r}{\xi} ,
\end{equation}
and let $\alpha$ be a positive constant that satisfies
\begin{equation}
  \label{eq:alphaGamma}
  \big(3er\alpha^{1/v_H}v_H\big)^{\gamma} \le \exp(-12\Gamma).
\end{equation}
Moreover, let
\begin{equation}
  \label{eq:c2}
  c_2 = \frac{1}{2} \cdot \left(\frac{\rho}{2v_H}\right)^{v_H} ,
\end{equation}
and let $\sigma$ and $C_2$ be positive constants satisfying
\begin{equation}
  \label{eq:sigma-C_2}
  \max\left\{ \sigma \cdot (2r)^{v_1}, \frac{(4r)^{v_1}}{C_2} \right\} \le \min\big\{\tau_{\ref{lemma:d-sets}}(z, \alpha, \lambda \leftarrow 2^{-v_1}) : z \in \br{v_H} \big\}
\end{equation}

Let $\Zfr$ be the family of all $Z\in\ZT$ such that
\begin{equation}
  \label{eq:reg1}
  e(\cH_Z) \ge \sigma n^{v_1}.
\end{equation}
Let $\Zsr$ be the family of all $Z\in\ZT\setminus\Zfr$ such that $\cH_Z$ contains a subhypergraph $\cH\subseteq\cH_Z$ which satisfies
\begin{equation}\label{eq:reg2hedges}
  e(\cH)\ge c_2 \cdot |Y_T| \cdot \left(\frac{m}{n}\right)^{v_1}
\end{equation}
and, for every nonempty $I \subseteq W_1$,
\begin{equation}\label{eq:reg2delta}
  \Delta_I(\cH) \le \max\left\{\left(\frac{m}{n}\right)^{v_1 - |I|}, C_2\cdot\frac{e(\cH)}{n^{|I|}}\right\}.
\end{equation}
Finally, let $\Zr = \Zfr \cup \Zsr$ and $\Zi=\ZT\setminus\Zr$. Since $\Zr$ and $\Zi$ form a partition of $\ZT$ for every $T \in \TTH$, it follows from~\eqref{eq:FreeT-FreeZ} that
\begin{equation}
  \label{eq:reg-irreg}
  \sum_{\substack{T\in\TTH \\ B_T = B}}|\FreeT| \le \sum_{\substack{T \in \TTH \\ B_T = B}} \sum_{Z \in \Zr} |\FreeZ| + \sum_{\substack{T \in \TTH \\ B_T = B}} \sum_{Z \in \Zi} |\FreeZ|.
\end{equation}
The regular and the irregular cases are estimates of the first and the second sum in the right-hand side of~\eqref{eq:reg-irreg}, respectively.

\subsection{The regular case -- summary}
\label{sec:regular-case-summary}

In the regular case, we will rely on the following upper bound on the cardinality of $\FreeZ$, which is established in Section~\ref{sec:regular-case} with the use of the Hypergeometric Janson Inequality (Lemma~\ref{lem:HJI}).

\begin{lemma}
  \label{lemma:regCase}
  There exists a positive constant $\tilc$ that depends only on $H$ such that the following holds for every $T \in \TTH$ and each $Z\in\Zr$. If $n$ is sufficiently large and $m\ge \tilC\twoDensTS$ for some $\tilC \ge 2$, then
  \[
    |\FreeZ| \le \exp\left(- \tilc \cdot \min\left\{\frac{c_2 \cdot \tilC \cdot |Y_T|}{n}, \frac{1}{C_2}, \sigma\right\} \cdot m\right) \cdot \binom{e(\Pi)}{m - e(Z)}.
  \]
\end{lemma}

This upper bound on $|\FreeZ|$ provided by Lemma~\ref{lemma:regCase} will be combined with the following estimate on the size of the sum over all $Z \in \ZT$.

\begin{lemma}
  \label{lem:number-ZZT}
  For every $T \in \TTH$,
  \[
    |\ZT| \cdot \binom{e(\Pi)}{m-e(T)-|Y_T|\cdot (r-1) \cdot \Dh} \le \exp\left(\frac{|Y_T| \cdot  m}{\xi n}\right) \cdot \binom{e(\Pi)}{m-e(T)}.
  \]
\end{lemma}
\begin{proof}
  Since, for every $Z \in \ZT$, the graph $Z \setminus T$ comprises precisely $|Y_T| \cdot (r-1) \cdot \Dh$ edges incident to $Y_T$, we have, letting $b = |Y_T|$,
  \[
    |\ZT| \le \binom{n}{(r-1) \Dh}^b \leByRef{eq:binCoef4} \left(\frac{en}{(r-1) \Dh}\right)^{b (r-1) \Dh}.
  \]
  On the other hand, by~\eqref{eq:usefulEnumBound}, which holds for all $y \le m' \le m \le e(\Pi) - \xi n^2$, we have
  \[
    \frac{\binom{e(\Pi)}{m-e(T)-b\cdot (r-1) \cdot \Dh}}{\binom{e(\Pi)}{m-e(T)}}
    \le
    \left(\frac{m}{\xi n^2}\right)^{b (r-1) \Dh}.
  \]
  The claimed bound follows after noting that
  \[
    \left(\frac{en}{(r-1) \Dh} \cdot \frac{m}{\xi n^2}\right)^{(r-1)\Dh} \le \exp\left(\frac{m}{\xi n}\right),
  \]
  as $(ea/x)^x \le e^a$ for all $x \in (0, \infty)$.
\end{proof}

Before we close this section, we show how these two lemmas can be used to estimate the first sum in the right-hand side of~\eqref{eq:reg-irreg}:
\[
  \SigBR = \sum_{\substack{T \in \TTH \\ B_T = B}} \; \underbrace{\sum_{Z \in \Zr} |\FreeZ|}_{\SigTR}.
\]
Since
\[
  m \ge C_H m_H \ge C_H \twoDensTS \geByRef{eq:CH} \frac{1}{c_2 \cdot \tilc_{\ref{lemma:regCase}}} \cdot \frac{35r}{\xi} \cdot \twoDensTS,
\]
Lemma~\ref{lemma:regCase} implies that, for every $T \in \TTH$,
\[
  \SigTR \le \sum_{Z \in \Zr} \exp\left(-\min\left\{\frac{|Y_T|}{n} \cdot \frac{35r}{\xi}, \frac{\tilc_{\ref{lemma:regCase}}}{C_2}, \tilc_{\ref{lemma:regCase}}\sigma\right\} \cdot m\right)\cdot \binom{e(\Pi)}{m - e(Z)}.
\]
Since $|Y_T| \le |H_T| \le 2\delta n / \rho$ for every $T \subseteq \Pic$ with at most $\delta m$ edges, we have, for every $T \in \TTH$,
\[
  \frac{|Y_T|}{n} \le \frac{2\delta}{\rho} \leByRef{eq:delta} \frac{\xi}{35r} \cdot \min\left\{\frac{\tilc_{\ref{lemma:regCase}}}{C_2}, \tilc_{\ref{lemma:regCase}}\sigma\right\}
\]
and, consequently,
\[
  \SigTR \le \sum_{Z \in \Zr} \exp\left(-\frac{35rm \cdot |Y_T|}{\xi n}\right)\cdot \binom{e(\Pi)}{m - e(Z)}.
\]
Since $e(Z) = e(T) + |Y_T| \cdot (r-1) \cdot \Dh$ for every $Z \in \ZT$, Lemma~\ref{lem:number-ZZT} gives
\[
  \begin{split}
    \SigTR & \le \exp\left(-\frac{34rm\cdot|Y_T|}{\xi n}\right) \cdot \binom{e(\Pi)}{m - e(T)} \le \exp\left(-\frac{17m\cdot|H_T|}{\xi n}\right) \cdot \binom{e(\Pi)}{m - e(T)},
  \end{split}
\]
where the second inequality follows from the inequality $|Y_Y| \ge |H_T|/(2r)$, see Lemma~\ref{lemma:YT}.

Let $\cL$ be the family of all triples $(t, \ell, h)$ that satisfy $t \ge \ell \ge 1$ and $\ell \log n < mh/(\xi n)$, cf.~\eqref{eq:HighDegCon}, and observe that
\[
  \SigBR \le \sum_{(t, \ell, h) \in \cL} \; \sum_{T\in\TTsubB} \SigTR  \le \sum_{(t, \ell, h) \in \cL} |\TTsubB| \cdot \exp\left(-\frac{17mh}{\xi n}\right) \cdot \binom{e(\Pi)}{m - t}.
\]
Since, by Lemma~\ref{lem:number-ZZT}, we have, for every $(t, \ell, h) \in \cL$,
\[
  \begin{split}
    |\TTsubB| \cdot \binom{e(\Pi)}{m-t} & \le \exp\left(14\ell \log n+\frac{2mh}{\xi n}\right) \cdot \binom{e(\Pi)}{m-e(B)} \\
    & \le \exp\left(\frac{16mh}{\xi n}\right) \cdot \binom{e(\Pi)}{m - e(B)},
  \end{split}
\]
we may conclude that
\[
  \SigBR \cdot \binom{e(\Pi)}{m-e(B)}^{-1} \le \sum_{(t, \ell, h) \in \cL} \exp\left(-\frac{mh}{\xi n}\right) \le |\cL| \cdot \exp\left(-\frac{m}{\xi n}\right) \le \frac{1}{n},
\]
as $h \ge 1$ for every $(t, \ell, h) \in \cL$.

\subsection{The irregular case -- summary}
\label{sec:irreg-case-summ}

In the irregular case, we will use Lemma~\ref{lemma:d-sets} to prove upper bounds on the number of graphs $Z$ that fall into $\Zi$ for some $T \in \TTH$; these upper bounds will be so strong that we will be able to get the desired estimate on the second term in the right-hand side of~\eqref{eq:reg-irreg} by combining them with the trivial estimate $\binom{e(\Pi)}{m - e(Z)}$ on the number of completions of $Z$ to a graph in $\FreeZ$. Since the nature of our argument precludes obtaining a strong bound on $|\FreeZ \cap \Zi|$ for every $T$, we will have to partition the family $\bigcup_{T \in \TTH} \Zi$ differently. To this end, for every positive integer $b$, define
\[
  \Zib = \bigcup_{\substack{T \in \TTH \\ |Y_T| = b}} \Zi .
\]

Given some $T \in \TTH$ and a $Z \in \ZT$, let $T'_Z \subseteq T$ be the graph obtained from $T$ by removing the $|Y_T| \cdot \Dh$ edges $vu$ such that $v \in Y_T$ and $u \in N_{i_T}(v)$. Note that $B_T \subseteq U_T \subseteq T'_Z$, as $N_{i_T}(v)$ was defined to be a subset of $N_{Z \setminus U_T}(v)$, and that $T'_Z$ can be defined in terms of $Z$ only because if $Z \in \ZT$, then $T = Z \cap \Pic$. Further, for every positive integer $b$ and every $T' \subseteq \Pic$, let
\[
  \ZibT = \{Z \in \Zib : T'_Z = T'\}.
\]
The following upper bound on cardinalities of the families $\ZibT$ is the main step in the analysis of the irregular case.

\begin{lemma}
  \label{lem:irregCase}
  For every $T' \subseteq \Pic$ and every $b \ge 1$,
  \[
    |\ZibT| \cdot \binom{e(\Pi)}{m-e(T') - r b \Dh} \le \exp\left(- \frac{\Gamma b m}{n}\right) \cdot \binom{e(\Pi)}{m-e(T')}.
  \]
\end{lemma}

This upper bound on $|\ZibT|$ provided by Lemma~\ref{lem:irregCase} will be combined with the following estimate on the size of the sum over all $T'$. For every $B \in \BPk$ and every nonnegative integer $t'$, let $\TTpsubB$ comprise all graphs $T' \subseteq \Pic$ with $t'$ edges such that $T' = T'_Z$ for some $Z \in \ZT$, where $T \in \TTH$ satisfies $B_T = B$ and $|Y_T| = b$. 

\begin{lemma}
  \label{lem:enumeratingTp}
  Suppose that $n \log n \ll m \le e(\Pi) - \xi n^2$. For every $B \in \BPk$ and all $t'$ and $b$,
  \[
    |\TTpsubB| \cdot \binom{e(\Pi)}{m-t'} \le \exp\left(\frac{20rmb}{\xi n}\right) \cdot \binom{e(\Pi)}{m-e(B)}.
  \]
\end{lemma}
\begin{proof}
  We adapt the argument used in the proof of Lemma~\ref{lem:enumeratingT}. Suppose that $T' \in \TTpsubB$. This means that there is a $T \in \TTH$ such that $|Y_T| = b$, $B = B_T \subseteq U_T \subseteq T' \subseteq T$, and $T \setminus T'$ comprises some $b\Dh$ edges incident to $Y_T \subseteq H_T$. Moreover, since $T \in \TTH$, we have
  \[
    e(U_T \setminus B_T) \lByRef{eq:HighDegCon} \frac{m|H_T|}{\xi n \log n}.
  \]
  Let $U_T'$ be the subgraph of $U_T \setminus B_T$ obtained by removing all edges touching the set $X_T$ of vertices whose degree in $U_T$ is $D$. Since every edge of $U_T\setminus U_T'$ has at least one endpoint in $X_T$ and $\Delta(B_T) \le k$, we have
  \[
    e(U_T \setminus B) \ge e(U_T') +|X_T| \cdot (D-k)/2 \ge e(U_T') +|X_T| \cdot D/3.
  \]
  
 We choose the $t' - e(B)$ edges of $T' \setminus B$ in three steps:
  \begin{enumerate}[label={(S\arabic*)}]
  \item
    \label{item:enumeratingTp-1}
    We choose the edges of $U_T'$.
  \item
    \label{item:enumeratingTp-2}
    We choose the edges of $T' \setminus B$ that touch $X_T \setminus H_T$.
  \item
    \label{item:enumeratingTp-3}
    We choose the remaining edges of $T' \setminus B$; they all touch $H_T$.
  \end{enumerate}

  We count the number of ways to build a graph $T' \in \TTpsubB$ with $u'$, $t'_X$, and $t'_H$ edges chosen in steps~\ref{item:enumeratingT-1}, \ref{item:enumeratingT-2}, and~\ref{item:enumeratingT-3}, respectively. An upper bound on $|\TTpsubB|$ will be obtained by summing over all choices for $u'$, $t'_X$, and $t'_H$. There are at most $\binom{e(\Pic)}{u'}$ ways to choose $u'$ edges of $U_T'$ and, as in the proof of Lemma~\ref{lem:enumeratingT},
  \[
    \binom{e(\Pic)}{u'} \cdot \frac{\binom{e(\Pi)}{m-e(B)-u'}}{\binom{e(\Pi)}{m-e(B)}} \le m^{2u'}.
  \]
  Next, let $N_2 = N_2(t_X, s)$ denote the total number of ways to choose the $t'_X$ edges touching $X_T \setminus H_T$ when $|X_T \setminus H_T| = s$. As in the proof of Lemma~\ref{lem:enumeratingT}, we have
  \[
    N_2 \cdot \frac{\binom{e(\Pi)}{m-e(B)-u'-t'_X}}{\binom{e(\Pi)}{m-e(B)-u'}} \le m^{Ds} .
  \]
  Finally, let $N_3 = N_3(t_X, h)$ denote the number of ways to choose the remaining $t'_H$ edges of $T' \setminus B$ when $|H_T| = h$. Recalling that $e(B) + u' + t'_X + t'_H = t'$ and arguing as in the proof of Lemma~\ref{lem:enumeratingT}, we obtain
  \[
    N_3 \cdot \frac{\binom{e(\Pi)}{m-t'}}{\binom{e(\Pi)}{m-e(B)-u'-t'_X}} \le \exp\left(\frac{2mh}{\xi n}\right).
  \]
  Since $|H_T| \le 2r|Y_T| = 2rb$, by Lemma~\ref{lemma:YT}, combining the above bounds, we obtain
  \begin{align*}
    |\TTpsubB| \cdot \frac{\binom{e(\Pi)}{m-t'}}{\binom{e(\Pi)}{m-e(B)}}
    & \le \sum_{\substack{u',t'_X,t'_H,s,h \\ u'+t'_X+t'_H = t' \\ u' + sD/3 \le mh/(\xi n \log n) \\ h \le 2rb}} m^{2u'} \cdot m^{Ds} \cdot \exp\left(\frac{2mh}{\xi n}\right) \\
    & \le nm^3 \sum_{h \le 2rb} \exp\left(\left(\frac{3\log m}{\log n} + 2\right) \cdot \frac{mh}{\xi n}\right)
      \le \exp\left(\frac{20rmb}{\xi n}\right),
  \end{align*}
  as claimed.
\end{proof}

Before we close this section, we show how these two lemmas can be used to estimate the second sum in the right-hand side of~\eqref{eq:reg-irreg}:
\[
  \begin{split}
    \SigBI & = \sum_{\substack{T \in \TTH \\ B_T = B}} \; \sum_{Z \in \Zi} |\FreeZ| \le \sum_{\substack{T \in \TTH \\ B_T = B}} \; \sum_{Z \in \Zi} \binom{e(\Pi)}{m-e(Z)} \\
    & = \sum_{t',b} \; \sum_{T' \in \TTpsubB} \underbrace{\sum_{Z \in \ZibT} \binom{e(\Pi)}{m-e(Z)}}_{\SigTpI}.
  \end{split}
\]
Since $e(Z) = e(T') + rb\Dh$ for every $Z \in \ZibT$, Lemma~\ref{lem:irregCase} implies that
\[
  \SigTpI \le \exp\left(- \frac{\Gamma b m}{n}\right) \cdot \binom{e(\Pi)}{m-e(T')} ,
\]
and, further, Lemma~\ref{lem:enumeratingTp} implies that
\[
  \begin{split}
    \SigBI \cdot \binom{e(\Pi)}{m-e(B)}^{-1} & \le \sum_{t',b} \exp\left(\frac{20rmb}{\xi n} - \frac{\Gamma b m}{n}\right) \leByRef{eq:Gamma} \sum_{t',b} \exp\left(-\frac{rmb}{\xi n}\right) \\
    & \le mn \cdot \exp\left(-\frac{rm}{\xi n}\right) \le \frac{1}{n}.
  \end{split}
\]

\subsection{The regular case}
\label{sec:regular-case}

In this section, we prove Lemma~\ref{lemma:regCase}, that is, for given $T \in \TTH$ and $Z \in \Zfr \cup \Zsr$, we give an upper bound on the number of graphs in $\FreeZ$.

\begin{proof}[{Proof of Lemma~\ref{lemma:regCase}}]
  Suppose that $\Pi = \partition$, let $T \in \TTH$, and fix an arbitrary $Z \in \Zfr \cup \Zsr$. If $Z \in \Zfr$, we let $\cH = \cH_Z$ and recall that $e(\cH) \ge \sigma n^{v_1}$, see~\eqref{eq:reg1}. Otherwise, $Z \in \Zsr$ and we let $\cH\subseteq\cH_Z$ be any hypergraph which satisfies both~\eqref{eq:reg2hedges} and~\eqref{eq:reg2delta}.

  Recall the definitions of $\Hm$, $\cH_Z$, $\psi$, and $\Phi_Z$ from Section~\ref{sec:high}. For every $j \in \br{r}$, randomly choose an equipartition $\{V_{j,w}\}_{w \in V(H)}$ of $V_j \setminus Y_T$ into $v_H$ parts. We let $\Phi_Z'$ be the family of all embeddings $\varphi \in \Phi_Z$ that satisfy
  \[
    \big(\varphi(w)\big)_{w \in W_1} \in \cH
    \quad \text{and} \quad
    \varphi(w) \in V_{\psi(w), w} \text{ for every $w \in V(\Hm)$}.
  \]
  Let $n' = \min\{|V| : V \in \Part\} \ge n/(2r)$. Since there are at least $e(\cH) \cdot (n'-|Y_T|-v_H)^{v_2}$ embeddings $\varphi \in \Phi_Z$ such that $\big(\varphi(w)\big)_{w \in W_1} \in \cH$ and, for each such $\varphi$, the probability that $\varphi \in \Phi_Z'$ is at least $v_H^{-v_H}$, there is a positive constant $c$ that depends only on $H$ such that
  \[
    \Ex\big[|\Phi_Z'|\big] \ge c \cdot e(\cH) \cdot n^{v_2}.
  \]
  Now, fix some partitions $\{V_{j,w}\}_{w \in V(H)}$ for which $|\Phi_Z'|$ is at least as large as its expectation and let
  \[
    \cK' = \{K_\varphi : \varphi \in \Phi_Z'\}.
  \]
  We claim that $K_\varphi \neq K_{\varphi'}$ for each pair of distinct $\varphi, \varphi' \in \Phi_Z'$. Since $\Hm$ has no isolated vertices and each $\varphi \in \Phi_Z'$ maps every $w \in V(\Hm)$ to its dedicated set $V_{\psi(w),w}$, one can recover $\varphi$ from the graph $\Kphi$. This means, in particular, that
  \begin{equation}
    \label{eq:KKpsize}
     |\cKp| = |\Phi_Z'| \ge c \cdot e(\cH) \cdot n^{v_2}.
   \end{equation}
  
   Suppose that $m\ge \tilC\twoDensTS$ for some $\tilC \ge 2$, and let $G'$ be a uniformly chosen subgraph of $\Pi \setminus Z$ with $m-e(Z)$ edges. The definition of $\cKp$ and~\eqref{eq:FreeZ-Pr} imply that
   \[
     |\FreeZ| \le \Pr(K \nsubseteq G'\text{ for every }K\in \cKp) \cdot \binom{e(\Pi)}{m-e(Z)}.
   \]
   We shall bound this probability from above using the Hypergeometric Janson Inequality. To this end, let $p=\frac{m-e(Z)}{e(\Pi)-e(Z)}$. Since
   \[
     e(Z) \le e(T) + (r-1) \cdot |H_T| \cdot \Dh \le 2r\delta m,
   \]
   as $|H_T| \le 2\delta n/\rho$ and $\Dh \le \rho m /n$, we have
   \[
     p \ge \frac{m-e(Z)}{n^2} \ge (1-2r\delta) \cdot \frac{m}{n^2} \ge \frac{m}{2n^2} \ge \frac{\tilC}{2} \cdot \twoDensTS,
   \]
   as $\delta \le \frac{1}{4r}$, see~\eqref{eq:delta}, and
   \[
     p \le \frac{m}{e(\Pi) - e(Z)} \le \frac{m}{n^2/5 - 2r\delta n^2} \le \frac{10m}{n^2},
   \]
   where the second inequality follows from part~\ref{item:Partg-lower} of Proposition~\ref{prop:bal-unbal-part-size}, as $\Pi \in \Partg$ and $\gamma \le \frac{1}{20r}$, see~\eqref{eq:gamma}, and the final inequality holds because $\delta \le \frac{1}{20r}$, see~\eqref{eq:delta}. For any $K, K' \in \cKp$, we write $K\sim K'$ if $K$ and $K'$ share an edge but $K \neq K'$. Let $\mu$ and $\Delta$ be the quantities defined in the statement of the Hypergeometric Janson Inequality (Lemma~\ref{lem:HJI}), that is
   \[
     \mu = \sum_{K\in\cKp}p^{e_K} = |\cKp| \cdot p^{\emi} \qquad \text{and} \qquad \Delta = \sum_{\substack{K, K' \in \cK' \\ K \sim K'}} p^{e_{K\cup K'}}.
   \]

   \begin{claim}
     \label{claim:mu-lower-high}
     There is a positive constant $c'$ that depends only on $H$ such that
     \begin{equation}
       \label{eq:mu-lower-high}
       \mu \ge c' \cdot \min\left\{\frac{c_2 \cdot \tilC \cdot |Y_T|}{n}, \sigma\right\} \cdot m.
     \end{equation}
   \end{claim}
   \begin{proof}
     It follows from~\eqref{eq:KKpsize} that
     \[
       \mu = |\cKp| \cdot p^{\emi} \ge c \cdot e(\cH) \cdot n^{v_2} \cdot p^{\emi}.
     \]
     Assume first that $Z \in \Zfr$. Since $e(\cH) \ge \sigma n^{v_1}$, we have
     \begin{equation}
    \label{eq:mu-lower-ZZfr}
    \mu \ge c \cdot \sigma \cdot n^{v_1+v_2} \cdot p^{\emi}.
  \end{equation}
  Since $v_1+v_2$ is the number of vertices of $\Hm$ and $\Hm \subseteq H$, Lemma~\ref{lemma:m2H} implies that $\mu \ge c \cdot \sigma \cdot m$, as $\tilC \ge 2$. If, on the other hand, $Z \in \Zsr$, then
  \begin{align*}
    \mu
    & \ge c \cdot c_2 \cdot |Y_T|\cdot \left(\frac{m}{n}\right)^{v_1} \cdot n^{v_2} \cdot p^{\emi} \ge c \cdot c_2 \cdot |Y_T| \cdot \left(\frac{pn}{10}\right)^{v_1} \cdot n^{v_2} \cdot p^{\emi} \\
    & \ge c'' \cdot c_2 \cdot |Y_T| \cdot \frac{n^{v_1 + v_2+1} p^{\emi + v_1}}{n}
  \end{align*}
  for some $c''$ that depends only on $H$. Let $H^*$ be the subgraph of $H$ induced by $\{\vcrit\} \cup V(\Hm)$ and note that $v_{H^*} = v_1+v_2+1$ and $e_{H^*} = e_{\Hm}+v_1$. Since $\vcrit$ is the centre of a critical star of $H$, it has at least $\chi(H) \ge 3$ neighbours and thus $e_{H^*} \ge v_1 \ge 3$. By Lemma~\ref{lemma:m2H}, with $F = H^*$,
  \[
    n^{v_1+v_2+1}p^{\emi + v_1} = n^{v_{H^*}} p^{e_{H^*}} \ge \frac{\tilC}{2} \cdot n^2p \ge \frac{\tilC m}{4},
  \]
  and we may conclude that $\mu \ge c' \cdot c_2 \cdot \tilC \cdot |Y_T| \cdot m / n$. This completes the proof of~\eqref{eq:mu-lower-high}.  
\end{proof}

\begin{claim}
  \label{claim:Delta-upper}
  There exists a positive constant $c'$ that depends only on $H$ such that
  \[
    \frac{\mu^2}{\Delta} \ge c' \cdot \min\left\{\frac{c_2 \cdot \tilC \cdot |Y_T|}{n}, \sigma, \frac{1}{C_2}\right\} \cdot m.
  \]
\end{claim}
\begin{proof}
  For every $I \subseteq W_1$ and $J \subseteq W_2$ let $\HIJ$ be the subgraph of $\Hm$ (and thus also of $H$) induced by $I \cup J$; note that $\HIJ$ may have isolated vertices. Further, let $\cK(I,J)$ be the set of all pairs $K, K' \in \cK'$ that agree exactly on (the image of) $I \cup J$, that is,
  \[
    \cK(I,J) = \left\{(K_\varphi, K_{\varphi'}) \in (\cKp)^2 : K_{\varphi} \cap K_{\varphi'} = \varphi(H_{I,J}) = \varphi'(H_{I,J}) \right\}.
  \]
  These definitions were made in such a way that
  \begin{equation}
    \label{eq:Delta-high-deg-upper}
    \begin{split}
      \Delta
      & = \sum_{K \in \cKp} \sum_{\substack{I \subseteq W_1, J \subseteq W_2 \\ \emptyset \neq H_{I,J} \subsetneq \Hm}} \sum_{\substack{K' \in \cKp \\ (K, K') \in \cKp(I,J)}} p^{2\emi - e(H_{I,J})} \\
      & \le \sum_{K \in \cKp} p^{\emi} \sum_{\substack{I \subseteq W_1, J \subseteq W_2 \\ \emptyset \neq H_{I,J} \subsetneq \Hm}} |\{K' \in \cKp : (K, K') \in \cKp(I,J)\}| \cdot p^{\emi-e_{\HIJ}} \\
      & \le \mu \sum_{\substack{I \subseteq W_1, J \subseteq W_2 \\ \emptyset \neq H_{I,J} \subsetneq \Hm}} \Delta_I(\cH) \cdot n^{v_2 - |J|} \cdot p^{\emi - e_{\HIJ}}.
    \end{split}
  \end{equation}
  
  Assume first that $Z\in\Zfr$. Using the trivial bound $\Delta_I(\cH)\le n^{v_1-|I|}$, which is valid for all $I \subseteq W_1$, and~\eqref{eq:Delta-high-deg-upper}, we obtain
  \[
    \begin{split}
      \frac{\Delta}{\mu} & \le \sum_{\substack{I \subseteq W_1, J \subseteq W_2 \\ \emptyset \neq H_{I,J} \subsetneq \Hm}} n^{v_1+v_2-|I|-|J|} \cdot p^{\emi - e_{\HIJ}} = \sum_{\substack{I \subseteq W_1, J \subseteq W_2 \\ \emptyset \neq H_{I,J} \subsetneq \Hm}} \frac{n^{v_1+v_2} p^{\emi}}{n^{v_{\HIJ}} p^{e_{\HIJ}}} \\
      & \le 2^{v_1+v_2} \cdot \frac{n^{v_1+v_2}p^{\emi}}{\min_{\emptyset \neq F \subseteq \Hm} n^{v_F}p^{e_F}} \leBy{L.~\ref{lemma:m2H}} 2^{v_1+v_2} \cdot \frac{n^{v_1+v_2} p^{\emi}}{n^2p}.
    \end{split}
  \]
  Since $n^{v_1+v_2} p^{\emi} \le \mu/(c \cdot \sigma)$, see~\eqref{eq:mu-lower-ZZfr} and $n^2p \ge m/2$, we may conclude that
  \[
    \frac{\mu^2}{\Delta} \ge \frac{c \cdot \sigma}{2^{v_1+v_2+1}} \cdot m.
  \]
  
  Suppose now that $Z \in \Zsr$. In this case, for all nonempty $I \subseteq W_1$,
  \[
    \begin{split}
      \Delta_I(\cH) & \le \max\left\{\left(\frac{m}{n}\right)^{v_1 - |I|}, C_2\cdot\frac{e(\cH)}{n^{|I|}}\right\}
      \leByRef{eq:reg2hedges} \max\left\{\frac{1}{c_2|Y_T|} \cdot \left(\frac{m}{n^2}\right)^{-|I|}, C_2 \right\} \cdot \frac{e(\cH)}{n^{|I|}} \\
      & \le \max\left\{\frac{1}{c_2|Y_T|} \cdot \left(\frac{10}{p}\right)^{|I|}, C_2 \right\} \cdot \frac{e(\cH)}{n^{|I|}}.
    \end{split}
  \]
  Denote by $\Delta_0$ and $\Delta_1$ the contributions to the sum in the right-hand side of~\eqref{eq:Delta-high-deg-upper} corresponding to $I = \emptyset$ and $I \neq \emptyset$, respectively, so that $\Delta \le \Delta_0 + \Delta_1$. Since $H_{\emptyset, J} = H[J] \subseteq H$ and $\Delta_\emptyset(\cH) = e(\cH)$, we have
  \[
    \frac{\Delta_0}{\mu} \le e(\cH) \cdot n^{v_2}p^{\emi} \cdot \sum_{\substack{J \subseteq W_2 \\ H[J] \neq \emptyset}} \frac{1}{n^{|J|} p^{e(H[J])}} \le 2^{v_2} \cdot \frac{e(\cH) \cdot n^{v_2}p^{\emi}}{\min_{\emptyset \neq F \subseteq H} n^{v_F}p^{e_F}}.
  \]
  Recalling that $e(\cH) \cdot n^{v_2}p^{\emi} \le \mu / c$, we conclude, using Lemma~\ref{lemma:m2H}, that
  \[
    \frac{\Delta_0}{\mu} \le \frac{2^{v_2}}{c} \cdot \frac{\mu}{n^2p} \le \frac{2^{v_2+1} \mu}{cm}.
  \]
  On the other hand,
  \[
    \begin{split}
      \frac{\Delta_1}{\mu} & \le e(\cH) \cdot n^{v_2}p^{\emi}  \cdot \sum_{\substack{\emptyset \neq I \subseteq W_1, J \subseteq W_2 \\ \emptyset \neq H_{I,J} \subsetneq \Hm}} \max\left\{\frac{1}{c_2|Y_T|} \cdot \frac{10^{|I|}}{p^{|I|}}, C_2 \right\} \cdot \frac{1}{n^{|I| + |J|} p^{e_{\HIJ}}}\\
      & = e(\cH) \cdot n^{v_2} p^{\emi} \cdot \sum_{\substack{\emptyset \neq I \subseteq W_1, J \subseteq W_2 \\ \emptyset \neq H_{I,J} \subsetneq \Hm}} \max\left\{ \frac{n}{c_2|Y_T|} \cdot \frac{10^{|I|}}{np^{|I|}}, C_2 \right\} \cdot \frac{1}{n^{v_{\HIJ}} p^{e_{\HIJ}}}.
    \end{split}
  \]
  Fix a nonempty $I \subseteq W_1$ and a $J \subseteq W_2$ such that $H_{I,J}$ is nonempty. Since $v_{\HIJ}+1$ and $e_{\HIJ} + |I| \ge 2$ are the numbers of vertices and edges of the subgraph of $H$ induced by $\{\vcrit\} \cup I \cup J$, Lemma~\ref{lemma:m2H} implies that
  \[
    \max\left\{ \frac{n}{c_2|Y_T|} \cdot \frac{10^{|I|}}{np^{|I|}}, C_2 \right\} \cdot \frac{1}{n^{v_{\HIJ}} p^{e_{\HIJ}}} \le \max\left\{\frac{n}{c_2|Y_T|} \cdot \frac{2 \cdot 10^{|I|}}{\tilC}, C_2 \right\} \cdot \frac{1}{n^2p}.
  \]
  Recalling again that $e(\cH) \cdot n^{v_2}p^{\emi} \le \mu/c$, we have
  \[
    \frac{\Delta_1}{\mu} \le \frac{\mu}{c} \cdot 2^{v_1+v_2} \cdot \max\left\{\frac{n \cdot 10^{v_1+1}}{c_2|Y_T| \cdot \tilC}, C_2\right\} \cdot \frac{2}{m}.
  \]
  We may conclude that
  \[
    \frac{\mu^2}{\Delta} \ge \frac{\mu^2}{\Delta_0+\Delta_1} \ge c' \cdot \min\left\{\frac{c_2 \cdot \tilC \cdot |Y_T|}{n}, \frac{1}{C_2}\right\} \cdot m,
  \]
  where $c'$ is a positive constants that depends only on $H$.
\end{proof}
   Finally, we invoke Lemma~\ref{lem:HJI} with $q = \frac{\mu}{\mu + \Delta} \le 1$ to conclude that
   \[
     \begin{split}
       \frac{|\FreeZ|}{\binom{e(\Pi)}{m-e(Z)}}
       & \le \Pr(K \nsubseteq G'\text{ for every }K\in \cKp) \le \exp\left(-\frac{\mu^2}{\mu+\Delta} + \frac{\mu^2\Delta}{2(\mu+\Delta)^2}\right) \\
       & \le \exp\left(-\frac{\mu^2}{2(\mu + \Delta)}\right) \le \exp\left(-\min\left\{\frac{\mu}{4}, \frac{\mu^2}{4\Delta}\right\}\right).
     \end{split}
   \]
   The assertion of the lemma now follows from Claims~\ref{claim:mu-lower-high} and~\ref{claim:Delta-upper}.
\end{proof}

\subsection{The irregular case}

In this section, we prove Lemma~\ref{lem:irregCase}, that is, for given $T' \subseteq \Pic$, we give an upper bound on the number of graphs $Z \in \Zi$, for some $T \in \TTH$ satisfying $|Y_T| = b$, such that $T_Z' = T'$.

\begin{proof}[{Proof of Lemma \ref{lem:irregCase}}]
  Fix some graph $T' \subseteq \Pic$, an integer $b \ge 1$, a colour $i \in \br{r}$, and distinct $v_1,\dotsc, v_b \in V_i$. We will describe a procedure that constructs, for every graph $Z$ such that $T'_Z = T'$ and $Y_{T_Z} = \{v_1, \dotsc, v_b\}$, a hypergraph $\cH \subseteq \cH_Z$ that satisfies condition~\eqref{eq:reg2delta} for every nonempty $I \subseteq W_1$. Our procedure will examine the neighbourhoods of $v_1, \dotsc, v_b$ in the graph $Z \setminus T'$ one-by-one and build $\cH$ in an online fashion. If $Z\in\Zib$, then the constructed hypergraph $\cH$ cannot have too many edges. More precisely, $\cH$ has to fail condition~\eqref{eq:reg2hedges} and, moreover, $\cH_Z$ must not satisfy~\eqref{eq:reg1}. This means, roughly speaking, that, when $Z \in \Zib$, the neighbourhoods of $v_1, \dotsc, v_b$ in $Z \setminus T'$ are highly correlated. This will allow us, with the use of Lemma \ref{lemma:d-sets}, to bound the number of choices for these neighbourhoods that result in a graph $Z \in \Zib$. Consequently, we will obtain an upper bound on  the size of the set $\ZibT$.

  Let
  \[
    \Ds = \left\lfloor \frac{\Dh}{v_1} \right\rfloor \ge \left\lfloor \frac{\rho m}{2v_1 n} \right\rfloor ,
  \]
  and let $\cH_0$ be the empty hypergraph with vertex set $\bigsqcup_{w \in W_1} V_{\psi(w)}$. Do the following for $s = 1, \dotsc, b$:
  \begin{enumerate}[label=(\roman*)]
  \item
    For every nonempty $I \subsetneq W_1$, let
    \[
      M_s^I =\left\{ L \in \prod_{w \in I}V_{\psi(w)} :  \deg_{\cH_{s-1}}(L) > \frac{C_2}{2}\cdot\frac{e(\cH_{s-1})}{n^{|I|}}\right\}.
    \]
  \item
    \label{item:alg-step-two}
    For each $j \in \br{r}$, choose an arbitrary collection $\{N_{j,w}(v_s)\}_{w \in W_1}$ of $v_1$ pairwise disjoint subsets of $N_j(v_s)$, each of size $\Ds$, denote $N(v_s) =\prod_{w \in W_1} N_{\psi(w),w}(v_s)$, and let
    \[
      \cH_s = \cH_{s-1} \cup \left\{K \in N(v_s) :  L \nsubseteq K\text{ for all }L\in \bigcup_{\emptyset\ne I\subsetneq W_1}M_s^I\right\}.
    \]
  \end{enumerate}
  Finally, let $\cH =\cH_b$.

  By construction, every $(v_w)_{w \in W_1} \in N(v_s)$ has distinct coordinates and hence $\cH \subseteq \cH_Z$. Moreover, for every nonempty $I\subsetneq W_1$,
  \[
    \begin{split}
      \Delta_I(\cH)
      &\le \frac{C_2}{2}\cdot\frac{e(\cH)}{n^{|I|}} + \Delta_I\big(N(v_s)\big) \le \frac{C_2}{2}\cdot\frac{e(\cH)}{n^{|I|}} + \prod_{w \in W_1 \setminus I} |N_{\psi(w)}(v_s)|
      \\
      & \le\max\left\{2\Dh^{v_1 - |I|}, C_2\cdot\frac{e(\cH)}{n^{|I|}}\right\} \le\max\left\{\left(\frac{m}{n}\right)^{v_1 - |I|}, C_2\cdot\frac{e(\cH)}{n^{|I|}}\right\},
    \end{split}
  \]
  as $|N_j(v_s)| = \Dh \le \rho m / n \le m/(2n)$ for all $j \in \br{r}$ and $s \in \br{b}$. Moreover, since $\Delta_{W_1}(\cH) \le 1 = (m/n)^{v_1-|W_1|}$, our $\cH$ satisfies~\eqref{eq:reg2delta} for every nonempty $I \subseteq W_1$.
  
  We say that $s \in \br{b}$ is \textit{useful} if
  \[
    e\big(\cH_s\setminus\cH_{s-1}\big) \ge 2^{-v_1} \cdot \Ds^{v_1}.
  \]
  If more than half of $s \in \br{b}$ are useful, then
  \[
    e(\cH) = \sum_{s=1}^b e\big(\cH_s \setminus \cH_{s-1}\big) \ge \frac{b}{2} \cdot 2^{-v_1} \cdot \Ds^{v_1} \ge 2^{-v_H} \cdot b \cdot \left\lfloor\frac{\rho m}{2v_1n}\right\rfloor^{v_1} \ge c_2 \cdot b \cdot\left(\frac{m}{n}\right)^{v_1},
  \]
  where the last inequality follows from~\eqref{eq:c2}; in particular $\cH$ satisfies condition~\eqref{eq:reg2hedges} and thus $Z\in\Zsrz$. Therefore, if $Z \in \Zib$, then at least half of $s \in \br{b}$ are not useful.

  \begin{claim}
    \label{claim:not-useful}
    Let $s \in \br{b}$ and suppose that $e(\cH_{s-1}) < \sigma n^{v_1}$. Then, there are at most
    \[
      \exp\left(-\frac{4\Gamma m}{n}\right) \cdot \binom{n}{r \Dh}
    \]
    choices for $N_1(v_s), \dotsc, N_r(v_s)$ such that $s$ is not useful.
  \end{claim}
  \begin{proof}
    For every $I \subseteq W_1$, denote $N_I(v_s) =\prod_{w \in I} N_{\psi(w),w}(v_s)$, where $\{N_{j,w}(v_s)\}$ is the collection defined in step~\ref{item:alg-step-two} of the algorithm building $\cH$. Letting $M_s^{W_1} = \cH_{s-1}$, we have
    \[
      e\big(\cH_s \setminus \cH_{s-1}\big) \ge \Ds^{v_1} - \sum_{\emptyset \neq I \subseteq W_1} |N_I(v_s) \cap M_s^I| \cdot \Ds^{v_1-|I|}.
    \]
    In particular, if $s$ is not useful then there must be some nonempty $I \subseteq W_1$ such that
    \begin{equation}
      \label{eq:not-useful-I}
      \big|N_I(V_s) \cap M_s^I\big| > 2^{-v_1} \cdot \Ds^{|I|}.
    \end{equation}
    Since $|V_j| \ge n/(2r)$ for every $j \in \br{r}$, we have
    \[
      |M_s^{W_1}| = e(\cH_{s-1}) < \sigma n^{v_1} \le \sigma \cdot (2r)^{v_1} \cdot \prod_{w \in W_1} |V_{\psi(w)}|.
    \]
    Moreover, for every $\emptyset \neq I \subsetneq W_1$,
    \[
      |M_s^I| \cdot \frac{C_2}{2} \cdot \frac{e(\cH_{s-1})}{n^{|I|}} \le \sum_{L \in M_s^I} \deg_{\cH_{s-1}}(L) \le \binom{v_1}{|I|}\cdot e(\cH_{s-1})
    \]
    and hence
    \[
      |M_s^I| \le \frac{1}{C_2} \cdot \binom{v_1}{|I|} \cdot n^{|I|} \le \frac{2^{v_1}}{C_2} \cdot n^{|I|} \le \frac{(4r)^{v_1}}{C_2} \cdot \prod_{w \in I} |V_{\psi(w)}|.
    \]
    
    Since we chose $\sigma$ to be sufficiently small and $C_2$ to be sufficiently large as a function of $\alpha$ and $v_1$, see~\eqref{eq:sigma-C_2}, Lemma~\ref{lemma:d-sets} applied $2^{v_1}-1$ times implies that there are at most
    \[
      (2^{v_1} - 1) \cdot \alpha^{\Ds} \cdot \prod_{w \in W_1} \binom{|V_{\psi(w)}|}{\Ds}
    \]
    choices of $N(v_s)$ such that~\eqref{eq:not-useful-I} holds for some nonempty $I \subseteq W_1$. On the other hand, the number of choices for $N_1(v_s), \dotsc, N_r(v_s)$ that can yield a given $N(v_s)$ is at most $\binom{n}{r \Dh - v_1 \Ds}$. We conclude that the number $X$ of choices for $N_1(v_s), \dotsc, N_r(v_s)$ that render $s$ not useful satisfies
    \[
      \begin{split}
        X
        & \le 2^{v_1} \cdot \alpha^{\Ds} \cdot \binom{n}{r \Dh - v_1 \Ds} \cdot \prod_{w \in W_1} \binom{|V_{\psi(w)}|}{\Ds} \\
        & = 2^{v_1} \cdot \alpha^{\Ds} \cdot \binom{n}{r \Dh} \cdot \binom{r \Dh}{v_1 \Ds} \cdot \underbrace{{\binom{n-r \Dh + v_1 \Ds}{v_1 \Ds}}^{-1} \cdot \prod_{w \in W_1} \binom{|V_{\psi(w)}|}{\Ds}}_{(\star)}.
      \end{split}
    \]
    Since $n - r \Dh \ge 2n/3 \ge |V_j|$ for every $j \in \br{r}$, we have
    \[
      (\star) \le \binom{2n/3 + v_1 \Ds}{v_1 \Ds}^{-1} \cdot \binom{2n/3}{\Ds}^{v_1} \le \binom{2n/3 + v_1 \Ds}{v_1 \Ds}^{-1} \cdot \binom{v_1 \cdot 2n/3}{v_1 \Ds} \leByRef{eq:binCoef3} v_1^{v_1 \Ds}.
    \]
    Finally, since $v_1 \Ds \ge \Dh - v_1 \ge 2\Dh/3 \ge \rho m / (3n)$, we conclude that
    \[
      \begin{split}
        X \cdot \binom{n}{r \Dh}^{-1}
        & \le 2^{v_1} \cdot \alpha^{\Ds} \cdot \binom{r \Dh}{v_1 \Ds} \cdot v_1^{v_1 \Ds} \leByRef{eq:binCoef4} \left(2 \cdot \alpha^{1/v_1} \cdot \frac{e r \Dh}{v_1 \Ds} \cdot v_1 \right)^{v_1 \Ds} \\
        & \le \left(3er \cdot \alpha^{1/v_1} v_1 \right)^{\frac{\rho m}{3n}} \leByRef{eq:alphaGamma} \exp\left(-\frac{4\Gamma m}{n}\right),
      \end{split}
    \]
    giving the assertion of the claim.
  \end{proof}
  
  We are now ready to prove the claimed upper bound on the size of the family $\ZibT$. Each graph $Z$ in this family can be constructed by specifying an $i \in \br{r}$, a sequence of distinct vertices $v_1, \dotsc, v_b \in V_i$, and a set $S \subseteq \br{b}$ of size at least $b/2$ such that, when we execute the algorithm described above, every $s \in S$ is not useful. Since the number of choices for $N_1(v_s), \dotsc, N_r(v_s)$ is at most $\exp(-4\Gamma m/n) \cdot \binom{n}{r \Dh}$ when $s \in S$, by Claim~\ref{claim:not-useful}, and at most $\binom{n}{r \Dh}$ when $s \in \br{r} \setminus S$, we have
  \[
    |\ZibT| \le r \cdot n^b \cdot 2^b \cdot \exp\left(-\frac{4\Gamma m}{n} \cdot \frac{b}{2}\right) \binom{n}{r \Dh}^b \leByRef{eq:binCoef4} \exp\left(-\frac{3\Gamma m}{2n} \cdot b\right) \left(\frac{en}{r \Dh}\right)^{b r \Dh}.
  \]
  On the other hand, by~\eqref{eq:usefulEnumBound}, which holds for all $y \le m' \le m \le e(\Pi) - \xi n^2$, we have
  \[
    \frac{\binom{e(\Pi)}{m-e(T')-b r\Dh}}{\binom{e(\Pi)}{m-e(T')}}
    \le
    \left(\frac{m}{\xi n^2}\right)^{b r \Dh}.
  \]
  It follows that
  \begin{multline*}
    |\ZibT| \cdot \binom{e(\Pi)}{m-e(T')-b r \Dh} \\
    \le \exp\left(-\frac{3\Gamma m}{2n} \cdot b\right) \cdot \left(\frac{en}{r \Dh} \cdot \frac{m}{\xi n^2}\right)^{b r\Dh} \binom{e(\Pi)}{m-e(T')}.
  \end{multline*}
  The claimed bound follows after noting that, since $(ea/x)^x \le e^a$ for all $x \in (0, \infty)$,
  \[
    \left(\frac{en}{r \Dh} \cdot \frac{m}{\xi n^2}\right)^{r\Dh} \le \exp\left(\frac{m}{\xi n}\right) \leByRef{eq:Gamma} \exp\left(\frac{\Gamma m}{2 n}\right).\qedhere
  \]
\end{proof}

\section{The 1-statement: the dense case}
\label{sec:dense}

Fix a partition $\Pi \in \Partg$. In this section, we verify the assumptions of Proposition~\ref{prop:sufficient-1-statement} in the case where 
\[
  e(\Pi) - \xi n^2 \le m\le \exnH.
\]
We start by introducing two additional parameters. Let $\eps$ and $\nu$ be positive constants satisfying
\begin{equation}
  \label{eq:eps-nu}
  \eps + v_H \nu \le \frac{1}{2r} \qquad \text{and} \qquad \eps \le \nu/8.
\end{equation}
Earlier on, we chose $\gamma$, $\delta$, and $\xi$ sufficiently small so that
\begin{equation}
  \label{eq:xi-dense}
  320\xi \le \nu \qquad \text{and} \qquad  \left(\frac{4e}{\nu \eps}\right)^\eps \cdot \left(\frac{320\xi}{\nu}\right)^{\nu/4} \le e^{-2} ,
\end{equation}
and, additionally,
\begin{equation}
  \label{eq:xi-delta-eps-nu-gamma}
  \xi + \delta \le 2\max\left\{\xi, \delta\right\} < \min\left\{\frac{\eps^2}{v_H^2}, \frac{\nu}{8r}\right\} \qquad \text{and} \qquad \gamma \le \frac{1}{20r}.
\end{equation}

In order to show that the assumptions of Proposition~\ref{prop:sufficient-1-statement} are satisfied, we will define a natural map $\cM \colon \FreesdP \to \BPk$ by letting $\cM(G)$ be the subgraph of $G \setminus \Pi$ obtained by deleting from it all vertices that are non-adjacent to more than $\nu n$ vertices of a different colour class of $\Pi$; we shall show that this graph has maximum degree $k$. We will then estimate the left-hand side of~\eqref{eq:Freesdgp} using ad-hoc, combinatorial arguments.

Suppose that $\Pi = \partition$. For a graph $G \in \FreesdP$, let $X_G$ denote the set of all vertices of $G$ that have fewer than $|V_j| - \nu n$ neighbours in some colour class $V_j$ other than their own. More precisely,
\[
  X_G = \bigcup_{i=1}^r \big\{v \in V_i : \deg_G(v, V_j) < |V_j| - \nu n \text{ for some $j \neq i$}\big\}.
\]
We first show that the set $X_G$ is rather small and that the graph $(G \setminus \Pi) - X_G$ has maximum degree at most $k$.

\begin{lemma}
  \label{lemma:XG-small}
  For every $G \in \FreesdP$, we have
  \[
    |X_G| \le \frac{n}{4r}.
  \]
\end{lemma}
\begin{proof}
  Since
  \[
    e(\Pi) - e(G \cap \Pi) = \frac{1}{2} \cdot \sum_{i = 1}^r \sum_{v \in V_i} \sum_{j \neq i} \big(|V_j| - \deg(v, V_j)\big) \ge \frac{1}{2} \cdot |X_G| \cdot \nu n,
  \]
  we have
  \[
    e(\Pi) - \xi n^2 \le e(G) \le e(G \cap \Pi) + \delta n^2 \le e(\Pi) + \delta n^2 - \frac{1}{2} \cdot |X_G| \cdot \nu n.
  \]
  We conclude that
  \[
    |X_G| \le 2 \cdot \frac{\delta + \xi}{\nu} \cdot n \leByRef{eq:xi-delta-eps-nu-gamma} \frac{n}{4r},
  \]
  as claimed.
\end{proof}

\begin{lemma}
  \label{lemma:dense-star}
  For every $G \in \FreesdP$, the maximum degree of $(G \setminus \Pi) - X_G$ is at most $k$.
\end{lemma}
\begin{proof}
  Suppose that there were a $G \in \FreesdP$ such that $(G \setminus \Pi) - X_G$ has a vertex $v$ of degree at least $k+1$. Let $\Pi = \partition$ and suppose that $v \in V_i$. Let $u_1, \dotsc, u_{k+1} \in V_i \setminus X_G$ be arbitrary neighbours of $v$ and let $u_{k+2}, \dotsc, u_{v_H-1}$ be arbitrary vertices of $V_i \setminus X_G$ that are distinct from $v$ and $u_1, \dotsc, u_{k+1}$; such vertices exist since, by Lemma~\ref{lemma:XG-small}, we have $|V_i \setminus X_G| \ge n/(2r) - n/(4r) = n/(4r)$. For each $j \in \br{r} \setminus \{i\}$, let
  \[
    N_j = V_j \cap N_G(v) \cap \bigcap_{\ell=1}^{v_H-1} N_G(u_\ell).
  \]
  and observe that, by the definition of $X_G$,
  \[
    |N_j| \ge |V_j| - v_H \cdot \nu n \ge n/(2r) - v_H \cdot \nu n \geByRef{eq:eps-nu} \eps n.
  \]
  Observe further that the subgraph of $G \cap \Pi$ that is induced by $N_1 \cup \dotsb \cup N_{i-1} \cup N_{i+1} \cup \dotsb \cup N_r$ is $K_{r-1}(v_H)$-free; indeed, otherwise $G$ would contain every $(r+1)$-colourable vertex-critical graph of criticality $k+1$ with at most $v_H$ vertices, contradicting the fact that $G$ is $H$-free. This implies, in particular, that
  \[
    e(\Pi) - e(G \cap \Pi) \ge e\big(K_{r-1}(\eps n)\big) - \ex\big(K_{r-1}(\eps n), K_{r-1}(v_H)\big) \ge (\eps n)^2/v_H^2,
  \]
  where the last inequality follows from Lemma~\ref{lemma:Turan}. Consequently,
  \[
    m = e(G \cap \Pi) + e(G \setminus \Pi) \le e(\Pi) - (\eps n)^2/v_H^2 + \delta n^2 \lByRef{eq:xi-delta-eps-nu-gamma} e(\Pi) - \xi n^2,
  \]
  a contradiction.
\end{proof}

For every $G \in \FreesdP$, let $B_G = (G \setminus \Pi) - X_G$ and note, by Lemma~\ref{lemma:dense-star}, we have $B_G \in \BPk$. Define, for every $B \in \BPk$,
\[
  \FreeB = \{G \in \FreesdP : B_G = B\}.
\]
The following proposition, which is the main result of this section, implies that the map $\cM \colon G \mapsto B_G$ satisfies the assumptions of Proposition~\ref{prop:sufficient-1-statement}.

\begin{prop}
  \label{prop:dense-main}
  For every $B \in \BPk$, we have
  \[
    |\FreeB| \le \exp(-n) \cdot \binom{e(\Pi)}{m - e(B)}.
  \]
\end{prop}

The proof of Proposition~\ref{prop:dense-main} will require one additional lemma, which states that the maximum degree of the graph $G \setminus \Pi$ cannot be very large.

\begin{lemma}
  \label{lemma:densBoundDeg}
  For every $G \in \FreesdP$, we have $\Delta(G \setminus \Pi) < \eps n$.
\end{lemma}
\begin{proof}
  Suppose that there was a $G \in \FreesdP$ such that $\Delta(G \setminus \Pi) \ge \eps n$ and pick an arbitrary $v$ with $\deg_{G \setminus \Pi}(v) \ge \eps n$. Suppose that $\Pi = \partition$ and recall from~\eqref{eq:Pi-unfriendly} that $\deg_G(v, V_i) \ge \eps n$ for every $i\in\br{r}$. For each $i$, let $N_i \subseteq N(v) \cap V_i$ be an arbitrary subset of size exactly $\eps n$. Observe that the subgraph of $G \cap \Pi$ that is induced by $N_1 \cup \dotsb \cup N_r$ is $K_r(v_H)$-free; indeed, otherwise $G$ would contain every vertex-critical $(r+1)$-colourable graph with at most $v_H$ vertices, contradicting the fact that $G$ is $H$-free. This implies, in particular, that
  \[
    e(\Pi) - e(G \cap \Pi) \ge e\big(K_r(\eps n)\big) - \ex\big(K_r(\eps n), K_r(v_H)\big) \ge (\eps n)^2/v_H^2,
  \]
  where the last inequality follows from Lemma~\ref{lemma:Turan}. Consequently,
  \[
    m = e(G \cap \Pi) + e(G \setminus \Pi) \le e(\Pi) - (\eps n)^2/v_H^2 + \delta n^2 \lByRef{eq:xi-delta-eps-nu-gamma} e(\Pi) - \xi n^2,
  \]
  a contradiction.
\end{proof}

\begin{proof}[{Proof of Proposition~\ref{prop:dense-main}}]
  Fix an arbitrary $B \in \BPk$ and choose some $G \in \FreeB$. Note that $X_G$ cannot be empty as otherwise $G \subseteq \Pi \cup B$, contradicting the fact that $G \in \FreesdP$. Moreover, by Lemma~\ref{lemma:densBoundDeg}, every vertex of $X_G$ has degree at most $\eps n$ in $G \setminus \Pi$. Denote $\Pi_{X_G} = \Pi \setminus (\Pi - X_G)$; in other words, $\Pi_{X_G}$ comprises all edges of $\Pi$ that have an endpoint in $X_G$. By the definition of $X_G$,
  \[
    e(\Pi_{X_G}) - e(G \cap \Pi_{X_G}) \ge \frac{1}{2} \cdot |X_G| \cdot \nu n.
  \]
  Observe that every graph in $G \in \FreeB$ may be constructed as follows:
  \begin{itemize}
  \item
    Choose a nonempty vertex set $X$ with at most $n/(4r)$ elements (to serve as $X_G$).
  \item
    Choose at most $|X| \cdot \eps n$ edges of $\Pic$, each touching $X$, to form $(G \setminus \Pi) \setminus B$.
  \item
    Choose at most $e(\Pi_X) - |X| \cdot \nu n / 2$ edges of $\Pi_X$ to form $G \cap \Pi_X$.
  \item
    Choose the remaining edges of $G$ from $\Pi - X$.
  \end{itemize}
  In particular, letting
  \[
    t_X = |X| \cdot \eps n \qquad \text{and} \qquad z_X = e(\Pi_X) - |X| \cdot \nu n /2,
  \]
  we have
  \[
    |\FreeB| \le \sum_{\substack{X \neq \emptyset \\ |X| \le n/(4r)}} \sum_{t\le t_X} \sum_{z \le z_X} \binom{|X| \cdot n}{t}  \binom{e(\Pi_X)}{z} \binom{e(\Pi - X)}{m - z - t - e(B)}.
  \]
  Note that, for all $X$ and all $t \le t_X$ and $z \le z_X$,
  \[
    \begin{split}
      \binom{e(\Pi_X)}{z} \cdot \binom{e(\Pi_X)}{z+t}^{-1}
      & \leByRef{eq:binCoef1} \left(\frac{z_X+t_X}{e(\Pi_X) - z_X - t_X}\right)^t \le \left(\frac{e(\Pi_X) + |X| \cdot (\eps - \nu/2) n}{|X| \cdot (\nu/2-\eps) n}\right)^t \\
      & \le \left(\frac{1+\eps-\nu/2}{\nu/2-\eps}\right)^t \leByRef{eq:eps-nu} \left(\frac{4}{\nu}\right)^t \le \left(\frac{4}{\nu}\right)^{t_X},
    \end{split}
  \]
  so that
  \[
    \begin{split}
      \binom{|X| \cdot n}{t} \binom{e(\Pi_X)}{z} \cdot \binom{e(\Pi)_X}{z+t}^{-1} & \le \binom{|X| \cdot n}{t_X} \cdot \left(\frac{4}{\nu}\right)^{t_X} \\
      & \leByRef{eq:binCoef4} \left(\frac{4e \cdot |X| \cdot n}{\nu t_X}\right)^{t_X} = \left(\frac{4e}{\nu \eps}\right)^{|X| \cdot \eps n}.
    \end{split}
  \]
  This gives
  \begin{equation}
    \label{eq:FreeB-second-bound}
    |\FreeB| \le \sum_{\substack{X \neq \emptyset \\ |X| \le n/(4r)}} \left(\frac{4e}{\nu \eps}\right)^{|X| \cdot \eps n} \sum_{z \le z_X} \sum_{t\le t_X} \underbrace{\binom{e(\Pi_X)}{z+t}\binom{e(\Pi - X)}{m - z - t - e(B)}}_{N_{X,z+t}}.
  \end{equation}
  
  By Vandermonde's identity, we have, for every $y$,
  \begin{equation}
    \label{eq:Nz-Vandermonde}
    N_{X,y} \le \binom{e(\Pi_X) + e(\Pi - X)}{m-e(B)} = \binom{e(\Pi)}{m-e(B)}.
  \end{equation}
  Moreover, direct calculation shows that
  \begin{equation}
    \label{eq:Nz-ratio}
    \frac{N_{X,y}}{N_{X,y+1}} = \underbrace{\frac{y+1}{e(\Pi_X)-y}}_{\rho_{X,y}} \cdot \underbrace{\frac{e(\Pi-X)-m+y+1+e(B)}{m-y-e(B)}}_{\rho_{X,y}'}.
  \end{equation}
  Set $y_X = z_X + t_X$ and
  \[
    y_X' = y_X + |X| \cdot \nu n /4 = e(\Pi_X) - |X| \cdot (\nu/4-\eps) n \leByRef{eq:eps-nu} e(\Pi_X) - |X| \cdot \nu n /8.
  \]
  Assume that $|X| \le n/(4r)$ and $y+1 \le y_X'$. Using $e(\Pi_X) \le |X| \cdot n$, we have $\rho_y \le 8/\nu$. Moreover,
  \[
    m-y-1-e(B) \ge e(\Pi) - \xi n^2 - e(\Pi_X) - kn \ge e(\Pi - X) - 2\xi n^2
  \]
  and, by part~\ref{item:Partg-lower} of Proposition~\ref{prop:bal-unbal-part-size} and our assumption that $\gamma \le \frac{1}{20r}$,  
  \[
    m-y-e(B) \ge e(\Pi) - 2\xi n^2 - |X|\cdot n \ge \frac{n^2}{5} -2\xi n^2 - \frac{n^2}{4r} \ge \frac{n^2}{20}.
  \]
  Consequently, $\rho_{X,y}' \le 40\xi$. Substituting these two estimates into~\eqref{eq:Nz-ratio} yields
  \begin{equation}
    \label{eq:Nz-ratio-bound}
    \frac{N_{X,y}}{N_{X,y+1}} \le \frac{320\xi}{\nu} \leByRef{eq:xi-dense} 1.
  \end{equation}
  We may conclude that, when $|X| \le n/(4r)$ and $y \le y_X$,
  \[
    \begin{split}
      N_{X,y}
      & = N_{X,y_X'} \cdot \prod_{y' = y}^{y_X'-1} \frac{N_{X,y'}}{N_{X,y'+1}} \leBy{\eqref{eq:Nz-Vandermonde}, \eqref{eq:Nz-ratio-bound}} \binom{e(\Pi)}{m-e(B)} \cdot \left(\frac{320 \xi}{\nu}\right)^{y_X'-y} \\
      & \leByRef{eq:Nz-ratio-bound} \binom{e(\Pi)}{m-e(B)} \cdot \left(\frac{320\xi}{\nu}\right)^{y_X'-y_X} = \binom{e(\Pi)}{m-e(B)} \cdot \left(\frac{320\xi}{\nu}\right)^{|X| \cdot \nu n/4}.
    \end{split}
  \]
  Finally, substituting this estimate into~\eqref{eq:FreeB-second-bound} yields
  \[
    \begin{split}
      |\FreeB| \cdot \binom{e(\Pi)}{m-e(B)}^{-1} & \le \sum_{\substack{X \neq \emptyset \\ |X| \le n/(4r)}} \left(\frac{4e}{\nu \eps}\right)^{|X| \cdot \eps n} (z_X+1)(t_X+1) \cdot \left(\frac{320\xi}{\nu}\right)^{|X| \cdot \nu n/4} \\
      & \le \sum_{x=1}^{n/(4r)} \binom{n}{x} \cdot (xn)^2 \cdot \left[\left(\frac{4e}{\nu \eps}\right)^\eps \cdot \left(\frac{320\xi}{\nu}\right)^{\nu/4}\right]^{xn} \\
      & \leByRef{eq:xi-dense} \sum_{x=1}^n n^{5x} \cdot e^{-2xn} \le e^{-n},
    \end{split}
  \]
  provided that $n$ is sufficiently large.
\end{proof}

\bibliographystyle{amsplain}
\bibliography{vx-critical}

\providecommand{\bysame}{\leavevmode\hbox to3em{\hrulefill}\thinspace}
\providecommand{\MR}{\relax\ifhmode\unskip\space\fi MR }
\providecommand{\MRhref}[2]{%
  \href{http://www.ams.org/mathscinet-getitem?mr=#1}{#2}
}
\providecommand{\href}[2]{#2}
\begin{thebibliography}{10}

\bibitem{BaBoSi09}
J.~Balogh, B.~Bollob{\'a}s, and M.~Simonovits, \emph{The typical structure of
  graphs without given excluded subgraphs}, Random Structures Algorithms
  \textbf{34} (2009), 305--318.

\bibitem{BalBusColLiuMorSha17}
J.~Balogh, N.~Bushaw, M.~Collares, H.~Liu, R.~Morris, and M.~Sharifzadeh,
  \emph{The typical structure of graphs with no large cliques}, Combinatorica
  \textbf{37} (2017), no.~4, 617--632.

\bibitem{BaMoSa15}
J.~Balogh, R.~Morris, and W.~Samotij, \emph{Independent sets in hypergraphs},
  J. Amer. Math. Soc. \textbf{28} (2015), 669--709.

\bibitem{BaMoSaWa16}
J.~Balogh, R.~Morris, W.~Samotij, and L.~Warnke, \emph{The typical structure of
  sparse {$K_{r+1}$}-free graphs}, Trans. Amer. Math. Soc. \textbf{368} (2016),
  no.~9, 6439--6485.

\bibitem{BalSam19}
J.~Balogh and W.~Samotij, \emph{An efficient container lemma}, Discrete Anal.
  (2020), Paper No. 17, 56.

\bibitem{Bol80}
B.~Bollob\'{a}s, \emph{A probabilistic proof of an asymptotic formula for the
  number of labelled regular graphs}, European J. Combin. \textbf{1} (1980),
  no.~4, 311--316.

\bibitem{Eng}
O.~Engelberg, \emph{On the typical structure of graphs not containing a fixed
  vertex-critical subgraph}, Master's thesis, Tel Aviv University, 2017.

\bibitem{ErKlRo76}
P.~Erd{\H{o}}s, D.~J. Kleitman, and B.~L. Rothschild, \emph{Asymptotic
  enumeration of {$K_n$}-free graphs}, Colloquio {I}nternazionale sulle
  {T}eorie {C}ombinatorie ({R}ome, 1973), {T}omo {II}, Accad. Naz. Lincei,
  Rome, 1976, pp.~19--27. Atti dei Convegni Lincei, No. 17.

\bibitem{ErdSim66}
P.~Erd{\H{o}}s and M.~Simonovits, \emph{A limit theorem in graph theory},
  Studia Sci. Math. Hungar. \textbf{1} (1966), 51--57.

\bibitem{ErSt46}
P.~Erd{\H{o}}s and A.~H. Stone, \emph{On the structure of linear graphs}, Bull.
  Amer. Math. Soc. \textbf{52} (1946), 1087--1091.

\bibitem{Har60}
T.~E. Harris, \emph{A lower bound for the critical probability in a certain
  percolation process}, Proc. Cambridge Philos. Soc. \textbf{56} (1960),
  13--20.

\bibitem{HuPrSt93}
C.~Hundack, H.~J. Pr\"omel, and A.~Steger, \emph{Extremal graph problems for
  graphs with a color-critical vertex}, Combin. Probab. Comput. \textbf{2}
  (1993), 465--477.

\bibitem{Jan90}
S.~Janson, \emph{Poisson approximation for large deviations}, Random Structures
  Algorithms \textbf{1} (1990), 221--229.

\bibitem{JaLuRu00}
S.~Janson, T.~{\L}uczak, and A.~Ruci{\'n}ski, \emph{Random graphs},
  Wiley-Interscience Series in Discrete Mathematics and Optimization,
  Wiley-Interscience, New York, 2000.

\bibitem{KoPrRo87}
P.~Kolaitis, H.~J. Pr{\"o}mel, and B.~Rothschild, \emph{{$K_{l+1}$}-free
  graphs: asymptotic structure and a $0-1$ law}, Trans. Amer. Math. Soc.
  \textbf{303} (1987), 637--671.

\bibitem{MouNenSte14}
F.~Mousset, R.~Nenadov, and A.~Steger, \emph{On the number of graphs without
  large cliques}, SIAM J. Discrete Math. \textbf{28} (2014), 1980--1986.

\bibitem{OsPrTa03}
D.~Osthus, H.~J. Pr\"omel, and A.~Taraz, \emph{For which densities are random
  triangle-free graphs almost surely bipartite?}, Combinatorica \textbf{23}
  (2003), no.~1, 105--150, Paul Erd\H os and his mathematics (Budapest, 1999).

\bibitem{PrSt92}
H.~J. Pr{\"o}mel and A.~Steger, \emph{The asymptotic number of graphs not
  containing a fixed color-critical subgraph}, Combinatorica \textbf{12}
  (1992), 463--473.

\bibitem{PrSt95}
\bysame, \emph{Random {$l$}-colorable graphs}, Random Structures Algorithms
  \textbf{6} (1995), 21--37.

\bibitem{PrSt96}
\bysame, \emph{On the asymptotic structure of sparse triangle free graphs}, J.
  Graph Theory \textbf{21} (1996), 137--151.

\bibitem{Si74}
M.~Simonovits, \emph{Extermal graph problems with symmetrical extremal graphs.
  {A}dditional chromatic conditions}, Discrete Math. \textbf{7} (1974),
  349--376.

\bibitem{Tu41}
P.~Tur{\'a}n, \emph{Eine {E}xtremalaufgabe aus der {G}raphentheorie}, Mat. Fiz.
  Lapok \textbf{48} (1941), 436--452.

\bibitem{Wor81}
N.~C. Wormald, \emph{The asymptotic distribution of short cycles in random
  regular graphs}, J. Combin. Theory Ser. B \textbf{31} (1981), 168--182.

\end{thebibliography}

\end{document}